\documentclass{elsarticle}
\usepackage{theorem,latexsym,amsmath,amsfonts}
\usepackage{amssymb,amsbsy}
\usepackage{epsfig}
\usepackage{color}
\usepackage[all]{xy}
\usepackage{tikz}
\input xyv2

\theoremstyle{change}
\theorembodyfont{\rm}
\newtheorem{prop}{Proposition:}[section]
\newtheorem{theo}[prop]{Theorem:}
\newtheorem{lem}[prop]{Lemma:}
\newtheorem{defi}[prop]{Definition:}

\newtheorem{const}[prop]{Constructions:}
\newtheorem{coro}[prop]{Corollary:}

\newtheorem{conv}[prop]{Convention:}
\newtheorem{exam}[prop]{Example:}
\newtheorem{leer}[prop]{}
\newtheorem{rema}[prop]{Remark:}
\newtheorem{nota}[prop]{Notation:}

\newenvironment{proof}{\par\noindent\textbf{Proof}}{\hfill\ensuremath{\Box}}

\def\cC{\mathcal{C}}
\def\Top{\mathcal{T}\!op}
\def\cM{\mathcal{M}}
\def\cR{\mathcal{R}}
\def\cN{\mathcal{N}}
\def\cK{\mathcal{K}}
\def\cV{\mathcal{V}}

\def\cE{\mathcal{E}}
\def\cS{\mathcal{S}}

\def\cP{\mathcal{P}}
\def\cQ{\mathcal{Q}}

\def\cL{\mathcal{L}}
\def\cB{\mathcal{B}}
\def\cA{\mathcal{A}}
\def\cD{\mathcal{D}}
\def\cC{\mathcal{C}}
\def\cE{\mathcal{E}}
\def\cX{\mathcal{X}}

\def\mC{\mathbb{C}}
\def\mM{\mathbb{M}}

\def\mE{\mathbb{E}}

\def\mS{\mathbb{S}}
\def\colim{\textrm{colim}}
\def\lim{\textrm{lim}}

\def\id{{\rm id}}
\def\Sing{\textrm{Sing}}

\def\hocolim{\textrm{hocolim}}
\def\colim{\textrm{colim}}
\def\Id{{\textrm{Id}}}
\def\Func{\mathcal{F}\textrm{unc}}
\def\op{{\textrm op}}
\def\Cat{\mathcal{C}\!at}
\def\Lax{{\textrm Lax}}
\def\hocolim{\textrm{hocolim}}
\def\ev{\textrm{ev}}

\def\SSets{\mathcal{SS}ets}

\def\uM{\underline{\mathcal{M}}}

\def\uV{\underline{\mathcal{V}}}
\def\uC{\underline{\mathcal{C}}}
\def\ob{\textrm{ob}}
\def\cat{\textrm{cat}}
\def\const{\textrm{const}}
\def\mor{\textrm{mor}}
\def\we{\textrm{we}}
\def\produ{\textrm{prod}}
\def\alg{\textrm{alg}}
\def\diag{\textrm{diag}}
\def\str{\textrm{str}}
\def\orb{\textrm{orb}}
\def\Reedy{\textrm{Reedy}}

\newcommand{\Sets}{\mathcal{S}ets}

\newcommand{\scatm}{\mathcal{SC}\!at^{\mathbb{M}}}
\newcommand{\catm}{\mathcal{C}\!at^{\mathbb{M}}}

\newcommand{\stopbm}{\mathcal{ST}\!op^{B\mathbb{M}}}

\newcommand{\Mon}{\mathcal{M}on}

\begin{document}
\title{\textbf{Homotopy Colimits of Algebras Over $\Cat$-Operads and Iterated Loop Spaces}}
\author{Z. Fiedorowicz}
\address{Department of Mathematics, The Ohio State University\\ Columbus, OH 43210-1174, USA}
\ead{fiedorow@math.osu.edu}

\author{M. Stelzer}
\address{Universit\"at Osnabr\"uck, Fachbereich Mathematik/Informatik\\ Albrechtstr. 28a, 49069 Osnabr\"uck, Germany}
\ead{mstelzer@uos.de}

\author{R.M.~Vogt}
\address{Universit\"at Osnabr\"uck, Fachbereich Mathematik/Informatik\\ Albrechtstr. 28a, 49069 Osnabr\"uck, Germany}
\ead{rvogt@uos.de}

\begin{abstract}
We extend Thomason's homotopy colimit construction in the category
of permutative categories to categories of algebras over an arbitrary
$\Cat$ operad and analyze its properties. We then use this homotopy colimit to 
prove that the classifying space functor induces an equivalence
between the category of $n$-fold monoidal categories and the category
of $\cC_n$-spaces after formally inverting certain classes of weak equivalences,
where $\cC_n$ is the little $n$-cubes operad. As a consequence we obtain
an equivalence of the categories of $n$-fold monoidal categories and the
category of $n$-fold loop spaces and loop maps after localization with 
respect to some other class of weak equivalences. We recover
Thomason's corresponding result about infinite loop spaces and
obtain related results about braided monoidal categories and $2$-fold
loop spaces.
\end{abstract}

\maketitle

\section{Introduction}
In the last two decades there has been an increasing interest in algebraic
structures on a category such as a monoidal structure, a symmetric
monoidal one, a braiding or a ribbon structure, motivated by questions
arising from knot theory or mathematical physics. This paper deals with
structures encoded by an operad.

Our motivation is the problem to determine those structures on a category
which correspond to $n$-fold loop spaces: it has been known for quite some
time that the classifying space of a monoidal category is an
$A_\infty$-monoid, whose group completion is weakly equivalent to a loop
space. In the same way, a symmetric monoidal category gives rise to an
infinite loop space \cite{May1}, and a braided monoidal category to a
double loop space \cite{Fied}. 
In \cite{BFSV} we introduced the notion of an $n$-fold
monoidal category. The structure of such a category is encoded by a
$\Sigma$-free operad $\mathcal{M}_n$ in $\mathcal{C}at$, the category of
small categories. The classifying space functor turns $\mathcal{M}_n$ into a
$\Sigma$-free topological operad $B\mathcal{M}_n$, and  we could show that
there is a sequence of weak equivalences of operads between
$B\mathcal{M}_n$ and the little $n$-cubes operad $\mathcal{C}_n$. Hence
each $n$-fold monoidal category gives rise to an $n$-fold loop space.

Now one would like to know
 whether each $n$-fold loop space can be obtained in this way
up to weak equivalence. There has been evidence for this: it has been known
for about forty years that each topological space is weakly equivalent to the
classifying space of a category, and in 1995 Thomason proved that each
infinite loop space comes from a permutative category and hence from a symmetric 
monoidal category, because a symmetric monoidal category is equivalent to 
a permutative category  \cite{Thom2}.
 The result is obtained in two major steps. Using
categorical coherence theory extensively, Thomason first shows that each
infinite loop space arises from a simplicial symmetric monoidal category.
He then applies the homotopy colimit construction in the category of
symmetric monoidal categories to get rid of the simplicial parameter. The
construction and analysis of this homotopy colimit is an essential part of
\cite{Thom1}. 

Refining Thomason's argument and combining categorical with homotopical
coherence theory we were able to achieve the first step for $n$-fold monoidal
categories: in \cite{FV} we proved that each $n$-fold loop space comes from
a simplicial $n$-fold monoidal category and it remains to bridge the gap from
simplicial $n$-fold monoidal categories to $n$-fold monoidal categories.

For this we extend Thomason's homotopy colimit construction for symmetric
monoidal categories to algebras over an arbitrary $\Cat$-operad. Thomason
constructed a Quillen model category structure on $\Cat$ whose weak equvalences
are those functors $F:\cC\to \cD$ for which $BF:B\cC \to B\cD$ is a weak
homotopy equivalence \cite{Thom3}. Unfortunately this model structure
is not monoidal so that
the rich literature on algebras in model categories cannot be applied to
$\Cat$, and, as a consequence, the results on homotopy colimits in model
categories are not available either. It is our intention to close this gap.

We define and investigate homotopy colimits $\hocolim^{\cM}X$ of diagrams $X$
of algebras over an arbitrary $\Sigma$-free
$\Cat$-operad $\cM$ and prove that 
there is a natural weak equivalence
$$\hocolim^{B\cM} BX\to B(\hocolim^{\cM}X)$$
(Theorem \ref{6_9}) if $\cM$ satisfies a factorization condition (Definition \ref{6_7}) . 
In the process we correct a minor
flaw in Thomason's argument (see Remark \ref{8_24}). Although we only use
the homotopy colimit construction to prove results about iterated loop spaces,
we believe that it is of separate interest. In particular, it sheds some light
into Thomason's original construction, where an operad cannot be seen explicitly.

We use this result to prove that for a $\Sigma$-free operad, which satisfies the
factorization condition, the classifying space functor
induces an equivalence of categories
$$
B:\mathcal{C}at^{\mathcal{M}}[{\we}^{-1}]\to 
\Top^{B\mathcal{M}}[{\we}^{-1}]
$$
where $\mathcal{C}at^{\mathcal{M}}$ and $\Top^{B\mathcal{M}}$ are the categories of
$\cM$-algebras in $\Cat$, respectively $B\cM$-algebras in $\Top$, and the weak
equivalences are those morphisms of algebras whose underlying morphisms are
weak equivalences in $\Cat$, respectively weak homotopy equivalences in $\Top$.
Since we do not have a model type structure on $\Cat^{\cM}$ it is not clear that
the localization $\mathcal{C}at^{\mathcal{M}}[{\we}^{-1}]$ exists. To make the above
equivalence hold we can use Grothendieck's language of universes, where
$\mathcal{C}at^{\mathcal{M}}[{\we}^{-1}]$ exists in some higher universe. 
In the foundational setting of G\"odel-Bernay we can show that the localization up to
equivalence exists (Definition \ref{7_4}) which we will denote by 
$\mathcal{C}at^{\mathcal{M}}[\widetilde{{\we}^{-1}}]$.
 
The operads $\cM_n$ and the operads $\cB r$ and $\widetilde{\Sigma}$ codifying the structures of braided monoidal
respectively  permutative categories satisfy the factorization condition. Combining the equivalence above
with the equivalence induced by the change of operad
functors, we obtain equivalences of categories
$$\begin{array}{rcccl}
\mathcal{C}at^{\mathcal{M}_n}[{\we}^{-1}] &\simeq &
\mathcal{T}op^{B\mathcal{M}_n}[{\we}^{-1}] &\simeq &
\mathcal{T}op^{\mathcal{C}_n}[{\we}^{-1}]\\
\mathcal{C}at^{\mathcal{B}r}[{\we}^{-1}] &\simeq &
\mathcal{T}op^{B\mathcal{B}r}[{\we}^{-1}] &\simeq &
\mathcal{T}op^{\mathcal{C}_2}[{\we}^{-1}]\\
\mathcal{C}at^{\widetilde{\Sigma}}[{\we}^{-1}] &\simeq &
\mathcal{T}op^{B\widetilde{\Sigma}}[{\we}^{-1}] &\simeq &
\mathcal{T}op^{\mathcal{C}_\infty}[{\we}^{-1}]
\end{array}
$$
In particular, each $n$-fold loop space arises up to weak equivalence and group completion from
an $n$-fold monoidal category, a double loop space from a braided monoidal category, an an infinte loop
space from a permutative category. So we obtain Thomason's
result before group completion. 

The paper is organized as follows: In the following two sections we introduce
the category theoretical tools we need for this paper. Sections 4 and 5 contain the construction
of the homotopy colimit and compare it with known constructions in model categories. In Section 6 we
determine the homotopy type of our homotopy colimit in the case of a $\Sigma$-free $\Cat$-operad which
satisfies the factorization condition. The proof of a
crucial technical lemma is deferred to Appendix B.
In Sections 7 and 8 we establish the equivalences of categories listed and explicitly show their relations to
iterated loop spaces by applying suitable group completion constructions. We have added three appendices: Appendix A
gives a detailed description of the input category of our homotopy colimit construction, Appendix B contains the proof 
mentioned above, and Appendix C the proof of the existence of various localizations up to equivalence.

\textit{Acknowledgement:} We are indepted to the referee for his diligent work. His detailed questions made us 
 realize that we had been
missing a relation in the construction of the homotopy colimit and helped us 
clarify a number of points. He also noted that our original proof of Lemma \ref{8_12} was incorrect.
 We want to thank Mirjam Solberg who found the same mistake in the proof of
Lemma \ref{8_12} and indicated how it can be fixed.
We also want to thank the Deutsche Forschungsgemeinschaft for 
support.

\section{Categorical prerequisites}
Throughout this paper we will work with categories enriched over a symmetric monoidal category
$\mathcal{V}$. We will start with a general $\mathcal{V}$, but from Definition \ref{2_1} onwards $\mathcal{V}$
will always be one of the categories  $\Cat$ of small categories, $\Sets$ of sets, $\SSets$ of simplical sets, and $\Top$ of 
$k$-spaces in the sense of \cite[5(ii)]{Vogt}. For basic concepts we refer to the books of Kelly \cite{Kelly} and Borceux \cite{Bor2}.

Recall that $(\mathcal{V},\square,\Phi)$ is a \textit{symmetric monoidal category} if we are given a bifunctor 
$\square:\mathcal{V}\times\mathcal{V}\to\mathcal{V}$ which is associative and commutative up to coherent natural isomorphisms and has a 
coherent unit object $\Phi$ \cite[6.1.2]{Bor2}. If the associativity and unit isomorphisms are the identities we call $(\mathcal{V},\square,\Phi)$ 
 a \textit{permutative category}. $\mathcal{V}$ is \textit{closed} if it has an internal Hom-functor, i.e. a bifunctor 
$\uV (-,-):\mathcal{V}^{\op}\times \mathcal{V} \to\mathcal{V}$ together with a natural isomorphism
$$
\mathcal{V}(X\square Y, Z)\cong\mathcal{V}(X,\uV (Y,Z)).
$$
A functor $F:\mathcal{V}\to\mathcal{V}'$ between symmetric monoidal categories
 is called \textit{symmetric monoidal} if there are natural transformations 
$$\Phi_{\mathcal{V}'}\to F(\Phi_{\mathcal{V}})\quad \textrm{ and }\quad F(-)\square_{\mathcal{V}'}
F(-)\to F(-\square_{\mathcal{V}}-)$$
 compatible with the coherence isomorphisms. If these maps are isomorphisms $F$ is called \textit{strong
symmetric monoidal}.

A \textit{$\cV$-category} or \textit{$\cV$-enriched category} is a category $\mathcal{C}$ equipped with an enrichment functor
$$
\uC (-,-):\mathcal{C}^{\op}\times \cC\to \cV
$$
and composition morphisms
$$
\uC (B,C)\ \square \ \uC (A,B)\to \uC (A,C)
$$
in $\mathcal{V}$ with identities $j_A:\Phi\to\uC (A,A)$, subject to the obvious conditions. $\mathcal{V}$-functors and $\mathcal{V}$-natural 
transformations are defined in the obvious way. 

A $\mathcal{V}$-category $\mathcal{C}$ is \textit{tensored over $\mathcal{V}$} if there is a functor
$$
\mathcal{V}\times\mathcal{C}\to\mathcal{C}, \quad (K,X)\mapsto K\otimes X
$$
such that 
\begin{enumerate}
 \item there are natural isomorphisms
$$ K\otimes (L\otimes X)\cong (K\square L)\otimes X \quad \textrm{and}\quad \Phi \otimes X\cong X,$$
satisfying the standard associativity and unit coherence conditions,
\item there are natural isomorphisms
$$ 
\cC(K\otimes X,Y)\cong \cV(K,\uC (X,Y)).
$$
\end{enumerate}
Setting $K=\Phi$, we obtain $\cC(X,Y)\cong \cV(\Phi,\uC (X,Y))$, thus recovering the underlying category $\cC$ from its
enriched version. It is easy to show that condition (2) implies an enriched version: there are natural isomorphisms
$$ \uC(K\otimes X,Y)\cong \uV(K,\uC (X,Y))$$
in $\mathcal{V}$.

Dually, a $\mathcal{V}$-category $\mathcal{C}$ is \textit{cotensored over $\mathcal{V}$} if there is a functor
$$
\mathcal{V}^{\op}\times\mathcal{C}\to\mathcal{C}, \quad (K,X)\mapsto X^K
$$
such that
\begin{enumerate}
 \item there are natural isomorphisms
$$ X^{K\square L} \cong (X^L)^K\quad \textrm{and}\quad X^{\Phi}\cong X,$$
satisfying the standard associativity and unit coherence conditions,
\item there are natural isomorphisms
$$ 
\cC(X,Y^K)\cong \cV(K,\uC (X,Y)).
$$
\end{enumerate}
Again, the last condition implies an enriched version: there are natural isomorphisms
$$\uC(X,Y^K)\cong \uV(K,\uC (X,Y))$$ in $\cV$.

\begin{leer}\label{2_0}
Let $\mathcal{L}$ be a small category, $\mathcal{C}$ a $\mathcal{V}$-category tensored over $\mathcal{V}$ and let 
$$
D:\mathcal{L}^\op\to\mathcal{V}\quad \textrm{ and } \quad E:\mathcal{L}\to\mathcal{C}
$$
be functors. We define $D\otimes_\mathcal{L} E$ be the coend (see \cite[IX.6]{MacLane}) of the functor 
$$
\mathcal{L}^\op\times\mathcal{L}\stackrel{D\times E}{\longrightarrow}\mathcal{V}\times\mathcal{C}
\stackrel{-\otimes-}{\longrightarrow}\mathcal{C}
$$
\end{leer}

A tensor is a special case of an indexed colimit and a cotensor a special case of an indexed limit (for a definition see \cite[pp. 71-73]{Kelly}). 
We call a $\mathcal{V}$-category \textit{$\mathcal{V}$-complete} if it has all small indexed limits and \textit{$\mathcal{V}$-cocomplete} if it has 
all small indexed colimits. Recall that $\mathcal{C}$ is $\mathcal{V}$-complete if it is complete in the ordinary sense and cotensored, and
$\mathcal{V}$-cocomplete if it is cocomplete in the ordinary sense and tensored \cite[Thm. 3.7.3]{Kelly}.

As mentioned in the beginning of this section, henceforth  $\mathcal{V}$ will be one of the categories $\Cat$, $\Sets$, $\SSets$, or $\Top$.
The symmetric monoidal structure of each of these categories is given by the product functor. The categories are closed, 
their internal Hom-functors are the functor categories, the sets of all maps, the simplicial Hom-functor, and the $k$-function spaces respectively. 
They are tensored and cotensored over themselves by the product and the internal Hom-functors.

Observe that $\Sets$ can be considered as a symmetric monoidal subcategory of each of the others if we interpret a set as a discrete category, 
discrete simplicial set,
 or discrete space respectively. Hence the following definition makes sense, where $\Sigma_m$ is considered a set with $\Sigma_m$-action.

\begin{defi}\label{2_1}
Let $\mathcal{V}$ be one of the categories $\Cat$, $\Sets$, $\SSets$, or $\Top$. A \textit{$\cV$-operad} is a permutative $\mathcal{V}$-enriched
category $(\mathcal{M},\oplus,0)$ such that $\ob \mathcal{M}=\mathbb{N},\ m\oplus n=m+n$, and the maps
$$
\coprod\limits_{r_1+\ldots+r_n=m} \uM (r_1,1)\times\ldots\times\uM (r_n,1) \times_{\Sigma_{r_1}\times\ldots\times\Sigma_{r_n}}\Sigma_m\to\uM(m,n)
$$
sending $(f_1,\ldots,f_n;\sigma)$ to $(f_1\oplus\ldots\oplus f_n)\circ \sigma$ are isomorphisms in $\mathcal{V}$ (note
 that $\Sigma_n$ 
operates on $n=1\oplus\ldots\oplus1$).

An \textit{$\cM$-algebra} in $\mathcal{V}$ is a strong symmetric monoidal functor $A:\mathcal{M}\to\mathcal{V}$.
 A \textit{map of $\mathcal{M}$-algebras} 
is a natural transformation $\alpha:A\to B$ compatible with the symmetric monoidal structure; in particular, 
it is determined by $\alpha(1):A(1)\to B(1)$.
\end{defi}

\textbf{Notation:} The morphism objects $\uM (m,n)$ of a $\mathcal{V}$-operad $\mathcal{M}$ are uniquely determined by the objects $\uM(r,1)$. 
As is common usage we denote $\uM (r,1)$ by $\uM (r)$. We call $\cM$ $\Sigma$-\textit{free} if for each $n$ the right action of $\Sigma_r$ on 
$\uM (r)$ is free. For an operad in $\Cat$ this means that $\Sigma_r$ acts freely on objects and morphisms of $\uM (r)$.

Let $\mathcal{V}^{\mathcal{M}}$ be the category of $\mathcal{M}$-algebras in $\mathcal{V}$.
\vspace{1ex}

Recall that a \textit{monad} on a category $\mathcal{C}$ is a functor
$$
\mathbb{M}:\mathcal{C}\to\mathcal{C}
$$
together with natural transformations
$$
\mu:\mathbb{M}\circ\mathbb{M}\to\mathbb{M}\quad \textrm{ and } \quad \iota:\Id_\mathcal{C}\to\mathbb{M}
$$
satisfying the associativity condition $\mu\circ\mu\mathbb{M}=\mu\circ\mathbb{M}\mu$ and the unit conditions 
$\mu\circ\iota\mathbb{M}=\mu\circ\mathbb{M}\iota$.

An \textit{algebra over} $\mathbb{M}$ or $\mathbb{M}$-\textit{algebra} in $\mathcal{C}$ is a pair $(A,\xi)$ consisting of an object 
$A$ of $\mathcal{C}$ and a map $\xi:\mathbb{M}A\to A$ satisfying
$$
\xi\circ\iota(A)=\id_A \quad \textrm{ and } \quad \xi\circ\mu(A)=\xi\circ\mathbb{M}\xi.
$$
A \textit{map of $\mathbb{M}$-algebras} $f:(A,\xi_A)\to(B,\xi_B)$ is a map $f:A\to B$ in $\mathcal{C}$ such that
$$
\xi_B\circ\mathbb{M}f=f\circ \xi_A.
$$
We denote the category of $\mathbb{M}$-algebras in $\mathcal{C}$ by $\mathcal{C}^\mathbb{M}$.

\begin{leer}\label{2_2a}
For a $\cV$-operad $\cM$ we have an adjoint pair of $\mathcal{V}$-functors
$$
F:\mathcal{V}\leftrightarrows\mathcal{V}^\mathcal{M}:U
$$
consisting of the free algebra functor $F$ and the underlying object functor $U$ given by $U(A)=A(1)$. The former is defined by 
$$
F(X)=\coprod\limits_{n\ge 0}\mathcal{M}(n)\times_{\Sigma_n}X^n.
$$
Let $\iota:\Id\to UF$ and $\varepsilon:FU\to\Id$ be the unit and conunit of this adjunction. Then $(\mathbb{M},\mu,\iota)=(UF,
U\varepsilon F,\iota)$ defines a monad on $\mathcal{V}$. Explicitly, the natural transformation $\iota$ is given by 
$$
\iota(X): X\cong \{ \id \}\times X\subset \cM (1)\times X\subset \coprod\limits_{n\ge 0}\mathcal{M}(n)\times_{\Sigma_n}X^n=\mathbb{M} X.
$$
If $f:X\to Y$ is a morphism to an $\cM$-algebra in $\cV$, there is a unique map of $\cM$-algebras $\bar{f}: \mathbb{M}(X)\to Y$,
the \textit{adjoint of} $f$, such that 
$\bar{f}\circ \iota(X)=f$. We also note that $\mu : \mathbb{M}^2(X)\to \mathbb{M}(X)$ is the adjoint of
 $\id_{\mathbb{M}(X)}$.
\end{leer}

\begin{prop}\label{2_2}
Let $\mathcal{M}$ be a $\mathcal{V}$-operad. Then $\mathcal{V}^\mathcal{M}$ and $\mathcal{V}^\mathbb{M}$ are isomorphic.
\end{prop}

\begin{proof}
If $A$ is an $\mathcal{M}$-algebra in $\mathcal{V}$ the adjoint of the structure maps 
$$
A:\mathcal{M}(n)\to \mathcal{V}(A(1)^n, A(1)) \eqno(\ast)
$$
factor as
$$
\mathcal{M}(n)\times A(1)^n\stackrel{\textrm{proj.}}{\longrightarrow} \mathcal{M}(n)\times_{\Sigma_n} A(1)^n
\stackrel{\xi_n}{\longrightarrow}A(1)  \eqno(\ast\ast)
$$
and the $\xi_n$ define an $\mathbb{M}$-algebra structure $\xi$ on $A(1)$. The correspondence $A\mapsto (A(1),\xi)$ extends
 to a functor $\mathcal{V}^\mathcal{M}\to\mathcal{V}^\mathbb{M}$. Conversely, given an $\mathbb{M}$-algebra $(A(1),\xi)$,
 the adjoints $(\ast)$ of the maps $(\ast\ast)$ define an $\mathcal{M}$-algebra $A$ in $\mathcal{V}$. This correspondence 
defines the inverse functor.
\end{proof}

If $\mathcal{M}$ is a $\mathcal{V}$-operad, we work in either of the categories $\mathcal{V}^\mathbb{M}$ and $\mathcal{V}^\mathcal{M}$
 with a preference for $\mathcal{V}^\mathbb{M}$.

\begin{leer}\label{2_3}
 Let $A$ be an $\mathcal{M}$-algebra in $\cV$ and $K\in\mathcal{V}$. Then the cotensor $A(1)^K$ in $\cV$ has a canonical 
$\mathcal{M}$-algebra structure
 whose structure maps are the adjoints of
$$
\mathcal{M}(n)\times(A(1)^K)^n\times K\cong\mathcal{M}(n)\times(A(1)^n)^K\times K
\stackrel{\id\times e}{\longrightarrow}  \mathcal{M}(n)\times A(1)^n\stackrel{a}{\longrightarrow} A(1)
$$
where $e$ is the evaluation and $a$ is the structure map of $A$.
\end{leer}

\begin{leer}\label{2_4}
If $\mathbb{M}$ is the monad on $\mathcal{V}$ associated with $\cM$ we enrich $\mathcal{V}^\mathbb{M}$ and hence $\cV^{\cM}$ over $\mathcal{V}$ as follows:

If $\mathcal{V}=\Cat$ then $\mM$ is a 2-monad. If $f,g:(A,\xi_A)\to(B,\xi_B)$ are maps of $\mathbb{M}$-algebras, a 2-morphism $s:f\Rightarrow g$ 
in $\Cat^\mathbb{M}$ is a natural transformation satisfying
$$
\xi_B\circ\mathbb{M}s=s\circ\xi_A
$$
If $\mathcal{V}=\SSets$ and $(A,\xi_A)$ and $(B,\xi_B)$ are $\mathbb{M}$-algebras in $\mathcal{V}$, then $B^{\Delta^n}$ is an 
$\mathbb{M}$-algebra by \ref{2_3},
and we define $\SSets^\mathbb{M}
(A,B)_n$ to be the set of all algebra maps $f:A\to B^{\Delta^n}$. 

If $\mathcal{V}=\Top$, we define $\Top^\mathbb{M}(A,B)$ to be the subspace of $\Top(A,B)$ of all maps of $\mathbb{M}$-algebras.
\end{leer}

Let $\mathcal{M}$ be a $\mathcal{V}$-operad and $\mathbb{M}$ its associated monad. The forgetful functor 
$U:\mathcal{V}^\mathcal{M}\to\mathcal{V}$ is 
faithful. If we consider $\mathcal{V}^\mathcal{M}(A,B)$ as a subobject of $\mathcal{V}(A(1),B(1))$ we obtain the $\mathcal{V}$-enrichment of 
$\mathcal{V}^\mathcal{M}$ corresponding to the $\mathcal{V}$-structure of $\mathcal{V}^\mathbb{M}$ under the isomorphism of Proposition \ref{2_2}.

\begin{prop}\label{2_5}
Let $\mathcal{M}$ be a $\mathcal{V}$-operad and $\mathbb{M}$ its associated monad. Then $\mathcal{V}^\mathcal{M}$ and $\mathcal{V}^\mathbb{M}$ 
are $\mathcal{V}$-complete and $\mathcal{V}$-cocomplete.
\end{prop}

\begin{proof}
We first show that $U:\mathcal{V}^\mathbb{M}\to\mathcal{V}$ creates all limits and cotensors. Let $D:\cL\to \mathcal{V}^\mathbb{M}$ 
be a diagram. Then
$$
\lim\ U\circ D\subset \prod_{L\in \cL} D(L)(1).
$$

Coordinatewise operations give the product an $\mathcal{M}$-structure which is inherited by $\lim\ U\circ D$. 
Hence $\lim\ D$ exists in $\mathcal{V}^\mathbb{M}$.

Let $A$ be an $\mathcal{M}$-algebra and $K\in\mathcal{V}$. Then $A(1)^K$ is the underlying object of the $\mathcal{M}$-algebra 
by \ref{2_3} which defines the cotensor
in $\mathcal{V}^\mathbb{M}$.

Next we prove that $\mathcal{V}^\mathbb{M}$ is $\mathcal{V}$-cocomplete. By \cite[VII.2.10]{EKMM} it suffices to show that $\mathcal{M}$ 
preserves reflexive coequalizers. The statement of \cite[VII.2.10]{EKMM} is proved for categories enriched over the category of weak Hausdorff
$k$-spaces, but the proof holds verbatim also for categories enriched over $\Cat,\SSets$, or $\Top$.

A coequalizer diagram
$$
\xymatrix{
X \ar@<0.5ex>[r]^f\ar@<-0.5ex>[r]_g & Y\ar[r]^h & Z
}
$$
is called \textit{reflexive} if there is a map $s:Y\to X$ such that $f\circ s=\id_Y=g\circ s$.

We show that $\mathbb{M}=U\circ F:\Cat\to\Cat$ preserves reflexive coequalizers.
The proofs for $\SSets$ and $\Top$ are analogous.

Since $F$ has a right adjoint it suffices to prove that $U$ preserves reflexive coequalizers. So let
$$
\xymatrix{
X \ar@<0.5ex>[r]^f\ar@<-0.5ex>[r]_g & Y\ar[r]^s & X
}
$$
be a diagram in $\Cat^\mathcal{M}$ such that $f\circ s=\id_Y=g\circ s$ and let 
$$
\xymatrix{
UX \ar@<0.5ex>[r]^{Uf}\ar@<-0.5ex>[r]_{Ug} & UY\ar[r]^h & Z
}
$$
be a coequalizer diagram in $\Cat$. We have to show that $Z$ has a unique $\mathcal{M}$-structure making $h$ a map of $\mathcal{M}$-algebras. The maps 
$$
\xi:\mathcal{M}(n)\times Z^n\to Z
$$
are defined on objects $(A; z_1,\ldots,z_n)$ by choosing $y_1,\ldots,y_n\in UY$ such that $h(y_i)=z_i$ and setting
$$
\xi(A;z_1,\ldots,z_n)=h(\beta(A;y_1,\ldots,y_n))
$$
where $\beta$ is the $\mathcal{M}$-structure map of $Y$. This definition is independent of the choice of the $y_i$: for let 
$x$ be an object in $UX$ such that $f(x)=y_1$; let $\overline{y}_1=g(x)$ and let $\overline{x}=\alpha(A;x,s(y_2),\ldots,s(y_n))$,
where $\alpha$ is the $\mathcal{M}$-structure map of $X$. Since $f,g$, and $s$ are maps of $\mathcal{M}$-algebras\\
\begin{align*}
f(\overline{x}) & =\beta(A;f(x), f\circ s(y_2),\ldots,f\circ s(y_n)) =\beta(A;y_1,y_2,\ldots,y_n))\\
g(\overline{x}) & =\beta(A;g(x), g\circ s(y_2),\ldots,g\circ s(y_n)) =\beta(A;\overline{y}_1,y_2,\ldots,y_n))
\end{align*}
Hence 
$$
h(\beta(A;y_1,y_2,\ldots,y_n))=h(\beta(A;\overline{y}_1,y_2,\ldots,y_n))
$$
On morphisms we proceed analogously: if suffices to define $\xi$ for the generating morphisms
$$
(w:z\to z')=h(v:y\to y')
$$
which can be done in the same way as for objects.
\end{proof}

\begin{coro}\label{2_6}
Let $\mathcal{M}$ be a $\mathcal{V}$-operad and $\mathbb{M}$ its associated monad. Then $\mathcal{V}^\mathcal{M}$ and $\mathcal{V}^\mathbb{M}$ 
are tensored and cotensored over $\mathcal{V}$.
\end{coro}

\begin{leer}\label{2_6a}
We recall the explicit construction of the tensor in $\mathcal{V}^\mathcal{M}$ from the proof of \cite[VII.2.10]{EKMM}. Note that the tensor in our categories $\cV$ is the product.
Let $\cM$ be a $\cV$-operad and $(\mM,\mu,\iota)$ its
associated monad. For $K$ in $\cV$ and $(X,\xi)$ in $\cV^{\mM}$ the tensor $K\otimes X$ is the coequalizer
$$
\xymatrix{
\mM(K\times \mM X) \ar@<0.5ex>[rr]^{\mM(\id\times \xi)}\ar@<-0.5ex>[rr]_{\mu\circ \mM \nu}
&& \mM(K\times X)\ar[rr] && K\otimes X
}
$$
in $\cV$, where $\nu$ is induced by the diagonal on $K$:
$$
\nu: K\times \mM X \cong \coprod\limits_{n\ge 0} K\times \mathcal{M}(n)\times_{\Sigma_n}X^n
\to \coprod\limits_{n\ge 0}\mathcal{M}(n)\times_{\Sigma_n}(K\times X)^n
$$
\end{leer}

\begin{leer}\label{2_7}
We have a sequence of adjunctions where $N$ is the nerve functor, $\Sing$ is the singular functor, $|-|$ the topological realization, and
$\cat$ the categorification functor. For details of the latter see \cite[Chapter II.4]{GZ}, where it is also proved that
$\cat\circ N=\Id$. The lower arrows are the left adjoints.
$$
\xymatrix
{
\Cat \ar@<0.5ex>[rr]^N && \SSets \ar@<0.5ex>[ll]^{\cat} \ar@<-0.5ex>[rr]_{|-|} && \Top \ar@<-0.5ex>[ll]_{\Sing}
}
$$
Since all these functors preserve products, they are strong monoidal functors. This is well-known for $\Sing,\ N, |-|$, and we now include 
a short proof for $\cat$. Consider each simplicial set as a discrete simplicial category. Let $\Delta$ denote the standard simplicial indexing
category and $\nabla:\Delta\to \Cat$ the functor sending $[n]$ to the category, also denoted by $[n]$, with objects $0,1,\ldots ,n$ and a unique morphism $k\to l$
whenever $k\leq l$. Then $\cat$ is the categorical realization functor
$$
\cat (K_\bullet)=K_\bullet\times_\Delta \nabla
$$
(see \ref{2_0} for the definition of $\times_\Delta$). Now
$$\begin{array}{rcl}
\cat(K_\bullet\times L_\bullet) &= &(K_\bullet\times L_\bullet)\times_\Delta \nabla \cong (K_\bullet\times L_\bullet)\times_{\Delta\times \Delta}
(\Delta\times \Delta)\times_\Delta\nabla \\
&\cong &(K_\bullet\times L_\bullet)\times_(\Delta\times \Delta) (\nabla\times \nabla)\cong
(K_\bullet\times_\Delta\nabla)\times (L_\bullet\times_\Delta\nabla) \\
& = &\cat(K_\bullet)\times\cat(L_\bullet)
\end{array}
$$
provided $(\Delta\times \Delta)\times_\Delta\nabla \cong \nabla\times \nabla$ as $(\Delta\times \Delta)$-diagram in $\Cat$, where $\Delta$
operates diagonally on $\Delta\times \Delta$ by right composition. As a simplicial set $\Delta(-,[k])$ is
the standard simplicial $k$-simplex, hence $\Delta(-,[k])=N([k])$. It follows that
$$(\Delta(-,[k])\times \Delta(-,[l]))\times_\Delta\nabla= \cat(N([k])\times N([l]))\cong \cat(N([k]\times[l]))=[k]\times[l]$$
natural in $[k]$ and $[l]$.

Since all functors are strong monoidal, any $\Cat$-enriched category $\cA$ is
$\SSets$-enriched via the nerve functor $N$ and each $\Top$-enriched category $\cB$ is $\SSets$-enriched via the singular functor $\Sing$
\cite[6.4.3]{Bor2}. Moreover, $\cA$ as a $\SSets$-category is tensored over $\SSets$ by
$$K\otimes^s A=\cat(K)\otimes A$$
and $\cB$ as a $\SSets$-category is tensored over $\SSets$ by
$$K\otimes^s B=|K|\otimes B,$$
where $\otimes^s$ stands for the tensor over $\SSets$ while $\otimes$ stands for the tensor over $\Cat$ respectively $\Top$.
\end{leer}

\begin{leer}\label{2_7a}
 For later use we investigate the behavior of the tensors in our categories of algebras under the nerve and the topological
 realization functor. 
 
Let $\cM$ be an operad in $\Cat$ and $(\mM,\mu_C,\iota_C)$ its associated monad. Let $X$ be an $\cM$-algebra.
Since the nerve functor preserves products and sums $N\cM$
is a simplicial operad and $NX$ is an $N\cM$-algebra. We obtain an induced functor
$$
N^{\alg}: \Cat^{\cM}\to \SSets^{N\cM}.
$$
Let $(\mS,\mu_S,\iota_S)$ be the monad associated to $N\cM$.
From the explicit definition of the free functor \ref{2_2a} we deduce that for $C$ in $\Cat$ there is a unique natural transformation
$$
\alpha(C): \mS(NC)\to N(\mM C)
$$
such that $\alpha(C)\circ \iota_S(NC)=N(\iota_C(C))$. If $\cM$ is $\Sigma$-free, it is easy to show that $\alpha$ is
a natural equivalence. In general, this is not true. In the diagram

$$
\xymatrix{
\mS NC\ar[rr]^{\iota_S(\mS NC)}\ar[d]_{\alpha(C)} & & \mS^2 NC\ar[rr]^{\mu_S(NC)}\ar[d]^{\mS\alpha(C)} && \mS NC \ar[dd]^{\alpha(C)}\\
N\mM C\ar[rr]^{\iota_S(N\mM C)} \ar[rrd]_{N(\iota_C(\mM C))} && \mS N\mM C \ar[d]^{\alpha(\mM C)}& &\\
 && N\mM^2 C \ar[rr]^{N\mu_C (C)} &&  N\mM C
 }
 $$
 the upper left square and the triangle commute by naturality of $\iota_S$ and the definition of $\alpha$. Since 
 $\mu_S(NC)\circ \iota_S(\mS NC)=\id$ and $N\mu_C (C)\circ N(\iota_C(\mM C))=\id$ the big outer diagram commutes.
 Hence, by the universal property of $\iota_S(\mS NC)$, the right square commutes.
 
 Let $\nu_C(K,C): K\times \mM C\to \mM(K\times C)$ and $\nu_S(L,X): L\times \mS X\to \mS(L\times X)$ be the maps used in
 the explicit description of the tensors $\otimes^c$ in $\Cat^{\mM}$ and $\otimes^s$ in $\SSets^{\mS}$ respectively.
 Then, by construction, $\alpha(K\times C)\circ \nu_S(NK, NC)=N\nu_C(K,C)\circ (\id_{NK}\times \alpha(C))$. If $(C,\xi)$
 is an $\mM$-algebra, the $\mS$-algebra structure on $NC$ is given by $\mS NC\xrightarrow{\alpha(C)} N\mM C\xrightarrow{N\xi}
 NC$. Using these observations it is now easy to show that the maps
 $$\mS(NK\times \mS NC)\xrightarrow{\mS(\id_{NK}\times \alpha(C))} \mS(NK\times N\mM C) \xrightarrow{\alpha(K\times \mM C)}
 N\mM(K\times \mM C)$$
 and
 $$ \mS(NK\times NC)\xrightarrow{\alpha(K\times C)} N\mM(K\times C)$$
 define a map from the coequalizer diagram \ref{2_6a} for $NK\otimes ^s NC$ to the diagram obtained from the 
 coequalizer diagram for $K\otimes^c C$ by applying the nerve functor $N$, and we obtain a natural transformation
 $$\tau(K,C): NK\otimes^s NC\to N(K\otimes^c C).$$
 
The topological realization functor behaves better than the nerve, because it preserves products and colimits. So our
arguments for the nerve functor carry over verbatim to the realization functor with the nice addition that $\alpha$ is
a natural isomorphism, and so is the induced map on the coequalizers defining the tensors. We obtain: there is a 
natural isomorphism
$$\omega(L,X) :\ |L|\otimes^t |X|\cong |L\otimes^s X|$$
where $\otimes^t$ is the tensor in $\Top^{|\cM|}$, where $\cM$ is an operad in $\SSets$ and $(X,\xi)$ is an $\mM$-algebra.
\end{leer}

\begin{leer}\label{2_8}
In the analysis of the homotopy type of the homotopy colimit we will make use of lax functors and their rectifications. Let $\cL$ be a small
category. A \textit{lax functor} $F:\cL \to \Cat$ is a pair of functions assigning a category
$F(L)$ to each $L\in \ob \cL$ and a functor $F(f):F(K)\to
F(L)$ to each morphism $f:K\to L$ of $\cL$ together with natural
transformations
$$ \rho(L):F(id_L) \to id_{F(L)} \qquad  \sigma(f,g):F(f\circ g)
\to F(f)\circ F(g)$$
such that the following diagrams commute:
$$\xymatrix{
F(f) \circ F(id_K) \ar[rrd]_{F(f)\rho(K)} && F(f)\ar[ll]_(0.4){\sigma(f,id_K)} \ar@{=}[d]\ar[rr]^(0.4){\sigma(id_L,f)} && F(id_L) \circ F(f)\ar[lld]^{\rho(L)F(f)}\\
&&F(f) &&
}
$$
$$
\xymatrix{
F(f\circ g\circ h)\ar[rr]^{\sigma(f\circ g,h)} \ar[d]_{\sigma(f,g\circ h)} && F(f\circ g)\circ F(h)\ar[d]^ {\sigma(f,g)F(h)} \\
F(f)\circ F(g\circ h) \ar[rr]^ {F(f)\sigma(g,h)} && F(f)\circ F(g)\circ F(h)
}$$

Let $F,\ G:\cL \to \Cat$ be lax
functors. A (left) \textit{lax natural transformation} $d: F\to G$ is a pair of
functions assigning a functor $d(L): F(L) \to G(L)$ to each $L\in
\ob\cL$ and a natural transformation $d(f):G(f)\circ
d(K)\to d(L)\circ F(f)$ to each morphism $f:K\to L$ of $\cL$
such that the following diagrams of natural transformations commute
$$\xymatrix{
G(id_L)\circ d(L) \ar[rr]^{d(id_L)}\ar[rd]_ {\rho(L)d(L)}&& d(L) \circ F(id_L) \ar[ld]^{d(L)\rho(L)}\\
& d(L) & 
}
$$
and for $g\circ f: K\to L\to M$
$$\xymatrix{
G(g\circ f)\circ d(K)\ar[rr]^{d(g\circ f)}\ar[dd]^{\sigma(g,f)d(K)}  && d(M)\circ F(g\circ f)\ar[rd]^{d(M)\sigma(g,f)}\\
&& & d(M)\circ F(g)\circ F(f)\\
G(g)\circ G(f)\circ d(K) \ar[rr]^{G(g)d(f)} & & G(g)\circ d(L)\circ F(f)\ar[ru]_{d(g)F(f)} &
}
$$
We can compose lax natural transformations in the obvious way and obtain a category of lax functors and lax natural
transformations. 
\end{leer}

\begin{rema}
Our lax functors and lax natural transformations are dual to Street's notions in \cite{Street}, \cite{Street} and are
called op-lax functors and op-lax transformations by Thomason in \cite{Thom0}.
\end{rema}

Lax functors can be rectified by Street's rectification constructions \cite{Street}. We will use his first one:

\begin{prop}\label{2_9}
Let $\cL$ be a small category. There is a functor $F\mapsto SF$ from the category of of lax functors $F:\cL \to \Cat$
and lax natural transformations to the category of strict functors $\cL\to \Cat$
and strict natural transformations with the following
properties:
\begin{itemize}
\item[(1)] For each $L\in \ob\cL$ there is a pair of adjoint functors
$$
\varepsilon(L):SF(L) \leftrightarrows F(L):\eta(L)
$$
The $\eta(L)$ combine to a lax natural transformation $\eta:F\to SF$ such that
for a lax natural transformation $d: F\to G$ the diagram
$$\xymatrix{
F(L)\ar[r]^{d(L)}\ar[d]_{\eta(L)} & G(L)\ar[d]^{\eta(L)}\\
SF(L)\ar[r]^{Sd(L)} & SG(L)
}
$$
commutes. If $F$ is a genuine functor then the $\varepsilon(L)$ combine to a
strict natural transformation $\varepsilon:SF \to F.$
\item[(2)] If $k:\cK\to \cL$ is a functor, then $F\circ k: \cK\to \Cat$
 is a lax functor and there
is a natural transformation $\xi:S(F\circ k)\to SF\circ k$ such
that
$$\xymatrix
{
&& S(F\circ k)(K)\ar[rrd]^{\varepsilon_{F\circ k}(K)} \ar[dd]^{\xi(K)}&& \\
F\circ k(K) \ar[rru]_{\eta_{F\circ k}(K)}\ar[rrd]_{\eta_F(k(K))}&&&& F\circ k(K)\\
&& SF\circ k(K)\ar[rru]^{\varepsilon_F(k(K))} &&
}
$$
commutes for all $K\in \ob \cK$. Moreover, $\xi$ is natural
with respect to lax natural transformations $F\to G.$
\end{itemize}
\end{prop}

\section{Lax morphisms of $\mM$-algebras}
Throughout this section let $\cM$ be an operad in $\Cat$ and $(\mM,\mu,\iota)$ its associated monad.
In most cases functors between $\mM$-algebras do not preserve the algebra structure in the 
strict sense. Therefore
we are interested in a lax version of $\Cat^\mathbb{M}$. We first fix some notation. 
\begin{nota}\label{3_1}
Let
$$
\input{7_1.tex}
$$
be a diagram of 1- and 2-cells in a 2-category $\cK$. We denote by 
$\sigma_2\circ_2\sigma_1$ and $\tau\circ_1\sigma_1$ the composite 2-cells
$$
\sigma_2\circ_2\sigma_1: f_1 \stackrel{\sigma_1}{\Longrightarrow} f_2\stackrel{\sigma_2}{\Longrightarrow}f_3
$$
$$
\tau\circ_1\sigma_1:g_1\circ f_1 \Longrightarrow g_2\circ f_2
$$
We use the convention that $g_1\circ_1\sigma_1=\id_{g_1}\circ_1\sigma_1$.
\end{nota}

\begin{defi}\label{3_2} A
\textit{lax morphism} 
$$
(f,\overline{f}):(A,\xi_A)\to (B,\xi_B)
$$
of $\mM$-algebras is a pair consisting of a functor $f:A\to B$ and a natural transformation 
$\overline{f}:\xi_B\circ\mathbb{M}f\Rightarrow f\circ\xi_A$ 
$$
\xymatrix{
\mM A \ar[rr]^{\mM f} \ar[dd]_{\xi_A} && 
\mM B \ar[dd]^{\xi_B}
\\
&\Downarrow \bar{f}
\\
A \ar[rr]_f && B
}
$$
satisfying 
\begin{enumerate}
\item[(1)] $\overline{f}\circ_1 \mu(A)= (\overline{f}\circ_1(\mathbb{M}\xi_A))\circ_2 (\xi_B\circ_1\mathbb{M}\overline{f})$
$$
\xymatrix{
\mM^2A \ar[rr]^{\mM^2f} \ar[dd]_{\mu A} &&
\mM^2B \ar[dd]^{\mu B}
&&
\mM^2A \ar[rr]^{\mM^2f} \ar[dd]_{\mM\xi_A} &&
\mM^2B \ar[dd]^{\mM\xi_B}
\\
&&&&& \Downarrow \mM\bar{f}
\\
\mM A \ar[rr]^{\mM f} \ar[dd]_{\xi_A} && 
\mM B \ar[dd]^{\xi_B}
&=&
\mM A \ar[rr]^{\mathbb{M}f} \ar[dd]_{ \xi_A}&& 
\mathbb{M}B \ar[dd]^{\xi_B}
\\
&\Downarrow \bar{f} &&&& \Downarrow \bar{f}
\\
A \ar[rr]_f && B && A \ar[rr]_f && B
}
$$
(recall that $\xi_A\circ\mathbb{M}\xi_A=\xi_A\circ\mu(A)$)
\item[(2)] $\overline{f}\circ_1 \iota(A)=\id_f$

$$
\xymatrix{
A \ar[rr]^{f} \ar[d]_{\iota A} &&
B \ar[d]^{\iota B}
&& A \ar[rrr]^f \ar[dd]_{\id} &&& B \ar[dd]^{\id} 
\\
\mathbb{M}A \ar[rr]^{\mathbb{M}f} \ar[dd]_{\xi_A} && 
\mathbb{M}B \ar[dd]^{\xi_B}
&=&
\\
&\Downarrow \bar{f} &&& A\ar[rrr]_f &&&B
\\
A \ar[rr]_f &&B 
}
$$
\end{enumerate}

A \textit{2-cell} $s:(f,\overline{f})\Rightarrow (g,\overline{g})$ between two lax morphisms
$(f,\overline{f}),(g,\overline{g}):A\to B$ is a natural transformation $s:f\Rightarrow g$ satisfying
$$\overline{g}\circ_2 (\xi_B\circ_1 \mathbb{M}s)\ =\ (s\circ_1 \xi_A)\circ_2 \overline{f}$$
\end{defi}

We can compose lax morphisms and $2$-cells between lax morphisms in the obvious way and obtain a 2-category $\Cat^{\mM}$Lax.

For later use it is convenient to have a more explicit description of a lax morphism.

\begin{leer}\label{3_3}
A lax morphism $(f,\overline{f}):\mathcal{A}\to\mathcal{B}$ of $\mM$-algebras is explicitly given by a functor 
$f:\mathcal{A}\to\mathcal{B}$ and morphisms
$$
\overline{f}(A;K_1,\ldots,K_n): A(fK_1,\ldots,fK_n)\to f(A(K_1,\ldots,K_n))
$$
in $\mathcal{B}$, where $A\in\mathcal{M}(n)$, $K_1,\ldots,K_n\in\mathcal{A}$, such that

\begin{enumerate}
\item[(1)] for $\alpha:A_1\to A_2$ in $\mathcal{M}(n)$ and $k_i:K_i\to K'_i$ in $\mathcal{A}$, the diagram 
$$
\xymatrix{
A_1(fK_1,\ldots,fK_n)\ar[rr]^{\overline{f}(A_1;K_1,\ldots,K_n)} \ar[d]^{\alpha(f(k_1),\ldots,f(k_n))}&&
f(A_1(K_1,\ldots,K_n)) \ar[d]^{f(\alpha(k_1,\ldots,k_n))}
\\
A_2(fK'_1,\ldots,fK'_n)\ar[rr]^{\overline{f}(A_2;K'_1,\ldots,K'_n)}&&
f(A_2(K'_1,\ldots,K'_n)) 
}
$$
commutes.
\item[(2)]
for $A\in\mathcal{M}(k)$, $B_i\in\mathcal{M}(r_i)$, $K_{ij}\in\mathcal{A}$
\begin{multline*}
\overline{f}(A\ast(B_1\oplus\ldots\oplus B_k); K_{11},\ldots,K_{1r_1},\ldots,K_{k1},\ldots,K_{kr_k}) \\
=\overline{f}(A;B_1(K_{11},\ldots,K_{1r_1}),\ldots, B_k(K_{k1},\ldots,K_{kr_k}))\\ 
\circ A(\overline{f}(B_1;K_{11},\ldots,K_{1r_1}),\ldots,\overline{f}(B_k;K_{k1},\ldots,K_{kr_k}))
\end{multline*}
where $\ast$ is the composition in the operad $\mathcal{M}$.
\item[(3)]
$\overline{f}(\id;K)=\id_{f(K)}$
\end{enumerate}
\end{leer}

\begin{rema}\label{3_3a}
Our definition of a lax morphism coincides for monoidal, permutative, and braided monoidal categories with the ones in the literature.
Lax $2$-fold monoidal functors as defined in \cite{BFSV} are lax morphisms in our sense with the additional
condition that the units 0 are strictly preserved.
We will provide details in Section 8.
\end{rema}

Like in the topological case, our homotopy colimits of $\cL$-diagrams will have a universal property which cannot
 be formulated in the category of functors from $\cL$ to  $\Cat^{\mM}$Lax and natural transformations. We have
to allow ``lax morphisms'' between diagrams.  We now introduce the relevant functor 2-category 
$\Func(\mathcal{L}, \Cat^{\mM}\Lax)$. The homotopy colimit functor will be defined as a functor from
 $\Func(\mathcal{L}, \Cat^{\mM}\Lax)$ to $\Cat^{\mM}$, which is left adjoint to the constant diagram functor.

\begin{leer}\label{3_4}
Let $\cL$ be a small category and $\cK$ be an arbitrary 2-category.

The objects of $\Func(\mathcal{L},\mathcal{K})$ are functors $D:\mathcal{L}\to\mathcal{K}$ of 1-categories, i.e. $\mathcal{L}$-diagrams in $\cK$. 
A morphism or 1-cell $j:D\to F$ of $\mathcal{L}$-diagrams  is a \textit{homotopy morphism} 
(or an op-lax natural transformation in Thomason's terminology \cite{Thom1},\cite{Thom2}). 
It assigns to each object $L$ in $\mathcal{L}$ a 1-cell
$$
j_L:DL\to FL
$$
and to each morphism $\lambda:L_0\to L_1$ a 2-cell
$$
j_{\lambda}:F\lambda\circ j_{L_0}\Rightarrow j_{L_1}\circ D\lambda
$$
in $\mathcal{K}$ such that $j_{id}=\id$ and for $L_0\stackrel{\lambda_0}{\to} L_1\stackrel{\lambda_1}{\to}L_2$
$$
j_{\lambda_1\circ\lambda_2} =(j_{\lambda_1}\circ_1 D\lambda_0)\circ_2 (F\lambda_1\circ_1j_{\lambda_0})
$$
$$
\xymatrix{\ar @{} [drr] |{\stackrel{j_{\lambda_0}}{\Rightarrow}}
DL_0 \ar[rr]^{D{\lambda_0}} \ar[d]^{j_{L_0}} 
&& \ar @{} [drr] |{\stackrel{j_{\lambda_1}}{\Rightarrow}}
DL_1 \ar[rr]^{D{\lambda_1}} \ar[d]^{j_{L_1}} 
&& 
DL_2  \ar[d]^{j_{L_2}}
\\
FL_0 \ar[rr]_{F{\lambda_0}}
&&
FL_1 \ar[rr]_{F{\lambda_1}}
&&
FL_2
}
$$

A 2-cell 
$\xymatrix{
D\ar@/_1pc/[rr]_{k} \ar@/^1pc/[rr]^{j}
& {\Downarrow s}
&
F
}$ in $\Func(\mathcal{L},\mathcal{K})$ assigns to each object $L$ in $\mathcal{L}$ a 2-cell
$$
\xymatrix{
DL\ar@/_2pc/[rr]_{k_L} \ar@/^2pc/[rr]^{j_L}
& {\Downarrow s_L}
&
FL
}
$$
in $\mathcal{K}$ such that for $\lambda : L_1\to L_2$ in $\mathcal{L}$
$$
k_\lambda\circ_2(F\lambda\circ_1 s_{L_1})=(s_{L_2}\circ_1 D\lambda)\circ_2 j_\lambda
$$
Composition in $\Func(\mathcal{L},\mathcal{K})$ is the canonical one.
\end{leer}
If $\cK$ is $\Cat^{\mM}\Lax$ the explicit coherence conditions of homotopy morphisms in $\Func(\mathcal{L},\mathcal{K})$
are quite involved because the morphisms in each diagram are lax. For the reader's convenience we formulate them in Appendix 1.

Let $\mathcal{K}^\mathcal{L}$ be the subcategory of $\Func(\mathcal{L},\mathcal{K})$ where the 1-cells are strict
 morphisms $j:D\to F$, i.e. for $\lambda:L_0\to L_1$
$$
F\lambda\circ j_{L_0}=j_{L_1}\circ D\lambda \quad\textrm{ and }\quad j_\lambda=\id.
$$

Now suppose that $\cK$ is a cocomplete, tensored 2-category. We are going to construct a rectification 2-functor
$$
\Func(\mathcal{L},\mathcal{K})\to \mathcal{K}^\mathcal{L}
$$
left adjoint to the inclusion functor. For this purpose we construct an isomorphism from the
category $\Func(\mathcal{L},\mathcal{K})(D,F)$ of 1- and 2-cells between the diagrams $D$ and $F$ in 
$\Func(\mathcal{L},\mathcal{K})$ to the category of 1- and 2-cells in the 2-category $\Cat^{\mathcal{L}^{op}\times\mathcal{L}}$
between two $ \mathcal{L}^{op}\times\mathcal{L}$-diagrams in $\Cat$  we are now going to define. 

If $D,F:\mathcal{L}\to\mathcal{K}$ are $\mathcal{L}$-diagrams in $\mathcal{K}$ we have an induced diagram 
$$
\mathcal{K}(D,F):\mathcal{L}^{op}\times\mathcal{L}\to\Cat, \quad (L,L')\mapsto\mathcal{K}(DL, FL').
$$
Let $L/\mathcal{L}/L'$ denote the category whose object are diagrams
$$
(L_0,f_0,g_0):L\stackrel{f_0}{\longrightarrow} L_0\stackrel{g_0}{\longrightarrow} L'
$$
and whose morphisms $\lambda_1:(L_0,f_0,g_0)\to(L_1,f_1,g_1)$ are morphisms $\lambda_1:L_0\to L_1$ in $\mathcal{L}$ making
$$
\xymatrix{
&& L_0 \ar[drr]^{g_0}\ar[dd]^{\lambda_1} && 
\\
L \ar[urr]^{f_0} \ar[rrd]_{f_1} && && L'
\\
&& L_1 \ar[urr]_{g_1}
}
$$
commute. We define
$$
-/\mathcal{L}/-: \mathcal{L}^{op}\times\mathcal{L}\to\Cat \quad\textrm{as}\quad  (L,L')\mapsto L/\mathcal{L}/L'
$$
The following result will imply a rectification result if $\mathcal{K}=\Cat^{\mM}$.

\begin{prop}\label{3_5}
There is an isomorphism of categories
$$
\alpha:\Func(\mathcal{L},\mathcal{K})(D,F)\cong\Cat^{\mathcal{L}^{op}\times\mathcal{L}}((-/\mathcal{L}/-),\mathcal{K}(D,F))
$$
2-natural in $F$ with respect to 1- and 2-cells in $\cK^{\cL}$.
\end{prop}

\begin{proof}
Let $j:D\to F$ be an object in $\Func(\mathcal{L},\mathcal{K})(D,F)$, i.e. for each $\lambda_1:L_0\to L_1$ we are given
$$
\xymatrix{\ar @{} [drr] |{\stackrel{j_{\lambda_1}}{\Rightarrow}}
DL_0 \ar[rr]^{D\lambda_1}\ar[d]^{j_{L_0}} &&  DL_1\ar[d]^{j_{L_1}}
\\
F L_0 \ar[rr]^{F\lambda_1} && FL_1
}
$$
If $\lambda_1:(L_0,f_0,g_0)\to(L_1,f_1,g_1)$ is a morphism in $L/\mathcal{L}/L'$ we have
a diagram
$$
\xymatrix{
& DL_0 \ar[rr]^{j_{L_0}}\ar[dd]^{D_{\lambda_1}} \ar @{} [ddrr] |{\Downarrow j_{\lambda_1}}
&&  FL_0\ar[dd]^{F_{\lambda_1}} \ar[dr]^{Fg_0}
\\
DL\ar[ur]^{Df_0}\ar[dr]_{Df_1} 
&&&& FL'
\\
& DL_1 \ar[rr]^{j_{L_1}} 
&& FL_1\ar[ur]_{Fg_1}
}\eqno(A)
$$
with commutative triangles. We define the functor
$$
\alpha(j)_{L,L'}:L/\mathcal{L}/L'\to\mathcal{K}(DL,FL')
$$
on objects by $\alpha(j)_{L,L'}(L_0,f_0,g_0)=Fg_0\circ j_{L_0}\circ Df_0$, and on the morphism $\lambda_1$ to be the 2-cell 
depicted by diagram (A). The conditions on $j$ make this a functor.

By construction, the collection of the $\alpha(j)_{L,L'}$ define a morphism
$$
\alpha(j):(-/\cL/-)\to\mathcal{K}(D,F)
$$
of $(\mathcal{L}^{op}\times\mathcal{L})$-diagrams in $\Cat$.

Now let $s:i\Rightarrow j$ be a morphism in $\Func(\mathcal{L},\mathcal{K})(D,F)$. By definition, $s$ consists of 2-cells
$$
\xymatrix{
DL\ar@/_2pc/[rrr]_{j_L} \ar@/^2pc/[rrr]^{i_L} 
&& {\Downarrow^{s_{L}}}
&
FL
}
$$
satisfying

\qquad\qquad $j_\lambda\circ_2(F\lambda\circ_1 s_{L_0})=(s_{L_1}\circ_1D\lambda)\circ_2 i_\lambda$\quad for $\lambda:L_0\to L_1$.
\hfill{(B1)}

Then $\alpha(s)$ is required to be a modification
$$
\xymatrix{
L/\mathcal{L}/L'\ar@/_2pc/[rrr]_{\alpha(j)_{L,L'}} \ar@/^2pc/[rrr]^{\alpha(i)_{L,L'}} 
&& {\Downarrow \alpha(s)_{L,L'}}
&
\mathcal{K}(DL,FL')
}
$$
satisfying

(B2)
$$
\xymatrix{
\overline{L}/\mathcal{L}/\overline{L}'\ar@/_2pc/[rrr]^{\alpha(j)_{\overline{L},\overline{L}'}} 
\ar@/^2pc/[rrr]^{\alpha(i)_{\overline{L},\overline{L}'}} \ar[dd]_{\lambda^\ast\circ\mu_\ast}
&& {\Downarrow \alpha(s)_{\overline{L},\overline{L}'}}
&
\mathcal{K}(D\overline{L},F\overline{L}')\ar[dd]^{(D\lambda)^\ast\circ(F\mu)_\ast}
\\
\\
L/\mathcal{L}/L'\ar@/_2pc/[rrr]_{\alpha(j)_{L,L'}} \ar@/^2pc/[rrr]^{\alpha(i)_{L,L'}}
&& {\Downarrow \alpha(s)_{L,L'}}
&
\mathcal{K}(DL,FL')
}
$$
$$
((D\lambda)^\ast\circ(F\mu)_\ast)\circ_1\alpha(s)_{\overline{L},\overline{L}'}
=\alpha(s)_{L,L'}\circ_1 (\lambda^\ast\circ\mu_\ast)
$$
for $\lambda:L\to\overline{L}$ and $\mu:\overline{L}'\to L'$ in $\mathcal{L}$.

For an object $(L_0,f_0,g_0)$ in $L/\mathcal{L}/L'$ we define the natural transformation $\alpha(s)_{{L},{L}'}$ by
$$
\alpha(s)_{{L},{L}'}(L_0,f_0,g_0)=F(g_0)\circ_1 s_{L_0}\circ_1 D(f_0)
$$
which satisfies condition (B2).

This defines the functor $\alpha$. By construction it is 2-natural in $F$. 

We now define the inverse functor
$$
\beta:\Cat^{\mathcal{L}^{\op}\times\mathcal{L}}((-/\mathcal{L}/-), \mathcal{K}(D,F))\to \Func(\mathcal{L},\mathcal{K})(D,F)
$$
Let $G,H:(-/\mathcal{L}/-)\to \mathcal{K}(D,F)$ be strict morphisms of $(\mathcal{L}^{\op}\times\mathcal{L})$-diagrams in $\Cat$.
 For each pair $(L,L')\in\mathcal{L}^{\op}\times\mathcal{L}$ we are given a functor
$$
G_{L,L'}:L/\mathcal{L}/L'\to\mathcal{K}(DL, FL')
$$
and for each pair of morphisms $\lambda:L\to\overline{L}$ and $\mu:\overline{L}'\to L'$ in $\mathcal{L}$ a commutative diagram
$$
\xymatrix{
\overline{L}/\mathcal{L}/\overline{L}' \ar[rr]^{G_{\overline{L},\overline{L}'}} \ar[d]_{\lambda^\ast\circ\mu_\ast}
&& 
\mathcal{K}(D\overline{L}, F\overline{L}') \ar[d]^{(D\lambda)^\ast\circ(F_\mu)_\ast}
\\
L/\mathcal{L}/L' \ar[rr]^{G_{L,L'}}
&&
\mathcal{K}(DL, FL) 
}
$$
so that
$$
G_{L,L'}(L_0, f_0\circ\lambda,\mu\circ g_0) =F\mu\circ G_{\overline{L},\overline{L}'}(L_0,f_0,g_0)\circ D\lambda,
$$
and for a morphism $\lambda_1:(L_0,f_0,g_0)\to (L_1,f_1,g_1)$ in $\overline{L}/\mathcal{L}/\overline{L}'$
$$
\begin{array}{rcl}
F\mu\circ G_{\overline{L},\overline{L}'}(\lambda_1:(L_0,f_0,g_0) 
&\to & (L_1,f_1,g_1))\circ D\lambda\\
= G_{L,L'}(\lambda_1:(L_0,f_0\circ\lambda,\mu\circ g_0) 
&\to& (L_1,f_1\circ\lambda, \mu\circ g_1))
\end{array}
$$
We define 
$$
\beta(G)_L=G_{L,{L}}(L,\id_L,\id_L):DL\to FL
$$
and for $\lambda:L\to \overline{L}$
$$
\beta(G)_\lambda=G_{{L},\overline{L}}(\lambda:(L,\id,\lambda)\to(\overline{L},\lambda,\id))
$$
Since $G_{{L},\overline{L}}(L,\id,\lambda)=F\lambda\circ G_{{L},{L}}(L,\id,\id)$ and 
$G_{{L},\overline{L}}(\overline{L},\lambda,\id)= G_{\overline{L},\overline{L}}(\overline{L},\id,\id)\circ D\lambda$, we have
$$
\beta(G)_\lambda: F\lambda\circ\beta(G)_L\Rightarrow\beta(G)_{\overline{L}}\circ D\lambda
$$
Given $L_0\stackrel{\lambda_1}{\longrightarrow} L_1\stackrel{\lambda_2}{\longrightarrow} L_2$, then
$$
F\lambda_2\circ_1 G_{L_0,L_1}(\lambda_1:(L_0,\id,\lambda_1)\to (L_1,\lambda_1,\id))=
G_{L_0,L_2}(\lambda_1:(L_0,\id,\lambda_2\circ\lambda_1)\to (L_1,\lambda_1,\lambda_2))
$$
and
$$
G_{L_1,L_2}(\lambda_2:(L_1,\id,\lambda_2)\to(L_2,\lambda_2,\id))\circ_1 D\lambda_1 =G_{L_0,L_2}(\lambda_2:(L_1,\lambda_1,\lambda_2)
\to (L_2,\lambda_2\circ\lambda_1,\id))
$$
so that the coherence conditions for the $\beta(G)_{\lambda}$ hold.

Now let $s:G\Rightarrow H$ be a modification satisfying condition (B2), i.e. for $\lambda:L\to \overline{L}$ and $\mu:\overline{L}'\to L'$
$$
F\mu\circ_1 s_{\overline{L},\overline{L}'}(L_0,f_0,g_0)\circ_1 D\lambda=s_{L,L'}(L_0,f_0\circ\lambda,\mu\circ g_0).
$$
We have to find a 2-cell
$$
\beta(s):\beta(G)\Rightarrow \beta(H)
$$
in $\Func(\mathcal{L},\mathcal{K})(D,F)$. We define
$$
\beta(s)_L=s_{L,L}(L,\id,\id):\beta(G)_L\Rightarrow \beta(H)_L.
$$
Since $s$ is a modification the following diagram commutes for $\lambda:L_0\to L_1$
$$
\xymatrix{
G_{L_0,L_1}(L_0,\id,\lambda) \ar[rr]^{s_{L_0,L_1}(L_0,\id,\lambda)} \ar[d]_{G_{L_0,L_1}(\lambda)}
&& H_{L_0,L_1}(L_0,\id,\lambda) \ar[d]_{H_{L_0,L_1}(\lambda)}
\\
G_{L_0,L_1}(L_1,\lambda,\id) \ar[rr]^{s_{L_0,L_1}(L_1,\lambda,\id)} 
&& H_{L_0,L_1}(L_1,\lambda,\id) 
}
$$
and we obtain for $\lambda:(L_0,\id,\lambda)\to (L_1,\lambda,\id)$
$$
\begin{array}{rcl}
\beta(H)_\lambda\circ_2(F\lambda\circ_1\beta(s)_{L_0})
&=& H_{L_0,L_1}(\lambda) \circ_2(F\lambda\circ_1 s_{L_0,L_0}(L_0,\id,\id))\\
&=& H_{L_0,L_1}(\lambda) \circ_2(s_{L_0,L_1}(L_0,\id,\lambda))\\
&=& s_{L_0,L_1}(L_1,\lambda,\id)\circ_2G_{L_0,L_1}(\lambda)\\
&=&(s_{L_1,L_1}(L_1,\id,\id)\circ_1 D\lambda)\circ_2\beta(G)_\lambda\\
&=& (\beta(s)_{L_1}\circ_1 D\lambda)\circ_2\beta(G)_\lambda
\end{array}
$$
Hence $\beta(s)$ satisfies (B1).
\end{proof}

Since $\cK$ is tensored over $\Cat$, Proposition \ref{3_5} and the properties of the tensor provide isomorphisms
2-natural with respect to strict morphisms $F\to F'$

\begin{leer}\label{3_6}
$$
\begin{array}{rcl}
\Func(\mathcal{L},\cK)(D,F)
& \cong & \Cat^{\mathcal{L}^\op\times\mathcal{L}}((-/\mathcal{L}/-),\cK (D,F))\\
& \cong & \cK^\mathcal{L}((-/\mathcal{L}/-)\otimes_\mathcal{L} D,F)
\end{array}
$$
\end{leer}
We obtain (see \cite[6.7.2]{Bor2})

\begin{prop}\label{3_7}
Let $\cK$ be a cocomplete, tensored 2-category, The there is a $\Cat$-enriched adjunction
$$
\xymatrix{
R:\Func(\mathcal{L},\cK) \ar@<0.5ex>[r] & \cK^\mathcal{L}:i \ar@<0.5ex>[l] 
}
$$
with $R(D)=(-/\mathcal{L}/-)\otimes_\mathcal{L}D$ and $i$ the inclusion functor.
\hfill$\square$
\end{prop}

We are interested in the case $\mathcal{K}=\Cat^{\mathbb{M}}$ where $\mathbb{M}$ is the monad associated with a $\Cat$-operad. 
We next analyze the adjunction in this situation.
In Section 5 we will deal with the case $\mathcal{K}=\Cat^{\mathbb{M}}\Lax$.
Let $B: \Cat \to \Top$ be the classifying space functor, i.e. the composite
of the nerve and the topological realization functor.

\begin{defi}\label{3_8}
A functor $F:\mathcal{A}\to\mathcal{B}$ of small categories is called a \textit{weak equivalence}
 if the induced map of classifying spaces $BF:B\mathcal{A}\to B\mathcal{B}$ is a weak homotopy equivalence.
 A homomorphism of $\mathbb{M}$-algebras $f:X\to Y$ is called a \textit{weak equivalence} if the underlying 
morphism of categories is a weak equivalence.
\end{defi}

We will need the following lemma.

\begin{lem}\label{3_9}
Let $\mathcal{M}$ be a $\Cat$-operad and $\mathbb{M}$ its associated monad. Let $F,G:\mathcal{L}^{\op}\to\Cat$ 
and $X:\mathcal{L}\to\Cat^{\mathbb{M}}$ be diagrams. Let $f,g:F\to G$ be strict maps of diagrams and suppose we are given
 a natural transformation $\tau_L:f(L)\Rightarrow g(L)$ for each $L$ in $\mathcal{L}$ such that 
$$
G(\lambda)\circ_1\tau_{L_0}=\tau_{L_1}\circ_1 F(\lambda)
$$
for each morphism $\lambda:L_0\to L_1$ in $\mathcal{L}$. Then the induced maps of classifying spaces
$$
B(f\otimes_\mathcal{L} \id), B(g\otimes_\mathcal{L} \id): B(F\otimes_\mathcal{L} X)\to B(G\otimes_\mathcal{L} X)
$$
are homotopic as homomorphisms of $B\mathbb{M}$-algebras.
\end{lem}

\begin{proof}
Let [1] denote the category $0\to 1$, and let $i_0,i_1:\ast\to[1]$ be the inclusion of the 0 and the 1 respectively. 
Each $\tau_L$ defines a functor $\tau_L:[1]\times F(L)\to G(L)$ such that $\tau_L\circ(i_0\times\id)=f(L)$ and 
$\tau_L(i_1\times\id)=g(L)$. These functors define a strict morphism $T:[1]\times F\to G$ of $\mathcal{L}^{\op}$-diagrams.
 By \ref{2_7a} we have a natural map
$$
\begin{array}{rcl}
[0,1]\otimes B(F\otimes_\mathcal{L} X)
&=& B([1])\otimes B(F\otimes_\mathcal{L} X)\to B([1]\otimes(F\otimes_\mathcal{L} X))\\
&=& B(([1]\times F)\otimes_\mathcal{L} X)\stackrel{B(T\otimes_\mathcal{L}\id)}{\longrightarrow} B(G\otimes_\mathcal{L} X)
\end{array}
$$
which defines the required homotopy in $\Top^{B\mathbb{M}}$.
\end{proof}

\vspace{1ex}
The unit $\alpha(D):D\to R(D)$ of the adjunction \ref{3_7} is a homotopy morphism and 
the counit $\beta(D):R(D)\to D$ a strict morphism of $\mathcal{L}$-diagrams.

We compare $(-/\mathcal{L}/-)$ with the functor
$$
\mathcal{L}:\mathcal{L}^{\op}\times\mathcal{L}\to\Cat,\quad (L_0,L_1)\mapsto\mathcal{L}(L_0,L_1)
$$
where we consider the set $\mathcal{L}(L_0,L_1)$ as a discrete category. Let
$$
p^{L_0,L_1}:L_0/\mathcal{L}/L_1 \to \mathcal{L}(L_0,L_1)\quad \textrm{ and } \quad
q^{L_0,L_1}:\mathcal{L}(L_0,L_1)\to L_0/\mathcal{L}/L_1
$$
be the functors defined by
$$
p^{L_0,L_1}(L_2,f,g)=g\circ f\quad \textrm{ and }\quad q^{L_0,L_1}(h)=(L_1,h,\id).
$$
Then $p^{L_0,L_1}\circ q^{L_0,L_1}=\id$, and there is a natural transformation
$$
\tau^{L_0,L_1}:\Id_{L_0/\mathcal{L}/L_1}\to q^{L_0,L_1}\circ p^{L_0,L_1}
$$
defined by
$$
\tau^{L_0,L_1}(L_2,f,g)=g:(L_2,f,g)\to (L_1,g\circ f,\id).
$$
The $p^{L_0,L_1}$ induce a morphism of $\mathcal{L}$-diagrams
$$
p:RD\to D
$$
and tracing $\id_D$ through the two natural isomorphisms of \ref{3_6} shows that $p=\beta(D)$.

The $q^{L_0,L_1}$ induce a morphism
$$
q^{L_1}:D(L_1)=\mathcal{L}(-,L_1)\otimes_\mathcal{L} D\to (-/\mathcal{L}/L_1)\otimes_\mathcal{L} D=RD(L_1)
$$
and for $\lambda:L_1\to L_2$ we have a natural transformation 
$\rho_\lambda:RD(\lambda)\circ q^{L_1}\Rightarrow q^{L_2}\circ D(\lambda)$
 induced by the natural transformation

$$
\xymatrix{
\mathcal{L}(L_0,L_1) \ar[rr]^{\lambda_\ast} \ar[dd]_{q^{L_0,L_1}}
&& \mathcal{L}(L_0,L_2) \ar[dd]^{q^{L_0,L_2}} \\
& \stackrel{\rho^{L_0}_\lambda}{\Rightarrow} & \\
L_0/\mathcal{L}/L_1 \ar[rr]^{\lambda_\ast} && L_0/\mathcal{L}/L_2 
}
$$
where
$$
\rho^{L_0}_\lambda(f)=\lambda:(L_1,f,\lambda)\to (L_2,\lambda\circ f,\id)
$$
The $q^L$ together with the $\rho_\lambda$ define a strict morphism
$$
q:(-/\mathcal{L}/-)\to \Cat^{\mathbb{M}}(D,RD)
$$
of $(\mathcal{L}^{\op}\times\mathcal{L})$-diagrams defined by the functors
$$
q(L_0,L_1):(L_0/\mathcal{L}/L_1)\to Cat^{\mathbb{M}}(D(L_0),RD(L_1))
$$
mapping an object $(L_2, f, g)$ to $\xymatrix{DL_0\ar[r]^{Df} & DL_2\ar[r]^{q^{L_2}} & RD(L_2)\ar[r]^{RD(g)} & RD(L_1)}$
 and a morphism $\lambda:(L_2, f, h\circ\lambda)\to (L_3,\lambda\circ f,h)$ to the 2-cell $RD(h)\circ_1\rho_\lambda\circ_1 D(f)$.
 The map of $\mathcal{L}$-diagrams $RD\to RD$ induced by $q$ is the identity, because $q$ takes values in the 
canonical maps $D(L_0)\to RD(L_1)$. Hence $q$ induces the unit $\alpha(D):D\to R(D)$.

If we fix $L_1$ the natural transformations $\tau^{L_0,L_1}:\Id_{L_0/\mathcal{L}/L_1}\to q^{L_0,L_1}\circ p^{L_0,L_1}$ satisfy 
the requirements of Lemma \ref{3_9} so that
$$
q:(-/\mathcal{L}/-)\otimes_\mathcal{L}D\to D
$$
is objectwise a weak equivalence for each diagram $D:\mathcal{L}\to\Cat^{\mathbb{M}}$. 

We summarize

\begin{leer}\label{3_10}
\textbf{Rectification Theorem:} Let $\mathbb{M}$ be the monad associated with a $\Cat$-operad $\mathcal{M}$.
 There is a $\Cat$-enriched adjunction
$$
\xymatrix{
R:\Func(\mathcal{L},\Cat^{\mathbb{M}})\ar@<0.5ex>[r] & (\Cat^{\mathbb{M}})^\mathcal{L}:i \ar@<0.5ex>[l] 
}
$$
with $R(D)=(-/\mathcal{L}/-)\otimes_\mathcal{L} D$ and $i$ the inclusion. The unit $\alpha(D):D\to R(D)$ is 
induced by the functors $q^{L_0,L_1}:\mathcal{L}(L_0,L_1)\to L_0/\mathcal{L}/L_1$, the counit $\beta(D):RD\to D$
 by the functors $p^{L_0,L_1}: L_0/\mathcal{L}/L_1\to\mathcal{L}(L_0,L_1)$. The morphism $\beta(D):R(D)\to D$ 
of diagrams is objectwise a weak equivalence; a homotopy inverse of $\beta(D)(L):RD(L)\to D(L)$ is given by
 $\alpha(D)(L): D(L)\to RD(L)$, which satisfies $\beta(D)(L)\circ \alpha(D)(L)=\id_{D(L)}$.
\hfill$\square$
\end{leer}

\section{The homotopy colimit functor}
We will construct our homotopy colimit functor for diagrams of algebras over a $\Sigma$-free operad in $\Cat$. 
$\Sigma$-freeness is not too much of a constrain because for any operad $\mathcal{M}$ in $\Cat$ there is a $\Sigma$-free operad
$\widetilde{\mathcal{M}}$ and a map of operads $\widetilde{\mathcal{M}}\to \mathcal{M}$ which is a weak equivalence 
of underlying categories.

So let $\mathcal{M}$ be a $\Sigma$-free operad in $\Cat$ and 
$$
X:\mathcal{L}\to \Cat^{\mathbb{M}}\Lax
$$ an $\mathcal{L}$-diagram in $\Cat^{\mathbb{M}}\Lax$.

\textbf{The category of representatives}

The homotopy colimit $\hocolim\ X$ will be the quotient category of a category $\cR_X$, called the \textit{category of representatives}.
\begin{leer}\label{4_1a}
Objects in $\cR_X$ are tuples
$$
(A;(K_i,L_i)_{i=1}^n)=(A;(K_1,L_1),\ldots,(K_n,L_n))
$$
consisting of an object $A$ 
in $\mathcal{M}(n)$, objects  $L_i$ in $\mathcal{L}$, and objects $K_i$ in $X(L_i)$. 

We visualize such an object as a planar rooted tree with one
node decorated by $A$ and $n$ inputs decorated by the $(K_i, L_i)$ in the given order:
$$
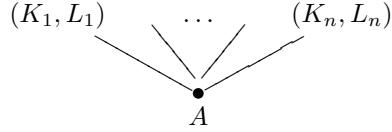
\xymatrix@C=2ex@R=1ex{
(K_1,L_1)  \ar@{-}[rrdd]+D*{\bullet}
&& \ldots &&
(K_n,L_n) \ar@{-}[lldd]+D*{\bullet}
\\
&& \\
&& {}\ar@{-}[uur]{}\ar@{-}[uul]&&
\\
&& A
}
$$
\end{leer}
\begin{leer}\label{4_1}
A morphism $(A_1; (K_i,L_i)_{i=1}^m) \to (A_2; (K'_j,L'_j)_{j=1}^n)$ in $\cR_X$ consists of 
\begin{enumerate}
\item  a map $\varphi:\{1,\ldots,m\}\to\{1,\ldots,n\}$ together with an ordering of $\varphi^{-1}(k)$ for 
each $k\in\{1,\ldots,n\}$. (If $m=0$, then $\varphi:\emptyset\to\{1,\ldots,n\})$,
\item  an $n$-tuple $(C_1,\ldots,C_n)$ of objects $C_k\in\mathcal{M}(|\varphi^{-1}(k)|)$, where $|S|$ denotes the cardinality of a set $S$,
\item  a morphism
$$
\gamma:A_1\longrightarrow A_2\ast(C_1\oplus\ldots\oplus C_n)\cdot\sigma
$$ 
in $\cM(m)$, where $\ast$ is the operad composition and $\sigma^{-1}$ is the underlying permutation of the unique order preserving map
$$
\{1<2<\ldots<m\}\longrightarrow\varphi^{-1}(1)\sqcup\ldots\sqcup\varphi^{-1}(n)
$$
into the ordered disjoint union of the ordered sets $\varphi^{-1}(1),\ldots,\varphi^{-1}(n)$,
\item for each $i=1,\ldots,m$ a morphism $\lambda_i:L_i\to L'_{\varphi(i)}$ in $\mathcal{L}$,
\item for each $k=1,\ldots,n$ a map
$$
f_k:C_k(\lambda_{i_1}K_{i_1},\ldots,\lambda_{i_r} K_{i_r})\longrightarrow K'_k
$$
in $X(L'_k)$ where $\lambda K$ is short for $X(\lambda)(K)$ and where $\{i_1<\ldots<i_r\}$ is the given ordering of $\varphi^{-1}(k)$, and 
$C_k(\lambda_{i_1}K_{i_1},\ldots,\lambda_{i_r}K_{i_r})\in X(L_k')$ is obtained by the action of $C_k\in\mathcal{M}(r)$ on $(\lambda_{i_1}K_{i_1},\ldots,\lambda_{i_r}K_{i_r})$. 
If $\varphi^{-1}(k)=\emptyset$, then $C_k\in\mathcal{M}(0)$ and $f_k$ is a map $C_k(\ast)\to K'_k$ in $X(L'_k)$; here $C_k(\ast)\in X(L'_k)$ is the image of the unique object of the trivial category.
\end{enumerate}
\end{leer}

\begin{leer}\label{4_1b} To define composition and to explain the relations that lead to $\hocolim\ X$ we will use the following pictorial description of a morphism:

$$
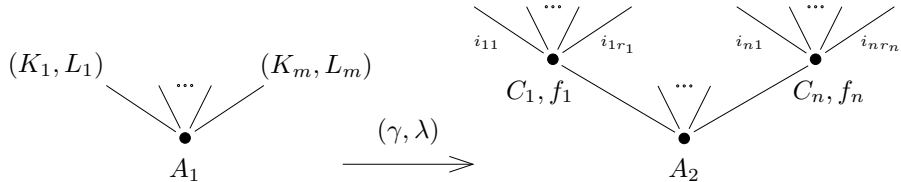
\input{22_1.tex}
$$

The tree on the right has $m$ inputs, and $\{ i_{11},\ldots ,i_{n,r_n}\}$ is a reordering of the set $\{1,\ldots,m\}$.
Here $\lambda$ stand for the collection of the $\lambda_i$.
The input labels of the node $C_k$ give $\varphi^{-1}(k)$ in its order from left to right. The tree on the left specifies the source of the morphism, 
the tree on the right together with the $f_j$'s and $\lambda_i$'s its target.
\end{leer}

\textbf{Composition:} Let
$
(A_1;(K_i,L_i)_{i=1}^m)  \longrightarrow  (A_2;(K'_j,L'_j)_{j=1}^n)$ be given by the picture
$$
\hspace*{5ex}\input{22_3.tex}
$$
and let
$
(A_2;(K'_1,L'_1)_{j=1}^n) \longrightarrow (A_3;(K''_l,L''_l)_{l=1}^p)$
be given by the picture
$$
\hspace*{5ex}\input{22_4.tex}
$$
with $\gamma_1:A_1\to A_2\ast(C_1\oplus\ldots\oplus C_n)\cdot\sigma$ and 
$\gamma_2:A_2\to A_3\ast (D_1\oplus\ldots\oplus D_p)\cdot\tau$. Let $\{1,\ldots,m\}\stackrel{\varphi}{\to}\{1,\ldots,n\}\stackrel{\psi}{\to}\{1,\ldots,p\}$ be the associated maps.

To define the composition we first combine the two trees on the right into one to obtain the ``composition picture''
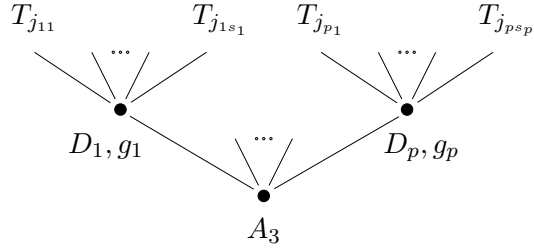
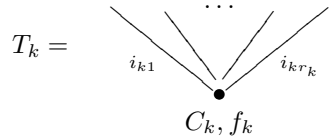
$$
\input{22_5.tex}
$$
where
$$
\begin{array}{cc}
\raisebox{-4ex}{$T_k=$}
&
\xymatrix@C=1ex@R=1ex{
{}\ar@{-}[rrrddd]+D*{\bullet}_{i_{k1}} &&& \ldots &&& {} \ar@{-}[lllddd]+D*{\bullet}^{i_{kr_k}}\\
&&&\\
&&&\\
&&& {}\ar@{-}[uuur]{}\ar@{-}[uuul] &&\\
&&& C_k,f_k
}
\end{array}
$$
The composite is given by the picture
$$
\hspace*{5ex}\input{22_6.tex}
$$
whose data we are now going to specify:
\begin{itemize}
\item $E_k=D_k\ast(C_{j_{k1}}\oplus\ldots\oplus C_{j_{ks_k}})$
\item $l_{k1},\ldots, l_{kt_k}$ are the labels of the inputs of $T_{j_{k1}},\ldots,T_{j_{ks_k}}$ as they appear in the composition picture from left to right.
\item $\arraycolsep0.2ex\begin{array}[t]{rl}
\gamma_3=(\gamma_2\ast\id)\circ\gamma_1:\! A_1\!\to\! 
& A_2\ast(C_1\oplus\ldots\oplus C_n)\cdot\sigma\longrightarrow\\
& A_3\ast((D_1\oplus\ldots\oplus D_p)\cdot\tau)\ast(C_1\oplus\ldots\oplus C_n)\cdot\sigma
\end{array}$
\end{itemize}
We leave it to the reader to check that the target of $\gamma_3$ is $A_3\ast(E_1\oplus\ldots\oplus E_p)\cdot\xi$ where $\xi\in\Sigma_m$ is the 
permutation sending $l_{11},\ldots,l_{pt_p}$ in the given order to $1,\ldots,m$ (see also Example \ref{4_2}) below.
\begin{itemize}
\item $\mu_i=\rho_{\varphi(i)}\circ\lambda_i:L_i \to L''_{\psi\circ\varphi(i)}$
\item $h_k$: The definition of the $h_k$ is a notational challenge. Since we have a hierarchy of indices we denote $j_{kq}$ also by $j(k,q)$. 
It is easy to keep track of the indices by consulting the composition picture. For a subset $\{i_1,\ldots,i_r\}$ of $\{1,\ldots,m\}$ ordered by
$i_1<\ldots <i_r$ we define
$$
\lambda K(\{i_1,\ldots,i_r\})=(\lambda_{i_1}K_{i_1},\ldots,\lambda_{i_r}K_{i_r})
$$
and $\mu K(\{i_1,\ldots,i_r\})$ and $\rho K'(\{j_1,\ldots,j_s\})$ accordingly.
\end{itemize}

For a morphism $\kappa:L\to L'$ in $\mathcal{L}$ we have a lax morphism $X(\kappa):X(L)\to X(L')$ which we denote by $(\kappa,\overline{\kappa}):X(L)\to X(L')$, 
where $\overline{\kappa}$ is the associated 2-cell. We define
$$
h_k:E_k(\mu K((\psi\circ\varphi)^{-1}(k))\longrightarrow K''_k
$$
to be the composite
$$
\xymatrix{E_k(\mu K((\psi\circ\varphi)^{-1}(k)))
=D_k(C_{j(k,1)}(\mu K(\Phi_1),\ldots,C_{j(k,s_k)}(\mu K(\Phi_{s_k}))
\ar[d]^{D_k(\overline{\rho}_{j(k,1)},\ldots,\overline{\rho}_{j(k,s_k)})}
\\
D_k (\rho_{j(k,1)} C_{j_{k,1}}(\lambda K(\Phi_1),\ldots,
\rho_{j(k,s_k)} C_{j(k,s_k)}(\lambda K(\Phi_{s_k}))))
\ar[d]^{D_1(\rho_{j(k,1)}(f_{j(k,1)}),\ldots,\rho_{j(k,s_k)}(f_{j(k,s_k)}))}
\\
D_k(\rho K'(\psi^{-1}(k))) 
\ar[d]^{g_k}
\\
K''_k
}
$$
where $\Phi_t$ is the ordered set $\varphi^{-1}(j_{kt})$. We also make use of the explicit description of $\overline{\rho}_j$ given in \ref{3_3}.

The identity of $(A;(K_1,L_1),\ldots, (K_nL_n))$ is the morphism whose data consists solely of identities.
 
It is tedious but straight forward to prove that this composition is associative.

\begin{exam}\label{4_2}
Morphism 1: 

$(A_1;(K_i,L_i)_{i=1}^7) \to (A_2;(K'_j,L'_j)_{j=1}^4)$ with picture
$$
\hspace*{5ex}\input{Bild1.tex}
$$
The picture incorporates the following data:
\begin{enumerate}
\item A map $\varphi :\{1,\ldots,7\}\to \{1,\ldots,4\}$ given by $\begin{pmatrix}
1 & 2 & 3 & 4 & 5 & 6 & 7\\
2 & 1 & 4 & 1 & 2 & 1 & 4
\end{pmatrix}$. The preimages of $\varphi$ are in natural order.
\item A quadruple of objects $(C_1,\ldots,C_4)$ with $C_1\in \cM(3),\ C_2\in \cM(2),\ C_3\in \cM(0),\ C_4\in \cM(2)$.
 \item A morphism $\gamma_1:A_1\to A_2\ast(C_1\oplus\ldots\oplus C_4)\cdot\sigma$ in $\cM(7)$, where
$\sigma^{-1}=\begin{pmatrix}
1 & 2 & 3 & 4 & 5 & 6 & 7\\
2 & 4 & 6 & 1 & 5 & 3 & 7
\end{pmatrix}$.
\item Morphisms $\lambda_i:L_i\to L_{\phi(i)},\ i=1,\ldots,7$, in $\cL$.
\item Morphisms $f_1:C_1(\lambda_2K_2,\lambda_4K_4, \lambda_6K_6)\to K_1'$ in $X(L_1')$ and $f_2, f_3,f_4$ accordingly, where
$f_3:C_3(\ast)\to K_3'$ is a morphism in $X(L_3')$.
\end{enumerate}

Morphism 2:

$(A_2;(K'_j,L'_j)_{j=1}^4) \to (A_3;(K''_l,L''_l)_{l=1}^3)$ with picture
$$
\hspace*{5ex}\input{Bild2.tex}
$$
where $\gamma_2:A_2\to A_3\ast(D_1\oplus D_2\oplus D_3)\cdot \tau$ \quad with $\tau^{-1}=\begin{pmatrix}
1 & 2 & 3 & 4 \\
2 & 3 & 1 & 4 
\end{pmatrix}$

The composite picture is
$$
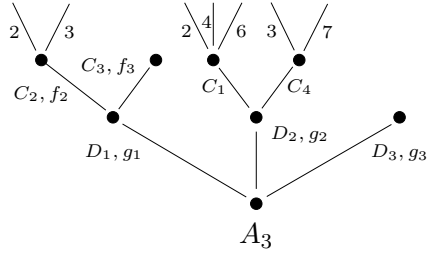
\input{Bild3.tex}
$$

The composite morphism 

 $(A_1;(K_i,L_i)_{i=1}^7) \to (A_3;(K''_l,L''_l)_{l=1}^3)$

has the picture
$$
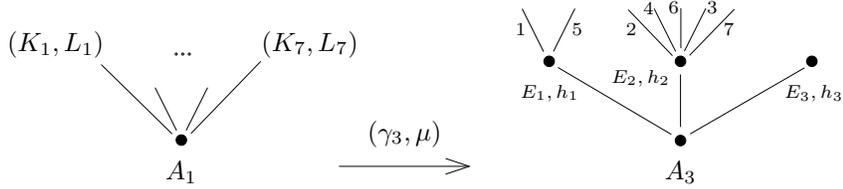
\input{Bild4.tex}
$$

with the following data
\begin{enumerate}
\item $ 
E_1=D_1\ast(C_2\oplus C_3),\quad  E_2=D_2\ast(C_1\oplus C_4), \quad E_3=D_3
$
\item 
$\gamma_3: A_1\stackrel{\gamma_1}{\longrightarrow} 
A_2\ast(C_1\oplus \ldots\oplus C_4)\cdot\sigma \stackrel{\gamma_2\ast\id}{\longrightarrow}
\\
\hspace*{5ex}A_3\ast(D_1\oplus D_2\oplus D_3)\cdot\tau\ast (C_1\oplus \ldots\oplus C_4)\cdot\sigma$
\item 
$$\begin{array}{llll}
\mu_1=\rho_2\circ\lambda_1, &
\mu_5=\rho_2\circ\lambda_5, &
\mu_2=\rho_1\circ\lambda_2, &
\mu_4=\rho_1\circ\lambda_4,
\\
\mu_6=\rho_1\circ\lambda_6, &
\mu_3=\rho_4\circ\lambda_3, & 
\mu_7=\rho_4\circ\lambda_7  &
\end{array}
$$
\item $$\begin{array}{c}
h_1:E_1(\mu_1K_1,\mu_5K_5)=D_1(C_2(\rho_2\lambda_1K_1,\rho_2\lambda_5K_5)),C_3(\rho_3(\ast))) \stackrel{D_1(\overline{\rho}_2,\overline{\rho}_3)}{\longrightarrow}
\\

D_1(\rho_2C_2(\lambda_1K_1,\lambda_5K_5),\rho_3C_3(\ast)) \stackrel{D_1(\rho_2(f_2),\rho_3(f_3))}{\longrightarrow}
D_1(\rho_2K'_2,\rho_3K'_3) \stackrel{g_1}{\longrightarrow} K''_1
\end{array}
$$
$h_2$ is defined analogously while
$
h_3:E_3(\ast)=D_3(\ast)\stackrel{g_3}{\longrightarrow} K''_3.
$
\end{enumerate}
\end{exam}

We check that $\gamma_3$ has the correct target:
\begin{multline*}
(D_1\oplus D_2\oplus D_3)\cdot\tau\ast(C_1\oplus\ldots\oplus C_4)\cdot\sigma
\\
=(D_1\oplus D_2\oplus D_3)\ast (C_2\oplus C_3\oplus C_1 \oplus C_4)\cdot(\overline{\tau}\cdot\sigma)
\end{multline*}
where
$\overline{\tau}=\begin{pmatrix}
1 & 2 & 3 & 4 & 5 & 6 & 7\\
3 & 4 & 5 & 1 & 2 & 6 & 7
\end{pmatrix}$

is the block permutation given by $\tau$. Then $(\overline{\tau}\cdot\sigma)=\sigma_3$.

\begin{leer}\label{4_3}
We define functors 
$$
\cM(n)\times (\cR_X)^n\to \cR_X
$$
on objects by mapping $B\in \cM(n)$ and an n-tuple of objects $(Y_1,\ldots,Y_n)$ in $\cR_X$
with $Y_i=(A_i;(K_{i1},L_{i1}),\ldots,(K_{ir_i},L_{ir_i}))$ to
$$
B(Y_1,\ldots,Y_n)==(B\ast(A_1\oplus\ldots\oplus A_n); (K_{11},L_{11}),\ldots,(K_{nr_n},L_{nr_n})).
$$
We extend this definition to morphisms in the canonical way.

Note that these functors do not define an $\cM$-algebra structure on $\cR_X$
\end{leer}

\begin{defi}\label{4_4}
A morphism with data $(\varphi,(C_1,\ldots,C_n),\gamma,\{\lambda_i\}_{i=1}^m,\{f_k\}_{k=1}^n)$ like in \ref{4_1} 
is called \textit{permutation-free} if $\{1,\ldots,m\}=\varphi^{-1}(1)\sqcup\ldots\sqcup\varphi^{-1}(n)$ as ordered sets. 
In particular, the target of $\gamma$ is $A_2\ast (C_1\oplus\ldots\oplus C_n)$, and in the pictorial description the
inputs to the nodes $C_1,\ldots,C_n$ are labelled $1,\ldots,m$ from left to right.
\end{defi}

\begin{leer}\label{4_5}
(1) The composite of permutation-free morphisms is again permutation-free.

(2) A permutation-free morphism is built from elementary ones of the following types, called \textit{atoms}:
\begin{enumerate}
\item[(1)] $\gamma: (A;(K_1,L_1),\ldots,(K_n,L_n)) \to (B; (K_1,L_1),\ldots,(K_n,L_n))$\\
with $\gamma:A\to B$ in $\mathcal{M}(n)$. The other data of this morphism are identities.
\item[(2)] $\lambda : (\id;(K,L)) \to (\id;(\lambda K,L))$\\
for a map $\lambda :L\to L' $ in $\cL$.
The other data are identities.
\item[(3)] $\ev(C): (C;(K_1,L),\ldots, (K_n,L))\to (\id; (C(K_1,\ldots,K_n),L))$, \\
called \textit{evaluation}, for objects $C$ in $\cM(n)$.
Its picture is
$$
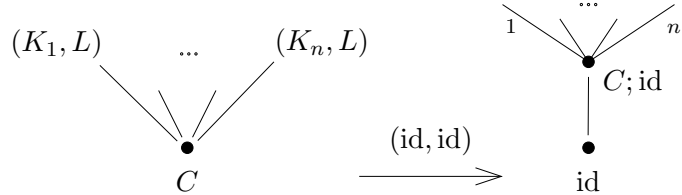
\input{Bild8.tex}
$$
\item[(4)] $f: (\id;(K,L))\to (\id;(K',L))$ for morphisms $f: K\to K'$ in $X(L)$.
\end{enumerate}
E.g. the morphism of Definition \ref{4_4} decomposes into morphisms obtained from atoms,
canonically up to order:
\begin{multline*}
(A_1;(K_i,L_i)_{i=1}^m)
\xrightarrow{\gamma} (B; (K_i,L_i)_{i=1}^m)    
\xrightarrow{B(\lambda_1,\ldots,\lambda_m)} (B;(\lambda_iK_i,L'_{\varphi(i)})_{i=1}^m)\\
\xrightarrow{A_2(\ev(C_1),\ldots,\ev(C_n))} (A_2;(C_j(\lambda K(\varphi^{-1}(j)),L'_j)_{j=1}^n)
\xrightarrow{A_2(f_1,\ldots,f_n)} (A_2;(K'_j,L'_j)_{j=1}^n))
\end{multline*}
where $B=A_2\ast(C_1\oplus\ldots\oplus C_n)$.
In this decomposition we could replace $B(\lambda_1,\ldots,\lambda_m)\circ \gamma$ by
$$ (A_1;(K_i,L_i)_{i=1}^m) \xrightarrow{A_1(\lambda_1,\ldots,\lambda_m)}
(A_1; (\lambda_iK_i,L'_{\varphi(i)})_{i=1}^m)
\xrightarrow{\gamma} (B;(\lambda_iK_i,L'_{\varphi(i)})_{i=1}^m)$$ 
\end{leer}

\textbf{The homotopy colimit construction}
 
 To turn the functors of \ref{4_3} into an action of $\cM$ on our homotopy colimit we have to impose relations on the objects and consequently on the morphisms
 of $\cR_X$. An additional relation is necessary for the expected universal property of the homotopy colimit.

\begin{leer}\label{4_7} \textbf{Relations on objects}
 $$
(A\cdot\sigma; (K_1,L_1),\ldots,(K_n,L_n)) \sim (A;(K_{\sigma^{-1}(1)},L_{\sigma^{-1}(1)})\ldots(K_{\sigma^{-1}(n)},L_{\sigma^{-1}(n)}))
$$
In terms of pictures this is
$$
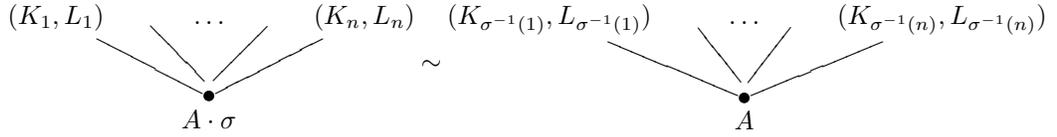
\arraycolsep-0.1ex
\begin{array}{ccc}
\xymatrix@C=2ex@R=1ex{
(K_1,L_1)  \ar@{-}[rrdd]+D*{\bullet}
&& \ldots &&
(K_n,L_n) \ar@{-}[lldd]+D*{\bullet}
\\
&& \\
&& {}\ar@{-}[uur]{}\ar@{-}[uul]&&
\\
&& A\cdot \sigma
}
& \raisebox{-4ex}{$\sim$} &
\xymatrix@C=2ex@R=1ex{
(K_{\sigma^{-1}(1)},L_{\sigma^{-1}(1)})  \ar@{-}[rrdd]+D*{\bullet}
&& \ldots &&
(K_{\sigma^{-1}(n)},L_{\sigma^{-1}(n)})\ar@{-}[lldd]+D*{\bullet}
\\
&& \\
&& {}\ar@{-}[uur]{}\ar@{-}[uul]&&
\\
&& A
}
\end{array}
$$
Let $[A;(K_i,L_i)_{i=1}^n)]$ denote the equivalence class of $(A;(K_i,L_i)_{i=1}^n))$.
\end{leer}
\begin{leer}\label{4_8} \textbf{Relations on morphisms}

A morphism $(A_1;(K_1,L_1),\ldots,(K_m,L_m)\to  (A_2;(K'_1,L'_1),\ldots,(K'_n,L'_n))$ with picture
$$
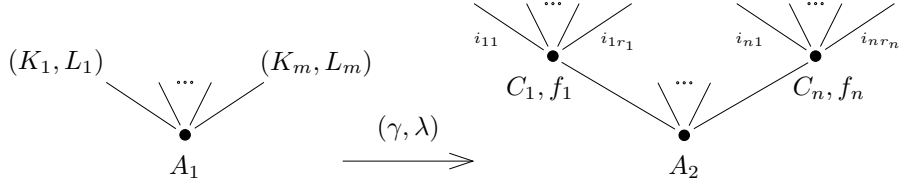
\input{22_1.tex}
$$
is related to a morphism determined by changing the given picture by one of the following moves
\begin{enumerate}
 \item The tree on the right may be changed to 
 $$
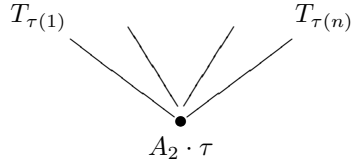
\xymatrix@C=1ex@R=1ex{
T_{\tau(1)} \ar@{-}[rrrddd]+D*{\bullet} &&& &&& T_{\tau(n)} \ar@{-}[lllddd]+D*{\bullet}
\\
&&&\\
&&&\\
&&& {}\ar@{-}[uuur]{}\ar@{-}[uuul]&&
\\
&&& A_2\cdot\tau
}
$$
where $\tau\in\Sigma_n$ and
$$
\begin{array}{cc}
\raisebox{-4ex}{$T_k=$}
&
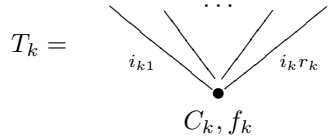
\xymatrix@C=1ex@R=1ex{
{}\ar@{-}[rrrddd]+D*{\bullet}_{i_{k1}} &&& \ldots &&& {} \ar@{-}[lllddd]+D*{\bullet}^{i_kr_k}\\
&&&\\
&&&\\
&&& {}\ar@{-}[uuur]{}\ar@{-}[uuul] &&\\
&&& C_k,f_k
}
\end{array}
$$
The remaining picture is unchanged.
\item A subtree $T_k$ of the tree on the right in the notation above may be replaced by
$$
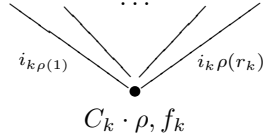
\xymatrix@C=1ex@R=1ex{
{}\ar@{-}[rrrddd]+D*{\bullet}_{i_{k\rho(1)}} &&& \ldots &&& {} \ar@{-}[lllddd]+D*{\bullet}^{i_k\rho(r_k)}\\
&&&\\
&&&\\
&&& {}\ar@{-}[uuur]{}\ar@{-}[uuul] &&&\\
&&& C_k\cdot\rho,f_k
}
$$
where $\rho\in\Sigma_k$. The remaining picture is unchanged.
\item If $\gamma:A_1\to A_2\ast (C_1\oplus\ldots \oplus C_n)\cdot \sigma$ decomposes
$$
\gamma: A_1\xrightarrow{\delta} A_2\ast (D_1\oplus\ldots \oplus D_n)\cdot \sigma \xrightarrow{\id_{A_2}\ast (\alpha_1\oplus\ldots\oplus\alpha_n)\cdot\sigma}
(C_1\oplus\ldots \oplus C_n)\cdot \sigma
$$
then our morphism is related to 
$$
\hspace*{-2ex}\input{22_7.tex}
$$
with $g_k:D_k(\lambda K(\varphi^{-1}(k))\xrightarrow{\alpha_k} C_k(\lambda K(\varphi^{-1}(k))\xrightarrow{f_k}K_k'$.
\item For $\tau\in \Sigma_m$ our morphism is related to
$$
\hspace*{7ex}\input{22_2.tex}
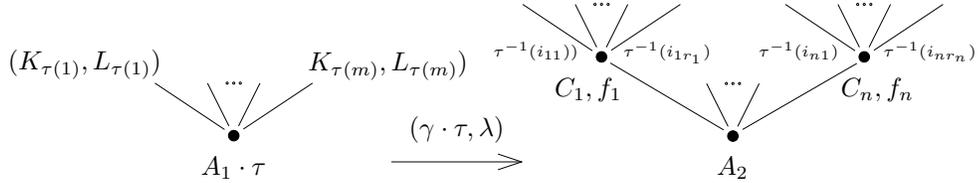
$$
\end{enumerate}
\end{leer}

\begin{defi}\label{4_9}
The \textit{homotopy colimit} $\hocolim\ X$ of the diagram $X$ is the quotient of $\cR_X$ by these relations.
\end{defi}

\textbf{Comments:} (1) Relation 1  changes $\varphi$ and the permutation $\sigma$ by a block permutation $\tau(r_1,\ldots,r_n)$. But since
$$
A_2\ast(C_1\oplus\ldots\oplus C_n)=(A_2\cdot\tau)\ast(C_{\tau(1)}\oplus\ldots\oplus C_{\tau(n)}) \cdot\tau^{-1}(r_1,\ldots,r_n)
$$
an unchanged  $\gamma$ fits.\\
(2) Relation 2 changes $\sigma$ and the order of $\varphi^{-1}(k)$. By the equivariance properties of the 
operad structure and the action of $\mathcal{M}$ on $X(L'_k)$, 
the same $\gamma$ and $f_k$ fit.\\
(3) Our morphism and the morphism of Relation 4 have equivalent sources and the same target.\\
(4) Relations (1), (2) and (4) ensure that the functors of \ref{4_3} give $\hocolim\ X$ an $\cM$-algebra structure.\\
(5) Relation (3) is necessary for the universal property of $\hocolim\ X$ (see below).

\textbf{Composition in $\hocolim\ X$:} Let $U\xrightarrow{u} V\xrightarrow{v} W$ be two morphisms in $\hocolim\ X$.
Let $(A_1;(K_1,L_1),\ldots,(K_m,L_m))\to (A_2;(K_1',L_1'),\ldots,(K_n',L_n')$ and
$(B_1;(\overline{K_1},\overline{L_1}),\ldots,(\overline{K_n},\overline{L_n})\to (K_1'',L_1''),\ldots,(K_p'',L_p''))$
be representatives of $u$ and $v$ respectively. Since $\cM$ is $\Sigma$-free there is a unique permutation $\tau\in\Sigma_n$
such that $(B_1;(\overline{K_1},\overline{L_1})=(A_2\cdot \tau;(K_{\tau(1)},L_{\tau(1)}),\ldots,(K_{\tau(n)},L_{\tau(n)}))$.
By Relation (1) we can change the representative of $u$ such that its target is the source of the chosen representative of $v$,
alternatively, we can change the representative of $v$ by Relation (3) so that we can compose it with the chosen representative of $u$.
The composites in $\cR_X$ of both cases agree, they represent the composite $v\circ u$.

It is now straight forward to show that this composition is well-defined.

\begin{leer}\label{4_10} \textbf{Standard representatives:} (1) It is sometimes helpful to consider the quotient category $\cQ_X$ of 
$\cR_X$ by the relation \ref{4_7} and 
\ref{4_8}(1),(2), and (4). The functors \ref{4_3} define an $\cM$-algebra 
structure on $\cQ_X$ and the projection
 $\cQ_X\to \hocolim\ X$ is a map of $\cM$-algebras.\\
 By Relation \ref{4_8}(4), a morphism in $\cQ_X$ is represented by a permutation-free morphism in $\cR_X$. In particular,
 as an $\cM$-algebra, $\cQ_X$ is generated by the classes of the atoms in $\cR_X$. A permutation-free representative is in general not
 unique: If we apply Relation \ref{4_8}(1) or (2) to it, we obtain a representative which is not permutation-free and can be 
 turned into a new permutation-free representative by applying Relation \ref{4_8}(4).

(2) If we choose a complete set $\orb\cM(n)\subset \ob \cM(N)$ of representatives of the $\Sigma_n$-orbits of 
objects in $\cM(n)$, each object
of $\cQ_X$ has a unique representative of the form $(A;(K_i,L_i)_{i=1}^m)$ with $A\in \orb\cM(m)$, because $\cM$ is $\Sigma$-free. 
Given this choice of representatives for the objects 
there are canonical choices of representatives
$$
(\varphi,(C_1,\ldots,C_n),\gamma,\{\lambda_i\},\{f_j\}):
(A;(K_i,L_i)_{i=1}^m) \to (B;(K'_j,L'_j)_{j=1}^n)
$$ for the morphisms, namely those where the preimages $\varphi^{-1}(j)$ are given the natural order. So $\cQ_X$ is isomorphic
to the category given by these \textit{standard representatives}. The composition and the $\cM$-algebra structure
of this category is transferred from $\cQ_X$. 
Since compositions of standard representatives need not be standard representatives an explicit definition of the composition requires an application
of Relation \ref{4_8}(2) which leads to unpleasant formulas. An explicit description of the transferred $\cM$-structure is even more tedious since it
involves the relations \ref{4_7} and \ref{4_8}.

(3) For later use we observe that Relation \ref{4_8}(3) preserves permutation-freeness.
\end{leer}

\begin{leer}\label{4_11}
\textbf{The universal property:} Let $X:\mathcal{L}\to\Cat^\mathbb{M}\Lax$ be a diagram and let
$$
c(\hocolim\ X):\mathcal{L}\to\Cat^\mathbb{M}\Lax
$$ 
be the constant diagram on $\hocolim\ X$. The \textit{universal homotopy morphism} 
$$j: X\to c(\hocolim\ X)$$ 
in $\Func(\mathcal{L},\Cat^\mathbb{M}\Lax)$ is given by lax morphisms 
$$
j_L=(j_L,\overline{j}_L):X(L)\to \hocolim\ X
$$
and 2-cells $j_\lambda :j_{L_0}\Rightarrow j_{L_1}\circ X\lambda$ for $\lambda: L_0\to L_1$ defined as follows. The functor $j_L: X(L)\to \hocolim\ X$
sends an object $K$ to $[\id;(K,L)]$ and $f:K_1\to K_2$ to the morphism represented by the atom
$f:((\id; (K_1,L))\to (\id;(K_2,L))$.

These functors do not define a strict morphism of $\mM$-algebras: Given $A\in\mathcal{M}(n)$, then $j_L(A(K_1,\ldots,K_n))$ is represented by
$$
\xymatrix{
A(K_1,\ldots,K_2) \ar@{-}[d]-D*{\bullet}
\\
\id
}
$$
while $A(j_L(K_1),\ldots,j_L(K_n))$ is represented by
$$
\xymatrix@C=1ex@R=1ex{
(K_1,L) \ar@{-}[rrdd]+D*{\bullet}
&& &&
(K_n,L)  \ar@{-}[lldd]+D*{\bullet}
\\
&& \\
&& {}\ar@{-}[uur]{}\ar@{-}[uul]&&
\\
&& A
}
$$
The natural transformation
$$
\overline{j}_L(A; K_1,\ldots,K_n): A(j_L(K_1),\ldots,j_L(K_n)) \to j_L(A(K_1,\ldots,K_n))
$$
is represented by the evaluation $\ev(A)$. Then $\overline{j}_L(\id ; K)=\id_{j_L(K)}$. We verify condition \ref{3_3}(1):
For $\alpha:A\to B$ in $\cM(n)$ and $f_i:K_i\to K'_i$ in $X(L)$ we have to show that
\begin{multline*}
 j_L(\alpha(f_1,\ldots,f_n))\circ \overline{j}_L(A:K_1,\ldots,K_n)\\
 =\overline{j}_L(B:K'_1,\ldots,K'_n)\circ \alpha(j_L(f_1),\ldots,j_L(f_n)).
\end{multline*}
A representative of the left side is given by the picture
$$
\input{rect1.tex}
$$
a representative of the right side by the picture
$$
\input{rect2.tex}
$$
By Relation \ref{4_8}(3) these representatives are related.

Condition \ref{3_3}(2) clearly holds: The composite of the canonical representatives of the morphisms on the right
is the canonical representative of the morphism on the left (here the canonical representatives are the ones used
in the definitions of $(j_L,\overline{j}_L)$).

Given a  morphism $\lambda:L_0\to L_1$ in $\mathcal{L}$ we have to define the 2-cell $j_\lambda$
$$
\xymatrix{ 
X(L_0) \ar[rr]^{X\lambda}\ar[ddr]_{j_{L_0}} &\ar @{} [d]|{\stackrel{j_\lambda}{\Rightarrow}}&
X(L_1) \ar[ddl]^{j_{L_1}}
\\
&\\
& \hocolim\ X
}
$$
For $K$ in $X(L_0)$ we get $j_{L_1}\circ X\lambda(K)=[\id;(\lambda K,L_1)]$ and $j_{L_0}(K)=[\id;(K,L_0)]$. 
We define $j_\lambda(K):j_{L_0}(K)\to j_{L_0}\circ X(\lambda)(K)$ to be represented by the atom $\lambda$ in $\cR_X$.

We note that $j_\lambda$ has to be compatible with the natural transformations $\overline{X\lambda}, \overline{j_{L_0}}$, and $\overline{j_{L_1}}$, 
which are part of the homotopy morphisms $(X\lambda,\overline{X\lambda})$, $(j_{L_0},\overline{j_{L_0}})$ and $(j_{L_1},\overline{j_{L_1}})$. 
The verification of this is left to the reader and we refer to Appendix A for more details about the 2-category $\Func(\mathcal{L},\Cat^\mathbb{M}\Lax)$.
\end{leer}

\begin{theo}\label{4_12}
Let $S$ be an $\mM$-algebra and $cS:\mathcal{L}\to\Cat^\mathbb{M}\Lax$ the constant $\mathcal{L}$-diagram on $S$. Let $k:X\to cS$ be 
a morphism in $\Func(\mathcal{L},\Cat^\mathbb{M}\Lax)$. Then there is a unique strict homomorphism
$$
r:\hocolim\ X\to S
$$
of $\mM$-algebras such that $(r,\overline{r})\circ(j_L,\overline{j}_L)=(k_L,\overline{k}_L)$ in $\Cat^\mathbb{M}\Lax$  and $r\circ_1j_\lambda=k_\lambda$, 
where $\overline{r}$ is the identity.
\end{theo}

\begin{proof}
\textit{Uniqueness:} Suppose $r$ exists. Then the conditions on $r$ imply:
\begin{enumerate}
 \item $r([\id;(K,L)])=k_L(K)$.
 \item $r([f])=k_L(f)$ for the atom $f:(\id;(K_1,L))\to (\id;(K_2,L))$.
 \item $([\ev(A)])=\overline{k}_L(A;K_1,\ldots,K_m)$ for the atom 
 $$\ev(A):A((K_1,L),\ldots,(K_n,L))\to (\id;(A(K_1,\ldots,K_n),L).)$$
 \item $r([\lambda])=k_{\lambda}(K)$ for the atom $\lambda:(\id;(K,L_0))\to (\id;(\lambda K,L_1))$.
\end{enumerate}

Since $r$ is a homomorphism of $\cM$-algebras, we obtain from (1) that
$$r([A;(K_i,L_i)_{i=1}^n]=A(k_{L_1}(K_1),\ldots,k_{L_m}(K_n))$$
and that $r$ is determined on atoms $\gamma:(A;(K_i,L_i)_{i=1}^n)\to (B;(K_i,L_i)_{i=1}^n)$.
Since $\hocolim\ X$ is generated as an $\cM$-algebra by the classes of the atoms, $r$ is uniquely
determined by the conditions.

\textit{Existence:} On objects $r$ has to be defined like above. To extend $r$ to morphisms we represent a
morphism by a permutation-free morphism in $\cR_X$ which has a canonical decomposition \ref{4_5}(2). The above
conditions determine $r$ on the classes of the morphisms of this decomposition and hence on the morphism.
We leave it to the reader to check that this definition is independent of the choice of the permutation-free
representative. Since the lax morphisms $(k_L,\overline{k}_L)$ satisfy the conditions \ref{3_3} this definition
factors through Relation \ref{4_8}(3).

To prove functoriality, we observe that a diagram $U\xrightarrow{u} V\xrightarrow{v} W$ can be lifted to a diagram
in $\cR_X$ with permutation-free morphisms, so that it suffices to check functoriality for compositions of atoms, 
which is straight-forward.

By construction $r$ is a strict homomorphism of $\cM$-algebras.
\end{proof}

\begin{leer}\label{4_13} \textbf{Addendum:}
Let $k':X\to cS$ be another morphism in $\Func(\mathcal{L},\Cat^\mathbb{M}\Lax)$ and $s:k\Rightarrow k'$ a 2-cell. 
Let $r':\hocolim X\to S$ be the homomorphism induced by $k'$. Then there is a unique 2-cell $t:r\Rightarrow r'$ in $\Cat^\mathbb{M}\Lax$ such that
$$
t\circ_1 j_L=s_L.
$$
\end{leer}

\begin{proof}
Since $t$ is a 2-cell in $\Cat^\mathbb{M}$ it has to satisfy
$$
\begin{array}{rcl}
t([A;(K_1,L_1),\ldots,(K_n,L_n)])
&=& t(A([\id; (K_1,L_1)],\ldots,[\id;(K_n,L_n)]))\\
&=& A(t(j_{L_1}(K_1)),\ldots,t(j_{L_n}(K_n)))\\
&=& A(s_{L_1}(K_1),\ldots,s_{L_n}(K_n))
\end{array}
$$
Hence $t$ is uniquely determined, and it is easy to check that it is a natural transformation.
\end{proof}

The universal property and its addendum imply

\begin{theo}\label{4_14}
$\hocolim:\Func(\mathcal{L},\Cat^\mathbb{M}\Lax)\to\Cat^\mathbb{M}$ is a 2-functor which is left adjoint to the constant 
diagram functor.\hfill\ensuremath{\Box}
\end{theo}

\begin{leer}\label{4_15}
 Let $c:\Cat^\mathbb{M}\to \Func(\mathcal{L},\Cat^\mathbb{M}\Lax)$ be the constant diagram functor.
 The unit of the adjunction \ref{4_14} is given 
by the universal map $j: X\to c(\hocolim\ X)$. The counit $r: \hocolim\ cS \to S$ is obtained from 
Theorem \ref{4_12} by taking $k=id_{cS}$. It is
explicitly given by the evaluation, which sends an object $[A;(K_i,L_i)_{i=1}^n]$ to $A(K_1,\ldots,K_n)$
 and which is the identity on the 
classes of the atoms \ref{4_5}.2.2 and \ref{4_5}.2.3, while classes of the atoms
 $\gamma :A\to B$ and $f:[id;(K,L)]\to [id:(K',L)]$ are sent to $\gamma :A(K_1,\ldots, K_n)\to
B(K_1,\ldots, K_n)$, respectively $f:K\to K'$.
\end{leer}

\begin{leer}\label{4_16}
\textbf{Change of indexing category:} Let $X:\cL \to \Cat^\mathbb{M}\Lax$ be a diagram and let $F:\cN\to \cL$ be a functor of small
categories. Then $F$ induces a functor
$$
F_\ast : \hocolim \ X\circ F \to  \hocolim \ X
$$
given on objects by
$$
F_\ast ([A;(K_1,N_1),\ldots, (K_n,N_n)] = [A;(K_1,F(N_1)),\ldots,(K_n,F(N_n))]
$$
and on atoms by 
$$
F_\ast (\mu) = F(\mu) \qquad \textrm{for } \mu:N\to N' \textrm{ in } \cN
$$
while all the other atoms of $\hocolim \ X\circ F$ are mapped identically to the corresponding ones in
$\hocolim \ X$.
\end{leer}

The construction implies:
\begin{prop}\label{4_17}
Let $\xymatrix{\cP\ar[r]^G & \cN \ar[r]^F & \cL}$ be functors of small categories. Then
$$
(F\circ G)_\ast = F_\ast\circ G_\ast: \hocolim \ X\circ F\circ G \to  \hocolim \ X
$$
\hfill\ensuremath{\Box}
\end{prop}

\begin{leer}\label{4_18}
Suppose we are given two functors $F,H: \cN\to \cL$ and a natural transformation $\tau : F\Rightarrow H$, then
$\tau$ induces a natural transformation
$$
\tau_\ast: F_\ast \Rightarrow H_\ast
$$
defined by
$$
\tau_\ast ([A;(K_1,N_1),\ldots, (K_n,N_n)]) = A(\tau(N_1),\ldots,\tau(N_n))
$$
with the atoms $\tau(N_i): [\id;(K_i,F(N_i)]\to [\id;(K_i,G(N_i)]$ corresponding to the morphism
$\tau(N_i): F(N_i) \to  G(N_i)$ in $\cL$.
\end{leer}

\begin{leer}\label{4_19}
\textbf{Comparison with Thomason's homotopy colimit constructions:} 
In \cite{Thom1} Thomason constructed a homotopy colimit functor in the category of non-unital permutative
categories. Permutative categories are algebras over the $\Cat$-version $\widetilde{\Sigma}$ of the Barratt-Eccles
operad, whose definition we will recall in Section 8. Let $\widehat{\Sigma}\subset \widetilde{\Sigma}$
be the suboperad obtained from $\widetilde{\Sigma}$ by replacing $\widetilde{\Sigma}(0)$ by the empty category. Then
the non-unitary permutative categories are precisely the $\widehat{\Sigma}$-algebras.

Thomason proved the universal property for his homotopy colimit \cite[Prop. 3.21]{Thom1}. So our construction
 and his agree for $\widehat{\Sigma}$-algebras.
This can also be seen directly from the constructions as we will explain in Section 8.

Thomason introduced the category $\Func(\mathcal{L},\Cat^\mathbb{M}\Lax)$ later in \cite{Thom2} and established the adjunction
\ref{4_14} in \cite[(1.8.2)]{Thom2}.

Now let $\cM$ be the initial operad in $\Cat$, i.e. $\cM (1)=\{ \id\}$ and $\cM (n)=\emptyset $ for $n\neq 1$. Then
$\Cat^{\mM}=\Cat^{\mM}\Lax=\Cat$ and our homotopy colimit of a diagram $X:\cL \to \Cat^{\mM}\Lax=\Cat$ 
coincides with the Grothendieck construction $\cL\int X$, which is Thomason's homotopy colimit of $X$ in $\Cat$.

The universal property of $\cL\int X$ \cite[1.3.1]{Thom0} is a special case of our Theorem \ref{4_12}. As mentioned above Thomason
did not introduce the category $\Func(\mathcal{L},\Cat)$ of strict functors $X:\cL\to \Cat$
and homotopy morphisms in the early paper \cite{Thom0}, so there is no analogue of Theorem \ref{4_14}; instead he considered lax functors
$X:\cL\to \Cat$ and homotopy morphisms of such in \cite[Section 3]{Thom0}.
\end{leer}

\section{Strictification and consequences}
If $\mathcal{L}$ is the trivial category consisting of one object and its identity morphism,
 then $\Func(\mathcal{L},\Cat^\mathbb{M}\Lax)=\Cat^\mathbb{M}\Lax$, and Theorem \ref{4_14} translates to

\begin{theo}\label{5_1}
There is a strictification 2-functor
$$
\str:\Cat^\mathbb{M}\Lax\to\Cat^\mathbb{M}
$$
which is left adjoint to the inclusion functor.
\hfill$\square$
\end{theo}

The unit of the adjunction is given by the lax morphism $j=(j,\overline{j}):X\to \str X$ of Theorem \ref{4_12}. Applying the universal property to
$\id_X$ we obtain a strict homomorphism
$$
r:\str X\to X
$$
such that $r\circ j=\id$. It follows that $r$ is the counit of the adjunction. In particular, both $(j,\overline{j})$ and $r$ are
natural in $X$.

\begin{leer}\label{5_2}
\textbf{Explicit description of $r$:} (See \ref{4_15})

Objects: $r$ maps $[A; (K_i)_{i=1}^n]$ to $A(K_1,\ldots, K_n)$

Atoms: $\gamma:[A;(K_i)_{i=1}^n]\to [B;(K_i)_{i=1}^n]$ is mapped to $\gamma(K_1,\ldots,K_n)$, the evaluation
 $\ev(C)$ is mapped to the identity, and $f:[\id,K]\to[\id; K']$ is mapped to $f:K\to K'$.
\end{leer}

The composite $j\circ r:\str X\to \str X$ maps $[A; (K_i)_{i=1}^n]$ to $[\id; A(K_1,\ldots,K_n)]$
 and we have a natural transformation
$$
s:\id_\str \Rightarrow j\circ r
$$
defined by
$$
s([A;(K_i)_{i=1}^n])=\ev(A):[A;(K_i)_{i=1}^n]\to [\id; A(K_1,\ldots,K_n)].
$$
It follows

\begin{prop}\label{5_3}
$r$ is a weak equivalence with section $j$.

\hfill$\square$
\end{prop}
If $X$ is a free $\mM$-algebra we can do better.

\begin{prop}\label{5_4}
  Let $F_C: \Cat\to \Cat^{\mM}$ be the free algebra
functor \ref{2_2a} and let $X=F_CY$. Then the natural homomorphism 
$r:\str X\to X$ has a section $k:X\to \str X$ in $\Cat^{\mM}$, natural in $Y$,
and there is a natural transformation $\tau: k\circ r\Rightarrow \id_\str$.
\end{prop}
\begin{proof}
 On objects we define $k(A;y_1,\ldots,y_n)=[A;(\id,y_1),\ldots,(\id,y_n)]$ and extend this
definition to morphisms in the canonical way. Then $k: F_CY\to \str F_CY$ is a homomorphism
such that $r\circ k=\id_{F_CY}$. Let $(B_i,\underline{y}_i)$ with $B_i\in \cM(r_i)$ and $\underline{y}_i=(y_{i1},\ldots,y_{ir_i})$
be an object in $F_CY$ and $A$ an object in $\cM(n)$.
Then
$$
k\circ r([A;(B_1,\underline{y}_1),\ldots ,(B_n,\underline{y}_n)])\ =\ [A\circ (B_1\oplus\ldots\oplus B_n);(\id,y_{11}),\ldots,(\id,y_{n,r_n})]
$$
 and the natural transformation $\tau([A;(B_1,\underline{y}_1),\ldots ,(B_n,\underline{y}_n)])$ is the map
$$ [A\circ (B_1\oplus\ldots\oplus B_n);(\id,y_{11}),\ldots,(\id,y_{n,r_n})\to
[A;(B_1,\underline{y}_1),\ldots ,(B_n,\underline{y}_n)]
$$
obtained by applying $A$ to the atoms $\ev(B_i):[B_i;(\id,y_{i1}),\ldots,(\id,y_{i,r_i})]\to [\id;(B_i;y_{i1},\ldots,y_{i,r_i})]$.
\end{proof}
\vspace{1ex}

$\Cat^\mathbb{M}$ is $\Cat$-enriched, tensored and cotensored \ref{2_5}.
 We now make use of this additional structure.

\begin{leer}\label{5_4a}
Let $X:\mathcal{L}\to\Cat^\mathbb{M}\Lax$ be an $\mathcal{L}$-diagram in $\Cat^\mathbb{M}\Lax$
and $Y:\mathcal{L}\to\Cat^\mathbb{M}$ an $\mathcal{L}$-diagram in $\Cat^\mathbb{M}$. Let
$i: \Cat^\mathbb{M}\to \Cat^\mathbb{M}\Lax$ be the inclusion functor.
 Then Proposition \ref{3_5}, Theorem \ref{5_1}, and Theorem \ref{3_10} provide the following chain of isomorphisms of categories
$$
\begin{array}{rcl}
\Func(\mathcal{L},\Cat^\mathbb{M}\Lax)(X,iY) & \cong & 
 \Cat^{\cL\times\mathcal{L}^{\op}}((-/\mathcal{L}/-),\Cat^\mathbb{M}\Lax(X,iY)) \\
&\cong& \Cat^{\cL\times\mathcal{L}^{\op}}((-/\mathcal{L}/-),\Cat^\mathbb{M}(\str X,Y))\\
&\cong& (\Cat^\mathbb{M})^{\cL}((-/\mathcal{L}/-)\otimes_\mathcal{L} \str X, Y)\\
& = & (\Cat^\mathbb{M})^{\cL}(R(\str X),Y)
\end{array}
$$
and, if $Y$ is the constant diagram $cS$ on $S\in \Cat^\mathbb{M}$, Theorem \ref{4_14} gives
$$
\begin{array}{rcl}
\Cat^\mathbb{M}(\hocolim\ X,S) &\cong& \Func(\mathcal{L},\Cat^\mathbb{M}\Lax)(X,cS)\\
&\cong & (\Cat^\mathbb{M})^{\cL}(R(\str X),cS)\\
&\cong & \Cat^\mathbb{M}(\colim R(\str X), S)\\
&\cong & \Cat^\mathbb{M}((-/\mathcal{L})\otimes_\mathcal{L} \str X, S).
\end{array}
$$
The last isomorphism is induced by the isomorphisms
$$
\begin{array}{rcl}
\colim R(Z) &=& \colim ((-/\mathcal{L}/-)\otimes_\mathcal{L}Z)=\ast\otimes_\mathcal{L}((-/\mathcal{L}/-)\otimes_\mathcal{L}Z)\\
&\cong& (\ast \times_\mathcal{L}(-/\mathcal{L}/-))\otimes_\mathcal{L}Z\cong (-/\mathcal{L})\otimes_\mathcal{L}Z
\end{array}
$$
for a diagram $Z:\cL\to \Cat^\mathbb{M}$.
\end{leer}

We obtain

\begin{theo}\label{5_5} The three functors 
$\hocolim$ and $(-/\mathcal{L})\otimes_\mathcal{L}\str$ and $\colim R\str$
$$ \Func(\mathcal{L},\Cat^\mathbb{M}\Lax)
\to \Cat^\mathbb{M}$$
sending $X$ to $\hocolim\ X$, to $(-/\mathcal{L})\otimes_\mathcal{L}\str X$, and to 
$\colim R\str X$ respectively are 
naturally isomorphic.
\end{theo}

If $\cM$ is the initial operad in $\Cat$ the strictification functor
$$\str : \Cat^\mathbb{M}\Lax= \Cat^\mathbb{M}=\Cat \to \Cat^\mathbb{M}=\Cat$$
is the identity functor. Since the tensor in $\Cat$ is just the product, Theorem \ref{5_5} and \ref{4_19} imply

\begin{prop}\label{5_6}
 Thomason's homotopy colimit functor $\cL\int -$ in $\Cat$ and the functor $(-/\mathcal{L})\times_\mathcal{L} -$ are
naturally isomorphic functors from the category $\Func(\mathcal{L},\Cat)$ of functors $X:\cL\to \Cat$
and homotopy morphisms to the category $\Cat$.
\end{prop}

\begin{prop}\label{5_7}
 The free functor $F_C:\Cat \to \Cat^{\mM}$ preserves homotopy colimits up to weak equivalences.
\end{prop}

\begin{proof}
 Let $X: \cL\to \Cat$ be a diagram in $\Cat$. By \ref{5_5} and \ref{5_6} we have to show that the canonical map
$$(-/\cL)\otimes_{\cL} \str F_CX \to F_C((-/\mathcal{L})\times_\mathcal{L}X)$$
is a weak equivalence. Since $F_C$ is a left adjoint $\Cat$-functor it preserves coends and tensors, so that
$F_C((-/\mathcal{L})\times_\mathcal{L}X)\cong (-/\cL)\otimes_{\cL} F_CX $ and we are left to show that
the homomorphism $r:\str F_CX\to F_CX$ induces a weak equivalence
$$(-/\cL)\otimes_{\cL} \str F_CX \to (-/\cL)\otimes_{\cL} F_CX. $$
But this follows from \ref{5_4}.
\end{proof}

\begin{rema}\label{5_8} We have seen that Thomason's homotopy colimit of a functor 
$X:\mathcal{L}\to\Cat$ can be identified with $(-/\mathcal{L})\times_\mathcal{L}X$.
In comparison to this, the Bousfield-Kan homotopy colimit $\hocolim\ Z$ of a diagram  
$Z:\mathcal{L}\to\SSets$ is $N( -/\mathcal{L})\otimes_\mathcal{L}Z$ 
where $N$ is the nerve functor. Both $\Cat$ and $\SSets$ are self-enriched, and the tensor is given
 by the product. So both constructions are tensoring constructions involving $(-/\mathcal{L})$
respectively the nerve of $(-/\mathcal{L})$.
Theorem \ref{5_5} shows that our homotopy colimit construction fits into this set-up, 
the only difference being the objectwise strictification of our diagrams. The strictification
serves as a substitute of the cofibrant replacement functor which enters into the definition of
the homotopy colimit in the categories of simplicial or topological algebras over simplicial respectively
topological operads.
\end{rema}

\begin{prop}\label{5_9}
 Let $S$ be an $\mM$-algebra and $cS:\cL\to \Cat^{\mM}\Lax$ the constant diagram on $S$. Then in $\Cat^{\mM}$
$$\hocolim\ cS\cong \cL \otimes \str S,$$
and the counit of the adjunction \ref{4_14} corresponds to the composite
$$\cL \otimes \str S\to \ast \otimes \str S\cong \str S\to S.$$
\end{prop}
\begin{proof} Let $F:\cL^{\op}\to \Cat$ be defined by $F(L)=L/\cL$. Then
 $$\hocolim\ cS\cong (-/\cL)\otimes_{\cL} \str S\cong \colim F\otimes \str S\cong \cL\otimes \str S.$$
The second statement holds because the counit factors through the change of indexing category functor $F_\ast$ for
$F: \cL \to \ast$ (see \ref{4_16}).
\end{proof}

\section{On the homotopy type of the homotopy colimit}

In this section we compare our homotopy colimit with the homotopy colimit constructions of algebras in the categories $\SSets$  
and $\Top$. Let 
$$
N:\Cat\to\SSets, \qquad   |-|:\SSets \to \Top, \qquad \Sing:\Top\to\SSets
$$
be the nerve, the topological realization functor, and the singular functor respectively, let $B=|N|$ be the classifying space functor.
We recall

\begin{defi}\label{6_1}
 A map of topological spaces is called a \textit{weak equivalence} if it is a weak homotopy equivalence. A functor 
$F:\cA\to \cB$ of categories and a map $f:K\to L$ of simplicial sets is called a \textit{weak equivalence} if $BF: B\cA\to B\cB$ 
respectively $|f|:|K|\to |L|$ is a weak equivalence.
\end{defi}

The categories $\SSets$ and $\Top$ are cofibrantly generated simplicial model categories with these weak equivalences; in particular, they are simplicially
enriched (\ref{2_7}). The realization and the singular functor are simplicial functors and we have a $\SSets$-enriched Quillen
equivalence \cite{Quillen}
$$
|-|:\SSets\leftrightarrows\Top\ :\Sing.
$$
Let $\mathcal{M}$ be a $\Sigma$-free operad in $\Cat$. Then $N\mathcal{M}$ is 
a simplicial and $B\mathcal{M}$ a topological operad. Let $\SSets^{N\mathbb{M}}$ and $\Top^{B\mathbb{M}}$ denote the 
categories of $N\mM$-algebras in $\SSets$ respectively $B\mM$-algebras in $\Top$, where $N\mM$ and $B\mM$ are the
monads associated with the operads $N\mathcal{M}$ and $B\mathcal{M}$ respectively.
It is well-known that the Quillen model structures of $\SSets$ and $\Top$ lift to cofibrantly generated simplicial Quillen model structure on
 $\SSets^{N\mathbb{M}}$ and $\Top^{B\mathbb{M}}$, whose weak 
equivalences are those homomorphisms of algebras whose underlying maps are weak equivalences in $\SSets$
respectively $\Top$ (e.g. see \cite{SV}; for the simplicial enrichment and the simplicial tensor recall \ref{2_7}).
Since the singular functor $\Sing:\Top\to\SSets$ and the realization functor preserve products, the classical Quillen equivalence
$
|-|:\SSets\leftrightarrows\Top\ :\Sing
$
lifts to a $\SSets$-enriched Quillen equivalence
\begin{leer}\label{6_1a}
$$
|-|^{\alg}:\SSets^{N\mathbb{M}}\leftrightarrows\Top^{B\mathbb{M}}\ :\Sing^{\alg}.
$$
\end{leer}

\begin{defi}\label{6_2}
Let $X:\mathcal{L}\to\SSets^{N\mathbb{M}}$ be an $\mathcal{L}$-diagram in $\SSets^{N\mathbb{M}}$, then its homotopy 
colimit $\hocolim^{N\mathbb{M}}X$ is defined by
$$
\hocolim^{N\mathbb{M}}X=N(-/\mathcal{L})\otimes_\mathcal{L}QX  
$$
where $QX\to X$ is a fixed functorial objectwise cofibrant replacement.

If $X:\mathcal{L}\to\Top^{B\mathbb{M}}$ is an $\mathcal{L}$-diagram in $\Top^{B\mM}$ its homotopy colimit is defined by
$$
\hocolim^{B\mathbb{M}}X=B(\ast,\mathcal{L},\mathcal{L})\otimes_\mathcal{L}Q_tX
$$
where $Q_tX\to X$ again is a fixed functorial objectwise cofibrant replacement and $B(\ast,\mathcal{L},\mathcal{L})$ is the
2-sided bar construction (see \cite{HV}).
\end{defi}

From \cite[18.5.3]{Hirsch} we obtain

\begin{lem}\label{6_3}
Let $f:X\to Y$ be a strict morphism of $\cL$-diagrams $X,Y: \mathcal{L}\to\SSets^{N\mathbb{M}}$ such that
 $f(L):X(L)\to Y(L)$ is a weak equivalence for each object $L$ in $\cL$. Then
 $$
 \hocolim^{N\mathbb{M}}f: \hocolim^{N\mathbb{M}}X \to \hocolim^{N\mathbb{M}}Y
 $$
 is a weak equivalence.
\end{lem}

\begin{rema}\label{6_4}
Hirschhorn defines the homotopy colimit as $N(-/\mathcal{L})\otimes_\mathcal{L}X$ without the cofibrant replacement 
\cite[18.1.2]{Hirsch}. Our definition has the advantage that Lemma \ref{6_3} holds.
Up to weak equivalences our definition is independent of the choice of $Q$ respectively
$Q_t$. For let $Q'$ be another objectwise cofibrant replacement functor, the morphisms $Q'X\leftarrow Q'QX\rightarrow QX$ 
are objectwise weak equivalences between objectwise cofibrant diagrams and hence induce weak equivalences
$$
N(-/\mathcal{L})\otimes_\mathcal{L}Q'X\leftarrow N(-/\mathcal{L})\otimes_\mathcal{L}Q'QX\rightarrow
N(-/\mathcal{L})\otimes_\mathcal{L}QX.
$$
by \ref{6_3}.
\end{rema}

\begin{prop}\label{6_5}
 For a diagram $X:\mathcal{L}\to\SSets^{N\mathbb{M}}$ there are
natural weak equivalences
$$
\hocolim^{B\mM}|X|\leftarrow \hocolim^{B\mM}|QX|\rightarrow |\hocolim^{N\mM} X|
$$
\end{prop}
\begin{proof} Since $|QX|\to |X|$ is an objectwise weak equivalence the left map is a weak equivalence by
homotopy invariance.\\
Note that $B(\ast,\mathcal{L},\mathcal{L})=B(-/\mathcal{L})$. Since realization is a simplicial left Quillen
functor it preserves coends, tensors, and maps cofibrant objects to cofibrant objects. So
we have a natural isomorphism
$$
|N(-/\mathcal{L})\otimes_\mathcal{L}QX|\cong N(-/\mathcal{L})\otimes_\mathcal{L}^s|QX|=|N(-/\mathcal{L})|\otimes_\mathcal{L}|QX|
$$
and each $|QX(L)|$ is cofibrant. Since $Q_t|QX|\to |QX|$ is an objectwise weak equivalence
$$
\hocolim^{B\mM}|QX|=|N(-/\mathcal{L})|\otimes_\mathcal{L}Q_t|QX|\to |N(-/\mathcal{L})|\otimes_\mathcal{L}|QX|
$$
is a weak equivalence by \cite[18.5.3]{Hirsch}.
\end{proof}

Now let 
$$
X:\mathcal{L}\to\Cat^\mathbb{M}\Lax
$$
be a $\mathcal{L}$-diagram in $\Cat^\mathbb{M}\Lax$. Composing with the strictification and the nerve functor we have
$$
\mathcal{L}\stackrel{X}{\longrightarrow}\Cat^\mathbb{M}\Lax\stackrel{\str}{\longrightarrow}
\Cat^\mathbb{M} \stackrel{N}{\longrightarrow}\SSets^{N\mathbb{M}}
$$
By definition of the tensor there is a map
$$
\xymatrix@R=1pt{
\alpha (X):N(-/\mathcal{L})\otimes_\mathcal{L}QN\str X \ar[r] &
N(-/\mathcal{L})\otimes_\mathcal{L}N\str X \ar[r] &
N((-/\mathcal{L})\otimes_\mathcal{L}\str X)
\\
\parallel && \parallel\\
\hocolim^{N\mathbb{M}}N\str X && 
N(\hocolim^\mathbb{M} X)
}
$$
which is natural with respect to strict morphisms of $\cL$-diagrams in $\Cat^{\mM}\Lax$.

\begin{defi}\label{6_6}
 We say that a $\Sigma$-free operad $\cM$ in $\Cat$ \textit{has the homotopy colimit property} if for each 
 $\mathcal{L}$-diagram $X:\mathcal{L}\to\Cat^\mathbb{M}\Lax$ the map $\alpha(X)$ is a weak
 equivalence.
\end{defi}

Our proof that certain operads $\cM$ in $\Cat$ have the homotopy colimit property relies on a good choice of representing objects
and morphisms in $\cR_X$. Here Relation \ref{4_8}.3 causes trouble, and only additional assumptions on $\cM$ help us to resolve these. 

For objects $A\in \cM(m)$, $B\in \cM(n)$, and $r_1,\ldots ,r_n$ in $\mathbb{N}$ such that $r_1+\ldots +r_n=m$ we
define a \textit{factorization category} 
$\cC(A,B,r_1,\ldots ,r_n)$: its objects are tuples $(C_1,\ldots,C_n;\alpha)$ consisting of objects
$C_i\in \cM(r_i)$ and a morphism $\alpha:A\to B\ast (C_1\oplus \ldots \oplus C_n)$ in $\cM(m)$. A morphism
$$
\gamma:(C_1,\ldots,C_n;\alpha)\to (D_1,\ldots,D_n;\beta)
$$
is a tuple $\gamma=(\gamma_1,\ldots ,\gamma_n)$ with $\gamma_i:C_i\to D_i$ in $\cM(r_i)$ such that
$$
\xymatrix{
&& A \ar[dll]_\alpha \ar[drr]^\beta &&\\
B\ast (C_1,\ldots,C_n)\ar[rrrr]^{\id\ast(\gamma_1\oplus \ldots \oplus\gamma_n)} &&&& B\ast (D_1,\ldots,D_n)
}
$$
commutes.

\begin{defi}\label{6_7}
We say that a $\Sigma$-free operad $\cM$ in $\Cat$ \textit{satisfies the factorization condition} if 
each connected component of the category $\cC(A,B,r_1,\ldots ,r_n)$ has an initial object.
\end{defi}

\begin{exam}\label{6_8}
 Let $(\cK(n);n\in\mathbb{N})$ be a collection of categories such that $\Sigma_n$ operates freely on
 $\cK(n)$ from the right. Then the free operad $\mathcal{FK}$ on this collection satisfies the
 factorization condition. We leave the proof as an exercise.
\end{exam}

\begin{theo}\label{6_9}
Each $\Sigma$-free operad $\cM$ in $\Cat$ which satisfies the factorization condition has the homotopy colimit property,
i.e.
for each diagram $X:\mathcal{L}\to\Cat^\mathbb{M}\Lax$ the map
$$
\alpha(X):\hocolim^{N\mathbb{M}}N\str X\longrightarrow N(\hocolim^\mathbb{M} X)
$$
is a weak equivalence, natural with respect to strict morphisms of $\cL$-diagrams in $\Cat^{\mM}\Lax$.
\end{theo}

We prove the theorem in steps. All steps exept of one hold for any $\Sigma$-free operad in $\Cat$. So unless
stated otherwise $\cM$ denotes an arbitrary $\Sigma$-free operad. 
We first reduce the theorem to the case that $X$ is a diagram in $Cat^{\mM}$.

\begin{lem}\label{6_10}
 It suffices to prove the theorem for diagrams in $Cat^{\mM}$.
\end{lem}

\begin{proof}
 Let $X:\mathcal{L}\to\Cat^\mathbb{M}\Lax$ be a diagram in $\Cat^\mathbb{M}\Lax$. For each $\mM$-algebra $Z$ the natural
lax morphism $j_Z: Z\to \str Z$ and the natural homomorphism $r_Z: \str Z\to Z$ induce morphisms $j:X\to \str X$ and
$r:\str X\to X$ of diagrams in $\Cat^\mathbb{M}\Lax$ such that $r\circ j=\id$. Hence
$$ \alpha(X): \hocolim^{N\mathbb{M}}N\str X\longrightarrow N(\hocolim^\mathbb{M} X)$$
is a retract of
$$ \alpha(\str X): \hocolim^{N\mathbb{M}}N\str (\str X)\longrightarrow N(\hocolim^\mathbb{M} \str X).$$
If the latter is a weak equivalence so is the former, because retracts of weak equivalences are weak equivalences.
But $\str X$ is a diagram in $Cat^{\mM}$.
\end{proof}

\vspace{1ex}
We now deal with the free case. Consider the diagram
$$
\xymatrix{
\Cat^\mathbb{M} \ar[rr]^{N^\alg} \ar[d]^{U_C} &&
\SSets^{N\mathbb{M}} \ar@<1ex>[d]^{U_S} 
\\
\Cat \ar[rr]^{N} \ar@<1ex>[u]^{F_C} &&
\SSets \ar[u]^{F_S}
}
$$
where $U_C$ and $U_S$ are the forgetful and $F_C$ and $F_S$ the free algebra functors. Clearly, $U_S\circ N^\alg=N\circ U_C$, and since
 $\mathcal{M}$ is $\Sigma$-free we have

\begin{leer}\label{6_11}
$N^\alg\circ F_C=F_S\circ N.$
\end{leer}
The functors $U_C$ and $F_C$ are $\Cat$-enriched, while $U_S$ and $F_S$ are simplicial functors. Since $F_C$ and $F_S$ are
enriched left adjoints they preserve tensors and coends.

Since $N^\alg$ is just $N$ applied to some algebra we usually drop $\alg$ from the notation.
  By \ref{5_6},\ $-/\mathcal{L}\times_\mathcal{L}Z$ is isomorphic to Thomason's homotopy colimit of the diagram $Z: \cL \to\Cat$, 
  and the natural map
 $$
N(-/\mathcal{L})\times_\mathcal{L}NZ\stackrel{\sim}{\longrightarrow}N(-/\mathcal{L}\times_\mathcal{L}Z)
$$ 
corresponds to Thomason's map 
$$\eta:\hocolim\ NX\to N(\cL\int X)= N(\hocolim\ X).
$$
By \cite[Thm. 1.2]{Thom0} this map is a weak equivalence.
Since $\mathcal{M}$ is $\Sigma$-free the free functor $F_S$ preserves weak equivalences. 
So we obtain a weak equivalence
$$
F_S(N(-/\mathcal{L})\times_\mathcal{L}NZ)\stackrel{\sim}{\longrightarrow} F_S\circ N(-/\mathcal{L}\times_\mathcal{L}Z)
$$
Since 
$$
\begin{array}{rcl}
F_S(N(-/\mathcal{L})\times_\mathcal{L}NZ)
&\cong& N(-/\mathcal{L})\otimes_\mathcal{L}(F_S\circ N)Z\\
&=&  N(-/\mathcal{L})\otimes_\mathcal{L}(N\circ F_C)Z
\end{array}
$$
and
$$
\begin{array}{rcl}
F_S(N(-/\mathcal{L}\times_\mathcal{L}Z))
&=& N\circ F_C(-/\mathcal{L}\otimes_\mathcal{L}Z)\\
&\cong& N(-/\mathcal{L}\otimes_\mathcal{L} F_C Z)
\end{array}
$$
the lower horizontal map in following commutative diagram
$$\xymatrix
{
N(-/\mathcal{L})\otimes_\mathcal{L}Q N\str F_CZ \ar[rr]^\alpha \ar[d] && N(-/\mathcal{L}\otimes_\mathcal{L}\str F_C Z)
\ar[dd] \\
N(-/\mathcal{L})\otimes_\mathcal{L}Q N F_CZ \ar[d] && \\
N(-/\mathcal{L})\otimes_\mathcal{L} F_S NZ \ar[rr] && N(-/\mathcal{L}\otimes_\mathcal{L} F_C Z)
}
$$
is a weak equivalence. The two left vertical maps are induced by $r: \str F_C Z\to F_CZ$ and the functorial cofibrant
replacement map $QNF_C Z \to NF_C Z=F_SNZ$. Since these maps are objectwise weak equivalences between objectwise cofibrant
diagrams (for any Y in $\SSets$ the algebra
$F_SY$ is cofibrant), the two vertical maps of the diagram are weak equivalences by \cite[18.5.3]{Hirsch}. The right
vertical map is a weak equivalence by the proof of \ref{5_7}. Hence $\alpha $ is a weak equivalence. We summarize:

\begin{lem}\label{6_12}
Given a diagram $X:\mathcal{L}\stackrel{Z}{\longrightarrow}\Cat\stackrel{F_C}{\longrightarrow}\Cat^\mathbb{M}$, the natural map
$$
\alpha(X):\hocolim^{N\mathbb{M}}N\str X\longrightarrow N(\hocolim^\mathbb{M} X)
$$
is a weak equivalence for any $\Sigma$-free operad in $\Cat$.\hfill\ensuremath{\Box}
\end{lem}

Now let $X:\mathcal{L}\to\Cat^\mathbb{M}$ be an arbitrary diagram. We will resolve $X$ to analyze its homotopy colimit:

\begin{leer}\label{6_13}
Let $\Delta_+$ denote the category of ordered sets $\underline{n}=\{-1<0<\ldots<n\}$, $n\geq -1$, and monotone maps 
preserving $-1$. 
We have an obvious inclusion of the simplicial indexing category $\Delta\subset \Delta_+$. Giving a functor
$X: \Delta_+^{\op}\to \mathcal{C}$ amounts to giving a simplicial object $X_\bullet$ in $\cC$ together with an object
$X_{-1}$, a morphism $\varepsilon = d_0: X_0\to X_{-1}$, and additional degeneracy morphisms $s_{-1}:X_n\to X_{n+1}$
for $n=-1,\ 0,\ldots$ satisfying the extra simplicial identities
$$\begin{array}{rcll}
   \varepsilon d_0 & = & \varepsilon d_1 & \textrm{i.e.}\quad d_0d_0=d_0d_1\\
s_{-1}s_i &=&  s_{i+1}s_{-1} & \\
d_0s_{-1} &=& \id & \\
d_is_{-1} &=& s_{-1}d_{i-1} & \textrm{for}\quad i>0
  \end{array}
$$
Equivalently, a functor $X:\Delta^{\op}_+\to\mathcal{C}$ into any category $\mathcal{C}$ consists of a simplicial object
 $X_\bullet$ in $\mathcal{C}$ together with a simplicial map
$$
\varepsilon:X_\bullet \to X_{-1}
$$
to the constant simplicial object on $X_{-1}$, which in degree $n$ is defined by $\varepsilon_n= (d_0)^{n+1}$, and which we  also
denote by $\varepsilon$,
a simplicial section
$$
s:X_{-1}\to X_\bullet
$$
given in degree $n$ by $(s_{-1})^{n+1}$,
and a simplicial homotopy $s\circ\varepsilon\simeq\id$.
\end{leer}

The Godement resolution of $\mM=U_C\circ F_C$ \cite[App.]{Gro} is a functor
$$
\mathbb{M}^+_\bullet:\Cat^\mathbb{M} \to\Cat^{\Delta_+^{\op}}
$$
better known as the functorial 2-sided bar construction
$$
\xymatrix{
B_\bullet(\mathbb{M},\mathbb{M},X) \ar@<0.5ex>[r]^(.7)\varepsilon & X\ar@<0.5ex>[l]^(.3)s
}
$$
with its augmentation $\varepsilon$ and section $s$ (see \cite[Chap. 9]{May}). Recall that $B_\bullet(\mathbb{M},\mathbb{M},X)$ 
is a simplicial object in $\Cat^\mathbb{M}$ and $\varepsilon$ is a simplicial map in $\Cat^\mathbb{M}$ to the constant
 simplicial object on $X$. The morphism $s$ is a section of $\varepsilon$ in $\Cat^{\Delta^{\op}}$ but not in 
$(\Cat^\mathbb{M})^{\Delta^{\op}}$. For further details see \cite[2.2.2]{Thom2}.

Let $\mathbb{M}_\bullet=B_\bullet(\mM,\mM,-)$ denote the restriction of $\mathbb{M}_\bullet^+$ to $\Delta^{\op}$. 
Applying the natural map $\alpha$ dimensionwise we obtain a map of bisimplicial sets
$$
\alpha_\bullet: N(-/\mathcal{L})\otimes_\mathcal{L}QN\str\mathbb{M}_\bullet X \longrightarrow N(-/\mathcal{L}\otimes_\mathcal{L}\str\mathbb{M}_\bullet X)
$$
Consider the following commutative diagram
$$
\xymatrix{
\diag(N(-/\mathcal{L})\otimes_\mathcal{L}QN\str \mathbb{M}_\bullet X) \ar[rr]^{\diag (\alpha_\bullet)}\ar[d]_{\varepsilon_1}&&
\diag N(-/\mathcal{L}\otimes_\mathcal{L}\str\mathbb{M}_\bullet X) \ar[d]^{\varepsilon_2} 
\\
N(-/\mathcal{L})\otimes_\mathcal{L}QN\str X\ar[rr]^\alpha &&
N(-/\mathcal{L}\otimes_\mathcal{L}\str X)
}
$$
where the vertical maps are induced by the augmentation $\varepsilon: \mathbb{M}_\bullet X\to X$.

Recall that
$\mathbb{M}_n = \mathbb{M}^{n+1}$, so that $\mathbb{M}_nX$ is a free $\mM$-algebra for each $n\ge 0$. Hence the natural maps $\alpha_n$ are weak equivalences by \ref{6_8}, and
$\diag (\alpha_\bullet)$ is a weak equivalence by \cite[IV.1.7.]{JG}. To prove Theorem \ref{6_5} we show that the two vertical
maps of the diagram are weak equivalences.

\begin{lem}\label{6_14}
For any diagram $X:\cL\to \Cat^{\mM}$ and any $\Sigma$-free operad in $\Cat$ the map
$$
 \varepsilon_1 : \diag(N(-/\mathcal{L})\otimes_\mathcal{L} QN\str \mathbb{M}_\bullet X )\to N(-/\mathcal{L})\otimes_\mathcal{L}QN\str X
$$
is a weak equivalence.
\end{lem}
\begin{proof}
If $K_\bullet$ is a simplicial object in $\SSets$ and $\nabla:\Delta\to\SSets$ is the functor which maps $[n]$ to the 
standard simplicial $n$-simplex then 
$$
\diag K_\bullet\cong\nabla\times_{\Delta^{\op}}K_\bullet
$$
By the argument of \cite[4.4]{MSV} the ``internal realization'' in $\SSets^{N\mathbb{M}}$ coincides with the ``external realization''
in $\SSets$; in other words,
$$
\diag A_\bullet\cong\nabla\otimes_{\Delta^{\op}} A_\bullet
$$
in $\SSets^{N\mathbb{M}}$ for any simplicial object $A_\bullet$ in $\SSets^{N\mathbb{M}}$. Since two tensors commute we obtain
$$
\begin{array}{rcl}
\diag(N(-/\mathcal{L})\otimes_\mathcal{L}QN\str\mathbb{M}_\bullet X)
& \cong & \nabla\otimes_{\Delta^{\op}}( N(-/\mathcal{L})\otimes_\mathcal{L}QN\str\mathbb{M}_\bullet X)\\
& \cong & N(-/\mathcal{L})\otimes_\mathcal{L}(\nabla\otimes_{\Delta^{\op}}QN\str\mathbb{M}_\bullet X)\\
& \cong & N(-/\mathcal{L})\otimes_\mathcal{L}\diag (QN\str\mathbb{M}_\bullet X)
\end{array}
$$
in $\SSets^{N\mathbb{M}}$. For an object $L$ in $\cL$ we consider the commutative diagram
$$
\xymatrix{
\diag(QN\str \mM_\bullet X(L)) \ar[rr]^{\diag(QNr_\bullet)} \ar[d]_{QN\str(\varepsilon)} && \diag(QN \mM_\bullet X(L))\ar[d]^{QN\varepsilon}\\
QN\str X(L) \ar[rr]^{QNr} && QNX(L)
}
$$
where $r$ is the  homomorphism of \ref{5_2}. By \ref{5_3}, the maps $Nr$ and $Nr_n: N\str \mM_n X(L)\to N\mM_nX(L)$ are weak equivalences. Hence so are
$QNr$ and, by  \cite[IV.1.7.]{JG}, also $\diag(QNr_\bullet)$. 

From \ref{6_11} we deduce that 
$$
N\circ\mathbb{M}^{n+1}=(N\mathbb{M})^{n+1}\circ N.
$$
It follows that $N \mM_\bullet X(L)=(N\mM)_\bullet NX(L)$ is a resolution of $NX(L)$, so that $QN\varepsilon: \diag(QN \mM_\bullet X(L)\to QNX(L)$
is a weak equivalence. Since 
$$
\diag(QN\str \mM_\bullet X(L)) \to QN\str X(L)
$$ 
is a weak equivalence for each object $L$ in $\cL$, \cite[18.5.3]{Hirsch} implies that
$$
\varepsilon_1: \diag(N(-/\mathcal{L})\otimes_\mathcal{L} QN\str (\mathbb{M}_\bullet X )\to N(-/\mathcal{L})\otimes_\mathcal{L}QN\str X
$$
is a weak equivalence.
\end{proof}

Finally we show

\begin{lem}\label{6_15}
 If $\cM$ satisfies the factorization condition then
$$
\varepsilon_2: \diag N(-/\mathcal{L}\otimes_\mathcal{L}\str\mathbb{M}_\bullet X)\to N(-/\mathcal{L}\otimes_\mathcal{L} \str X).
$$
is a weak equivalence.
\end{lem}

Since $|\diag(N(-/\mathcal{L}\otimes_\mathcal{L}\str\mathbb{M}_\bullet  X)|\cong|B(-/\mathcal{L}\otimes_\mathcal{L}\str\mathbb{M}_\bullet X)|$ 
we may investigate
$$
|B(-/\mathcal{L}\otimes_\mathcal{L}\str\mathbb{M}_\bullet X )| \to B(-/\mathcal{L}\otimes_\mathcal{L}\str X)
$$
Since $B(-/\mathcal{L}\otimes_\mathcal{L}\str\mM_\bullet X)$ is a simplicial $CW$-complex, the inclusions of the spaces
of degenerate elements are closed cofibrations. Hence its realization is equivalent
 to its fat realization, which ignores degeneracies and which is known to be equivalent to its homotopy 
colimit along $\Delta^{\op}$ in $\Top$:

\begin{leer}\label{6_16}
$$\xymatrix{
|B(-/\mathcal{L}\otimes_\mathcal{L}\str\mathbb{M}_\bullet X )|& \parallel B(-/\mathcal{L}\otimes_\mathcal{L}\str\mathbb{M}_\bullet X )\parallel
\ar[l]_\simeq \ar[r]^(.63){\simeq} &\hocolim^{\Top}_{\Delta^{\op}}BH
}$$
\end{leer}
where $H$ is the functor
$$
H: \Delta^{\op}\to \Cat,\qquad [n]\mapsto (-/\mathcal{L})\otimes_\mathcal{L}\str\mathbb{M}_n  X.
$$ 
If $\cX=(-/\mathcal{L})\otimes_\mathcal{L}\str X$ we have to show that the augmentation $\varepsilon$ induces a weak equivalence
$$
 \hocolim^{\Top}_{\Delta^{\op}} BH\to B\cX.$$
We now follow in part Thomason's argument of \cite[p 1641ff]{Thom1}:

Let $F:\mathcal{L}\to\Cat$ be a functor. The dual Grothendieck construction $F\int\mathcal{L}$ is the 
category whose objects are pairs $(L,X)$ with an object $L$ of $\mathcal{L}$ and an object $X$ of $F(L)$. A morphism
$$
(l,x):(L,X)\to (L',X')
$$
consists of a morphism $l:L'\to L$ in $\mathcal{L}$ and a morphism $x:X\to F(l)X'$ in $F(L)$. Composition
 is defined by $(l_2,y)\circ(l_1,x)=(l_1\circ l_2, F(l_1)y\circ x)$. It is related to Thomason's 
Grothendieck construction $\mathcal{L}\int F$ by
$$
F\int\mathcal{L}=(\mathcal{L}\int F^{\op})^{\op}
$$
where $F^{\op}:\mathcal{L}\to\Cat$ sends $L$ to $F(L)^{\op}$. 

We have a commutative diagram
$$
\xymatrix{
B((\Delta^{\op}\int H^{\op})^{\op})\ar[d] & B(\Delta^{\op}\int H^{\op})\ar[l]_\tau \ar[d]& \hocolim^{\Top}_{\Delta^{\op}} B(H^{\op})
\ar[l]_h \ar[ld]\ar[r]^\tau & \hocolim^{\Top}_{\Delta^{\op}} BH\ar[d]\\
B\cX & B(\cX^{\op})\ar[l]_\tau \ar[rr]^\tau && B\cX
}
$$
where $h$ is the weak equivalence of Thomason's homotopy colimit construction in $\Cat$ \cite{Thom0}, the unspecified maps are induced by the
augmentation, and $\tau: B(\cC^{\op}) \cong B\cC$ is the well-known natural homeomorphism. Hence we are left to prove that
the augmentation induces a weak equivalence
$$
\varepsilon: H\int \Delta^{\op} \to \cX=\hocolim^{\Cat^{\mM}}X.
$$
By Quillen's Theorem A \cite[\S 1]{Quillen} it suffices to show that the comma category $\varepsilon/K$ is contractible 
for each object $K$ of $\cX$.

Let $F:\mathcal{L}\to\Cat$ be a functor and $\varepsilon:F\int\mathcal{L}\to\mathcal{C}$ be any functor. Let $j(L):F(L)\to F\int\mathcal{L}$ be 
the functor sending $X$ to $(L,X)$ and let $\varepsilon(L)=\varepsilon\circ j(L)$. Then for any object $C$ of $\mathcal{C}$ the comma category 
$\varepsilon/C$ is isomorphic to $(\varepsilon(\ )/C)\int\mathcal{L}$ where $\varepsilon(\ )/C:\mathcal{L}\to\Cat$ sends $L$ to $\varepsilon(L)/C$ \cite[Lemma 4.6]{Thom1}.

We apply this to our functor $H$: Define 
$$\varepsilon_n=\varepsilon\circ j([n]): H([n])\to H\int \Delta^{\op}\to \hocolim^{\Cat^{\mM}}X,
$$ 
so that $\varepsilon_n: (-/\cL)\otimes_{\cL}\str\mM_n X\to (-/\cL)\otimes_{\cL}\str X.$
Let $T[P,\bar{L}]$ stand for the object $[T;(P_i,\bar{L}_i)_{i=1}^r]$ in $\hocolim^{\Cat^{\mM}}X$, and let $G$ be the functor
$$ G: \Delta^{\op} \to\Cat, \quad [n]\to\varepsilon_n/T[P,\bar{L}].$$
We have to show that $G\int \Delta^{\op}$ is contractible. By \cite[Thm.1.2]{Thom0} and \ref{6_15} we have homotopy equivalences
$$B(G\int \Delta^{\op})\quad\simeq \quad\hocolim^{\Top}_{\Delta^{\op}} BG \quad\simeq \quad\parallel [n]\mapsto BG[n] \parallel
$$
If $\cM$ has nullary operations, then $[A]$ is an object in $\hocolim^{\Cat^{\mM}}X$ if $A$ is an object in $\cM(0)$. The category
$\varepsilon_n/[A]$ is isomorphic to $\cM(0)/A $ and hence contractible, which implies that $\varepsilon/[A]$ is contractible.
The case of $T[P,\bar{L}]$ as above follows from

\begin{lem}\label{6_17}
If $\cM$ satisfies the factorization condition, the functor $G$ extends to a lax functor $G_+:\Delta_+^{\op} \to \Cat$ 
in the sense of \ref{2_8} such that $G_+([-1])=\varepsilon_{-1}/T[P,\bar{L}]$,
where $\varepsilon_{-1}: (-/\cL)\otimes_{\cL}\str X\to (-/\cL)\otimes_{\cL}\str X  $ is the identity (recall that $\mathbb{M}_nX=\mathbb{M}^{n+1}X$).
\end{lem}

The proof is deferred to Appendix B.

We now apply Street's first rectification construction \ref{2_9} with $k:\Delta^{\op}\subset \Delta^{\op}_+$. Since $G_+\circ k=G$
is a strict functor we obtain maps of simplicial spaces
$$
\xymatrix{
B((SG_+)|\Delta^{\op}) & B(SG)\ar[l]_(.4){B\xi} \ar[r]^{B\varepsilon} & BG
}
$$
which are dimensionwise homotopy equivalences, because an adjunction induces a homotopy equivalence of classifying spaces. Hence 
$$
\parallel BG\parallel \quad \simeq \quad \parallel B((SG_+)|\Delta^{\op}) \parallel \quad \simeq \quad B(SG_+([-1])\ \simeq \ \ast.
$$
The latter holds by definition of $\Delta_+^{\op}$ and the fact that 
$$B(SG_+([-1])) \simeq B(G_+([-1]))\simeq \ast$$
because $G_+([-1])$ has a terminal object. This completes the proof of Lemma \ref{6_15}.

Theorem \ref{6_9} has some immediate consequences:

\begin{leer}\label{6_18}
 \textbf{Cofinality Theorem:} Let $\cM$ be a $\Sigma$-free operad having the homotopy colimit property.
Let $X:\cL \to \Cat^\mathbb{M}\Lax$ be a diagram and let $F:\cN\to \cL$ be a functor of small
categories. Suppose that for each object $L$ in $\cL$ the space $B(L\downarrow F)$ is contractible. Then
$$
F_\ast : \hocolim \ X\circ F \to  \hocolim \ X
$$
is a weak equivalence.
\end{leer}

\begin{proof}
By \ref{6_9} the natural horizontal maps in the commutative diagram
$$
\xymatrix{
\hocolim_{\cN}^{B\mM}B(\str (X\circ F))\ar[rr] \ar[d]^{F_\ast} && B(\hocolim_{\cN}^{\mM}X\circ F)\ar[d]^{F_\ast}\\
\hocolim_{\cL}^{B\mM}B(\str X)\ar[rr]  && B(\hocolim_{\cL}^{\mM}X)
}
$$
are weak equivalences. By \cite[19.6.7]{Hirsch} the left vertical map is a weak equivalence, hence the result follows.
\end{proof}

\begin{leer}\label{6_19}
 \textbf{Homotopy invariance:}
 Let $\cM$ be a $\Sigma$-free operad having the homotopy colimit property.
  Let $f: X\to Y$ be a strict morphism of $\cL$-diagrams in $\Cat^{\mM}\Lax$ such that for each
object $L\in \cL$ the underlying map $f(L)$ of the lax morphism $(f(L),\overline{f(L)}):X(L)\to Y(L)$
is a weak equivalence. Then 
$$\hocolim\ f: \hocolim\ X\to \hocolim\ Y
$$
is a weak equivalence.
\end{leer}

\begin{proof}
By \ref{5_3} the morphism $\str f: \str X\to \str Y$ of diagrams is an objectwise weak equivalence. Hence $N\str f$ is objectwise
a weak equivalence, which in turn implies that $\hocolim^{N\mathbb{M}}(N\str f)$ is a weak equivalence by \ref{6_3}.
\end{proof}

\section{Equivalences of categories}

Throughout this section let $\cM$ be a $\Sigma$-free $\Cat$-operad. Then $N\cM$ is  a simplicial and 
$B\cM$ a topological operad. Let $\mM,\ N\mM$ and $B\mM$ be their associated monads.

For a category $\cC$ let $\cS\cC$ denote the category of simplicial objects in $\cC$. In par\-ticular,
$\cS^2\cS ets$ is the category of bisimplicial sets. We work with the diagram
\begin{leer}\label{7_1}
$$
\xymatrix{
\scatm \ar[rr]^{\cS N^\alg} \ar[dd]^{\hocolim} && \cS^2\cS ets^{N\mM}\ar[dd]^{\diag}\ar[rr]^{\cS |-|^{\alg}} &&
cw\stopbm\ar[r]^{\subset}\ar[dd]^{|-|^{\alg}} &\stopbm\ar[ddl]^{|-|^{\alg}}\\
& I && II &&\\
\Cat^\mathbb{M} \ar[rr]^{N^\alg} \ar[dd]^{U_C} &&
\SSets^{N\mathbb{M}} \ar[dd]^{U_S}\ar[rr]^{|-|^{\alg}}&&\Top^{B\mM}\ar[dd]^{U_T} &\\
& III&& IV &&\\
\Cat \ar[rr]^{N} &&
\SSets \ar[rr]^{|-|} && \Top &
}
$$
where $cw\stopbm$ is the full subcategory of $\stopbm$  of those simplicial $B\mM$-algebras whose 
underlying spaces are simplicial
$CW$-complexes with cellular structure maps. Again, $|-|^{\alg}$ is the usual topological realization functor,
which lifts to algebras because it is a product preserving left adjoint. The squares $III$ and $IV$ commute while
the square $II$ commutes up to natural isomorphisms. We will deal with square $I$ later.
\end{leer}

We define \textit{weak equivalences} in each of these categories to be morphisms which are mapped by the functors 
of the diagram to weak equivalences in $\Top$. For $\scatm$ we here choose the composition through $\stopbm$. The
weak equivalences in all categories except for $\scatm,\ \Cat^\mathbb{M}$ and $cw\stopbm$ are part of a Quillen model
structure so that the localizations of these categories with respect to their classes of weak equivalences exist. It is easy to see that
the localization $cw\stopbm [\we^{-1}]$ exists and that it can be considered as a full subcategory of $\stopbm [\we^{-1}]$.
We cannot show that the localizations of $\scatm$ and $\Cat^\mathbb{M}$ exist in the G\"odel-Bernay set theory setting. One way out
is to work in the setting of Grothendieck universes where the localizations exist in a possibly higher universe. But if
$\cM$ has the homotopy colimit property we can offer another remedy:

\begin{defi}\label{7_4} Let $\cC$ be a category and $\we\subset \mor \cC$ be a class of morphisms in $\cC$. A \textit{localization up to 
equivalence with respect to} $\we$ is a category $\cC [\widetilde{\we^{-1}}]$ together
with a functor $\gamma:\cC\to \cC [\widetilde{\we^{-1}}]$ having the following universal property:
if $F:\cC\to \cD$ is a functor which maps morphisms in $\we$ to isomorphisms then there is a functor
 $\bar{F}:\cC [\widetilde{\we^{-1}}]\to \cD$, unique up to natural equivalence, such that $F$ and $\bar{F}\circ \gamma$
 are naturally equivalent.
\end{defi}

\begin{leer}\label{7_2}
Let $\SSets^{N\mathbb{M}}_{\cat}[\we^{-1}] \subset \SSets^{N\mathbb{M}}[\we^{-1}]$ and 
$\cS^2\cS ets^{N\mM}_{\cat}[\we^{-1}]  \subset \cS^2\cS ets^{N\mM}[\we^{-1}]$ denote the full subcategories of objects in the image of
$N^{\alg}$ respectively $\cS N^{\alg}$. Let
$\gamma_S:\SSets^{N\mathbb{M}}\to \SSets^{N\mathbb{M}}[\we^{-1}]$ and $\gamma_{S^2}: \cS^2\cS ets^{N\mM}\to cS^2\cS ets^{N\mM}[\we^{-1}]$
denote the localizations.
 Define
$$
\gamma_C=\gamma_S\circ N^\alg:  \catm \to \SSets^{N\mathbb{M}}_{\cat}[\we^{-1}]$$
$$\gamma_{SC}=\gamma_{S^2}\circ \cS N^\alg: \scatm \to \cS^2\cS ets^{N\mM}_{\cat}[\we^{-1}] 
$$
\end{leer}

We will prove the following result in Appendix C.

\begin{prop}\label{7_3} Let $\cM$ be an operad in $\Cat$ having the homotopy colimit property. Then
\begin{enumerate}
 \item  in the setting of Grothendieck universes 
 $$\catm [\we^{-1}]\cong  \SSets^{N\mathbb{M}}_{\cat}[\we^{-1}] \quad\textrm{and}\quad \scatm [\we^{-1}]\cong \cS^2\cS ets^{N\mM}_{\cat}[\we^{-1}].
 $$
 In particular, the localizations exist in the universe we started with and are induced by the functors of \ref{7_2}.
 \item In the G\"odel-Bernay set theory setting the functors of \ref{7_2} are localizations up to equivalence. So we may choose
 $\SSets^{N\mathbb{M}}_{\cat}[\we^{-1}] $ as $\Cat^\mathbb{M}[\widetilde{\we^{-1}}]$ and $\cS^2\cS ets^{N\mM}_{\cat}[\we^{-1}]$
 as $\scatm [\widetilde{\we^{-1}}]$.
 \end{enumerate}
\end{prop}
\begin{conv}\label{7_3a} 
 Throughout this section we state results concerning localizations in the G\"odel-Bernay foundational setting so that
 localizations up to equivalence occur. In Grothendieck's setting of universes, these can be replaced by genuine 
 localizations.
\end{conv}

The main result of this section is 

\begin{theo}\label{7_5} If $\cM$ is a $\Sigma$-free operad in $\Cat$ having the homotopy co\-limit property, 
then the nerve and realization functor induce an equivalence of categories 
$$
\Cat^\mathbb{M}[\widetilde{\we^{-1}}]\simeq \SSets^{N\mathbb{M}}[\we^{-1}]\simeq \Top^{B\mathbb{M}} [\we^{-1}]
$$
\end{theo}

The second equivalence has already been discussed in Section 6, and in \cite[2.7]{FV} it is shown that for a
$\Sigma$-free $\Cat$-operad the functors of diagram \ref{7_1} induce equivalences of
categories
$$
\scatm [\widetilde{\we^{-1}}]\simeq cw\stopbm [\we^{-1}]  \simeq \stopbm [\we^{-1}].
$$
We establish the theorem by showing that square I of the diagram commutes up to a broken arrow of weak equivalences
and by proving

\begin{theo}\label{7_6} For any $\Sigma$-free operad $\cM$ in $\Cat$ 
 $\hocolim^{\mM}: \scatm \to \Cat^{\mM}$ and the constant diagram functor $c: \Cat^{\mM}\to \scatm$ induce
an equivalence of categories
$$
 \cS^2\cS ets^{N\mM}_{\cat}[\we^{-1}]  \simeq \SSets^{N\mathbb{M}}_{\cat}[\we^{-1}].
$$
\end{theo}

We start with investigating square I of diagram \ref{7_1}.

\begin{lem}\label{7_7}
 Let $A_\bullet$ be a simplicial $\mM$-algebra in $\Cat$ and $Q^{\Reedy}QN\str A_\bullet$ 
the functorial Reedy-cofibrant replacement of
$QN\str A_\bullet$. Then there are natural weak equivalences
$$\hocolim^{N\mM}N\str A_\bullet \leftarrow N(-/\Delta^{\op})\otimes_{\Delta^{\op}}Q^{\Reedy}QN\str A_\bullet
\rightarrow \diag NA_\bullet
$$
In particular, if $\cM$ is a $\Sigma$-free operad having the homotopy colimit property, then
square $I$ of diagram \ref{7_1} commutes up to a broken arrow of natural weak equivalences.
\end{lem}

\begin{proof}
 $Q^{\Reedy}X_\bullet $ is degreewise cofibrant for any $X_\bullet$ in $\cS^2\cS ets^{N\mM}$ by \cite[15.3.11]{Hirsch}. 
Hence the left morphism is a weak equivalence by the argument of Remark \ref{6_3}. 

By \cite[18.7.4]{Hirsch} the
Bousfield-Kan map defines a natural weak equivalence
$$
N(-/\Delta^{\op})\otimes_{\Delta^{\op}}Q^{\Reedy}QN\str A_\bullet\rightarrow \diag Q^{\Reedy}QN\str A_\bullet .
$$
Since the natural map $Q^{\Reedy}QN\str A_\bullet\rightarrow  NA_\bullet$ is degreewise a weak equivalence
$\diag Q^{\Reedy}QN\str A_\bullet\rightarrow \diag NA_\bullet$ is a weak equivalence by \cite[IV.1.7]{JG}.
\end{proof}

\begin{coro}\label{7_8} Let $\cM$ be a $\Sigma$-free operad having the homotopy colimit property. Then
 a morphism $f_\bullet:A_\bullet \to B_\bullet$ in $\cS\Cat^\mathbb{M}$ is a weak equivalence if and only if
$\hocolim\ f_\bullet$ is a weak equivalence. {\hfill\ensuremath{\Box}}
\end{coro}

For the proof of theorem \ref{7_6} we consider the adjunction

(A)\hspace{2.3cm} $\hocolim^{\mM}: \Func (\Delta^{\op},\Cat^{\mM}\Lax)\leftrightarrows \Cat^{\mM}:c$

The functor $c$ factors through the inclusion $\cS\Cat^{\mM}\subset \Func (\Delta^{\op},\Cat^{\mM}\Lax)$.
The theorem is a consequence of the next two lemmas.

\begin{lem}\label{7_9} Let $\cM$ be a $\Sigma$-free operad.
 The counit $\varepsilon(A):\hocolim^{\mM} cA\to A$ of adjunction (A) is a weak equivalence for all $A$ in $\Cat^{\mM}$.
\end{lem}
\begin{proof}
 By \ref{5_9} the counit $\varepsilon(A)$ can be identified with the composition
$$
\Delta^{\op}\otimes \str A\xrightarrow{F\otimes \id} \str A\xrightarrow{r} A
$$
where $F:\Delta^{\op} \to \ast$ is the constant functor. Now $r$ is a weak equivalence by \ref{5_3}
and $F\otimes \id$ is a weak equivalence by Lemma \ref{3_9}, because 
$\Delta^{\op}$ has an initial object.
\end{proof}

For the next lemma we will need the following additional adjunctions:
$$
\begin{array}{lcrclc}
 \textrm{(B)}& \phantom{filler} & \Cat^{\mM}\Lax(A,i_1B)&\cong& \Cat^{\mM}(\str A,B)&\phantom{filler, filler} \\
\textrm{(C)}& & \Func (\Delta^{\op},\Cat^{\mM})(Y,i_2Z)&\cong & (\Cat^{\mM})^{\Delta^{\op}}(RY,Z)&\\
\textrm{(D)}& & (\Cat^{\mM})^{\Delta^{\op}}(U,\bar{c}S) &\cong & \Cat^{\mM}(\colim\ U,S)&
\end{array}
$$
where $i_1:\Cat^{\mM}\to \Cat^{\mM}\Lax$ and $i_2: (\Cat^{\mM})^{\Delta^{\op}}\to \Func (\Delta^{\op},\Cat^{\mM})$
are the inclusions, $RY=(-/\Delta^{\op}/-)\otimes_{\Delta^{\op}}Y$, and $\bar{c}: \Cat^{\mM}\to (\Cat^{\mM})^{\Delta^{\op}}$
is the constant diagram functor.

For an $X$ in $\scatm=(\Cat^{\mM})^{\Delta^{\op}}$ consider the diagram (where we suppress various inclusion functors)
$$
\xymatrix{
X\ar[r]^{\bar{j}} \ar[ddrr]_j & \str X \ar[r]^{\alpha(\str X)}& R\str X\ar[dd]^{\mu(\str X)}\ar[r]^{\beta(\str X)}
& \str X \ar[r]^r & X\\
\\
&&\colim\ R\str X=\hocolim^{\mM} X &
}
$$
where $j,\ \bar{j}, \ \alpha(\str X)$, and $\mu(\str X)$ are the units of the adjunctions 
(A),...,(D), and $\beta(\str X)$ and $r$ are
the counits of the adjunctions (C) and (D) ($\bar{j}$ and $r$ are applied degreewise). The triangle 
commutes where the maps are considered as $1$-cells in $\Cat^{\mM}\Lax$. We recall that $j$ and $\bar{j}$
are homotopy morphisms of simplicial objects in $\Cat^{\mM}\Lax$, that $\alpha(\str X)$ is a
 homotopy morphism of simplicial objects in $\Cat^{\mM}$, while the other three maps are strict
morphisms of simplicial objects in $\Cat^{\mM}$. 

\begin{lem}\label{7_10} Let $\cM$ be a $\Sigma$-free operad.
 For $X$ in $\scatm$ we have natural weak equivalences 
$$
c(\hocolim^{\mM} X)\xleftarrow{\mu(\str X)} R\str X\xrightarrow{r\circ\beta(\str X)} X
$$
in $\scatm$.
\end{lem}
\begin{proof}
 By \ref{7_8} it suffices to show that the homotopy colimits of these morphisms are weak equivalences.

By \ref{3_10} and \ref{5_3} the homotopy colimits of $r$ and $\beta(\str X)$ are weak equivalences. It remains
to prove that $\mu(\str X)$ is a weak equivalence. By construction, $r\circ \beta(\str X)\circ \alpha(\str X)\circ \bar{j}=\id_X$.
Hence $\hocolim (\alpha(\str X)\circ \bar{j})$ is a weak equivalence. Since $j$ is the unit of the adjunction (A)
$$
\hocolim^{\mM} X\xrightarrow{\hocolim^{\mM}j} \hocolim^{\mM}c(\hocolim^{\mM} X)\xrightarrow{\varepsilon(\hocolim^{\mM}X)}
\hocolim^{\mM} X
$$
is the identity. Since $\varepsilon(\hocolim^{\mM}X)$ is a weak equivalence by \ref{7_9}, so is $\hocolim^{\mM} j$ and hence $\mu(\str X)$.
\end{proof}

\section{Applications}

Throughout this section we will use Convention \ref{7_3a}.

Since we will consider various different $\Cat$-operads $\cM$ we will shift from the categories $\Cat^{\mM}$
of algebras over the associated monads $\mM$ back to the isomorphic categories $\Cat^{\cM}$ of $\cM$-algebras.

The operads $\cM$ in this section will be \textit{reduced}, i.e. $\cM_0$ consists of a single element. Hence
$\cM$-algebras are naturally based by the images of the nullary operation.

For the comparison of categories of algebras recall the following result:

\begin{prop}\label{8_1}
 Let $\varepsilon:\mathcal{P}\to\mathcal{Q}$ be a morphism of $\Sigma$-free
simplicial or topological  operads such that
$\varepsilon(n):\mathcal{P}(n)\to\mathcal{Q}(n)$ is a weak equivalence for
each $n\in\mathbb{N}$. Then the forgetful functor and its left adjoint
define a Quillen equivalence
$$L:\SSets^{\cP}\rightleftarrows \SSets^{\cQ}:R\quad\textrm{respectively}\quad 
L:\mathcal{T}op^{\mathcal{P}}\rightleftarrows \mathcal{T}op^{\mathcal{Q}}:R
$$
\end{prop}

The topological version is the reformulation in the language of model category structures of
the well-known change of operads  functor of \cite{May1}, defined by a
functorial two-sided bar construction. In fact, the two-sided bar construction is just
 the derived functor of the left adjoint (see also \cite[Prop. 2.8]{FV}). The simplicial version
 follows from the topological one by applying the Quillen equivalence \ref{6_1a}.

In connection with Theorem \ref{7_5} we obtain

\begin{prop}\label{8_2}
 Let $\varepsilon:\mathcal{P}\to\mathcal{Q}$ be a morphism of $\Sigma$-free
$\mathcal{C}at$-operads having the homotopy colimit property. Suppose that each
$\varepsilon(n):\mathcal{P}(n)\to\mathcal{Q}(n)$ is a weak equivalence.
Then the forgetful functor and its left adjoint induce an
equivalence of categories
$$
\xymatrix{\mathcal{C}at^{\mathcal{Q}}[\widetilde{we^{-1}}] \ar[rr] && 
\mathcal{C}at^{\mathcal{P}}[\widetilde{we^{-1}}]
}
$$
\end{prop}
\vspace{2ex}

\textbf{Iterated monoidal categories}

$k$-fold monoidal categories were introduced in \cite{BFSV}. Their structure is encoded
by the $\Sigma$-free operad $\mathcal{M}_k$ in $\Cat$ defined as follows: the category
$\mathcal{M}_k(m)$ is a poset  whose objects are words of length $m$ in the alphabet
$\{1,2,\dots,m\}$ combined together using $k$ binary operations $\square_1,\square_2,\dots,\square_k$,
which are strictly associative and have common unit 0.
We moreover require that the objects of $\mathcal{M}_k(m)$ are precisely those words where
each generator $\{1,2,\dots,m\}$ occurs exactly once. 

To give an explicit condition for the 
existence of a morphism $\alpha\to \beta$ we need some notation. For a subset $S\subset \{1,2,\dots,m\}$
we define $\alpha\cap S$ to be the word obtained from $\alpha$ by replacing each generator not in $S$ by the
unit 0 and reducing the word. There exists a morphism $\alpha\to \beta$ if and only if for any 
two element subset $\{ a,b \}$ of $\{1,2,\dots,m\}$ with $\alpha\cap\{a,b\}=a\square_i b$ we have 
either $\beta\cap\{a,b\}=a\square_j b$ with $j\ge i$ or $\beta\cap\{a,b\}=b\square_j a$ with
$j>i$. 

For objects $\alpha \in \mathcal{M}_k(m)$ and $\beta_i\in \mathcal{M}_k(r_i), \ 1\leq i\leq m$, 
operad composition $\alpha\ast (\beta_1\oplus \ldots \oplus \beta_m)$ is defined as follows: replace
the generator $s$ in $\beta_i$ by $r_1+\ldots +r_{i-1}+s $ and then substitute the generator $i$ in
$\alpha$ by this new $\beta_i$ to obtain a word in $\mathcal{M}_k(r_1+\ldots +r_m)$. This extends
canonically to morphisms. As an operad, $\cM_k$ is generated by the interchange morphisms
$$
\eta^{ij}: (1\square_j 2)\square_i(3\square_j 4)\to (1\square_i 3)\square_j (2\square_i 4)$$
with $1\leq i<j\leq k$. For more details see \cite{BFSV}.

We next compare our Definition \ref{3_3} of a lax morphism $(F,\overline{F}):\cA \to \cB$ of $2$-fold
monoidal categories with the definition of \cite[p. 282]{BFSV}. For $i=1,\ 2$ define
$$
\lambda^i_{A,B}=\overline{F}(1\square_i 2;A,B): F(A)\square_i F(B)\to F(A\square_i B).
$$
Since 
$$(1\square_i 2)\ast ((1)\oplus (1\square_i 2))= 1\square_i 2\square_i 3 = (1\square_i 2)\ast ((1\square_i 2)\oplus (1))
$$
both transformations satisfy the associativity condition of monoidal functor by \ref{3_3}(2), i.e. the diagram
$$
\diagram
F(A)\square_i F(B)\square_i F(C)
\rrto^{\lambda^i_{A,B}\square_i id_{F(C)}}
\dto^{id_{F(A)}\square_i \lambda^i_{B,C}}
&&F(A\square_i B)\square_i F(C)
\dto^{\lambda^i_{A\square_i B,C}}\\
F(A)\square_i F(B\square_i C)
\rrto^{\lambda^i_{A,B\square_i C}}
&&F(A\square_i B\square_i C)
\enddiagram
$$
commutes.

By \ref{3_3}(1) the hexagonal interchange  diagram
$$
\diagram
(F(A)\Box_2 F(B))\Box_1(F(C)\Box_2 F(D))
\xto[r]^{\eta^{12}F}
\dto^{\scriptstyle {\lambda^2_{A,B}\Box_1\lambda^2_{C,D}}}
&(F(A)\Box_1 F(C))\Box_2(F(B)\Box_1 F(D))
\dto^{\scriptstyle \lambda^1_{A,C}\Box_2\lambda^1_{B,D}}\\
F(A\Box_2 B)\Box_1 F(C\Box_2 D)
\dto^{\scriptstyle {\lambda^1_{A\Box_2 B,C\Box_2 D}}}
&F(A\Box_1 C)\Box_2 F(B\Box_1 D)
\dto^{\scriptstyle {\lambda^2_{A\Box_1 C,B\Box_1 D}}}\\
F((A\Box_2 B)\Box_1(C\Box_2 D))
\xto[r]^{ F\eta^{12}}
& F((A\Box_1 C)\Box_2(B\Box_1 D))
\enddiagram
$$
commutes, where $\eta^{12}F$ is short for $\eta^{12}(F(A),F(B),F(C),F(D))$ and $F\eta^{12}$ is short for
$F(\eta^{12}(A,B,C,D))$. Also observe that by \ref{3_3}(2) the left and right vertical maps equal
$$
\overline{F}((1\square_2 2)\square_1(3\square_2 4);A,B,C,D)\ \ \textrm{respectively}\ \ 
\overline{F}( (1\square_1 3)\square_2 (2\square_1 4)  ;A,B,C,D).
$$
Hence a lax morphism $(F,\overline{F})$ of $2$-fold monoidal categories in the sense of \cite{BFSV} is a lax
morphism in the sense of \ref{3_3} with the additional condition that $F$ preserves the 0 and that the
natural transformation $\overline{F}(0):0\to F(0)$ is the identity.

\begin{lem}\label{8_3}
 The operads $\mathcal{M}_k$ satisfy the factorization condition \ref{6_7} for all $k\geq 1$.
\end{lem}
\begin{proof}
Let  $\cC(\gamma,\alpha,r_1,\ldots ,r_m)$ be a non-empty factorization category 
(see \ref{6_7}), where $\gamma\in \mathcal{M}_k(q)$ with $q=r_1+\ldots +r_m$. 
In particular, there is an object $(\beta_1,\ldots.\beta_m;\nu)$ with objects $\beta_i\in \mathcal{M}_k(r_i)$
and a morphism $\nu:\gamma\to \alpha\ast (\beta_1\oplus\ldots\oplus\beta_m)$.
 Represent the ordered
set $\{ 1<\ldots <q\}$ as the ordered disjoint union $S_1\sqcup\ldots\sqcup S_m$ with $|S_i|=r_i$ and
define $\epsilon_i=\gamma \cap S_i$. By the above description of morphisms there is a (unique)
morphism $\rho: \gamma\to \alpha\ast (\epsilon_1\oplus \ldots \oplus \epsilon_m)$ and the 
mophisms $\nu$ factors (uniquely) through $\rho$. Hence the object
$(\epsilon_1,\ldots,\epsilon_m;\rho)$ is initial in $\cC(\gamma,\alpha,r_1,\ldots ,r_m)$.
\end{proof}

 Let $\mathcal{C}_k$ denote the little $n$-cubes operad of \cite{BV1}.
 In \cite{BFSV} there is
constructed a $\Sigma$-free topological operad $\mathcal{D}_k$ and
maps of operads
$$
\xymatrix{B\mathcal{M}_k & \mathcal{D}_k \ar[r]^{\eta_k} 
\ar[l]_{\varepsilon_k} 
&\mathcal{C}_k}
$$
such that each $\varepsilon_k(n)$ and each $\eta_k(n)$ is a homotopy
equivalence.
Let $\iota_k : \mathcal{M}_k \to \mathcal{M}_{k+1}$ denote the canonical
inclusion functor and $\rho_k: \mathcal{C}_k \to \mathcal{C}_{k+1}$ 
the canonical
embedding. Although not stated explicitly, a quick check of 
\cite[Chapt. 6]{BFSV} shows that there is also an embedding of operads
$\delta_k: \mathcal{D}_k \to \mathcal{D}_{k+1}$ making the diagram
$$
\xymatrix{
B\mathcal{M}_k \ar[d]^{B\iota_k} & \mathcal{D}_k \ar[r]^{\eta_k} 
\ar[l]_{\varepsilon_k} \ar[d]^{\delta_k}
&\mathcal{C}_k \ar[d]^{\rho_k} \\
B\mathcal{M}_{k+1} & \mathcal{D}_{k+1} \ar[r]^{\eta_{k+1}}
\ar[l]_{\varepsilon_{k+1}} & \mathcal{C}_{k+1} 
}
$$
commute. If we define $\mathcal{M}_\infty = \textrm{colim}_k \mathcal{M}_k$
and $\mathcal{D}_\infty = \textrm{colim}_k \mathcal{D}_k$, we obtain
induced morphisms of topological operads
$$
\xymatrix{B\mathcal{M}_\infty & \mathcal{D}_\infty \ar[r]^{\eta_\infty}
\ar[l]_{\varepsilon_\infty}
&\mathcal{C}_\infty}
$$
Since $\iota_k$ is an inclusion, each $B\iota_k(n)$ is a cofibration.
Since each $\rho_k(n)$
is a cofibration by \cite[Proof of Lemma 2.50]{BV2} and each $\delta_k(n)$ is one by inspection,
each $\varepsilon_\infty (n)$ and each $\eta_\infty (n)$ is a homotopy
equivalence. The arguments of the proof of Lemma \ref{8_3} also apply to $\mathcal{M}_\infty$ so that
we obtain

\begin{lem}\label{8_4}
 The operad $\mathcal{M}_\infty$ satisfies the factorization condition \ref{6_7}.
\end{lem}

Hence Theorem \ref{7_5} and Proposition \ref{8_1} imply 

\begin{theo}\label{8_5}
For $1\leq k\leq \infty$
the classifying space functor $B$ and the change of operads functors
induced by $\varepsilon_k$ and $\eta_k$ determine equivalences of categories
$$
\mathcal{C}at^{\mathcal{M}_k}[\widetilde{{\we}^{-1}}] \simeq
\mathcal{T}op^{B\mathcal{M}_k}[{\we}^{-1}] \simeq
\mathcal{T}op^{\mathcal{C}_k}[{\we}^{-1}]
$$
\end{theo}

\pagebreak
\textbf{Symmetric monoidal categories}

It is well-known that each symmetric monoidal category is equivalent to a permutative category. 
Permutative categories are algebras over the $\Cat$-version $\widetilde{\Sigma}$ of the Barratt-Eccles
operad with $\widetilde{\Sigma}(n)$ the translation categories of $\Sigma_n$
 (see \cite[1.1]{May1}) and $\widetilde{\Sigma}(0)$ the unit category. Let $\mathcal{P}erm$ 
denote the category of permutative categories. If $e_n$ denotes the unit in $\Sigma_n$
then $\widetilde{\Sigma}(2)$ is the category
$$
\bar{\tau}: e_2\leftrightarrows \tau:\bar{\tau}^{-1}
$$
where $\tau$ is the transposition.
The correspondence between a permutative category
$(\cA,\Box,0_{\cA},c)$ and its associated $\widetilde{\Sigma}$-algebra 
$\widehat{\cA}:\widetilde{\Sigma}\to \Cat$ is given by
 $$
A\Box B=\widehat{\cA}(e_2;A,B)\qquad 0=\widehat{\cA}(e_0;\ast)\qquad c_{A,B}=\widehat{\cA}(\bar{\tau};A,B):A\Box B \cong B\Box A.
$$
For more details see \cite[Section 4]{May1}.
 
If $(F,\overline{F}): \widehat{\cA}\to \widehat{\cB}$ is a lax morphism between $\widetilde{\Sigma}$-algebras with associated permutative
categories $(\cA,\Box,0_{\cA},c)$ and $(\cB,\boxtimes,0_{\cB},c')$, we have morphisms
$$
\overline{F}(e_2;A,B): FA\boxtimes FB\to F(A\Box B)\qquad \overline{F}(e_0;\ast):0_{\cB}\to F(0_{\cA}).
$$
Condition \ref{3_3}(1) ensures the commutativity of
$$
\xymatrix{
FA\boxtimes FB\ar[rr] \ar[d]_{c'_{FA,FB}} & & F(A\Box B)\ar[d]^{Fc_{A,B}}\\
FB\boxtimes FA\ar[rr] & & F(B\Box A)
}
$$
and \ref{3_3}(2),(3) the commutativity of
$$
\xymatrix{
FA\ar[r]^{\id}\ar[d]_{\id} & F(0_{\cA}\Box A) & & FA\ar[r]^{\id}\ar[d]_{\id} & F(A\Box 0_{\cA})\\
0_{\cB}\boxtimes FA\ar[r] & F(0_{\cA})\boxtimes FA\ar[u] & & FA \boxtimes 0_{\cB}\ar[r]& FA\boxtimes F(0_{\cA})\ar[u].\\
}
$$
So lax morphisms of $\widetilde{\Sigma}$-algebras define lax morphisms of permutative categories. Coherence implies that the converse 
also holds.

Since each object in $\widetilde{\Sigma}(n)$ is initial, we have

\begin{lem}\label{8_6}
 The operad $\widetilde{\Sigma}$ satisfies the factorization condition \ref{6_7}.
\end{lem}

The projections
$$
B(\widetilde{\Sigma})\longleftarrow B(\widetilde{\Sigma})\times
\mathcal{C}_\infty\longrightarrow \mathcal{C}_\infty
$$
are weak equivalences of $\Sigma$-free operads. Hence, by Theorem \ref{7_6} and
Proposition \ref{8_1} we obtain

\begin{theo}\label{8_7} The classifying space functor and the change of
operads
functors induce equivalences of categories
$$
\mathcal{P}erm[\widetilde{{\we}^{-1}}]=\mathcal{C}at^{\widetilde{\Sigma}}[\widetilde{{\we}^{-1}}]
\simeq\mathcal{T}op^{B(\widetilde{\Sigma})}[{\we}^{-1}]
\simeq\mathcal{T}op^{\mathcal{C}_\infty}[{\we}^{-1}].
$$
\end{theo}

Let $\widehat{\Sigma}\subset \widetilde{\Sigma}$
be the suboperad obtained from $\widetilde{\Sigma}$ by replacing $\widetilde{\Sigma}(0)$ by the empty category. Then
 the $\widehat{\Sigma}$-algebras are the non-unitary permutative categories. In \cite{Thom1} Thomason defined a
 homotopy colimit functor in the category of $\widehat{\Sigma}$-algebras which has the universal property \ref{4_12}.
Hence it has to be isomorphic to our construction. This can also be seen directly:
 Consider the category of standard representatives \ref{4_10}.2 for $\cQ_X$. There is just one $\Sigma_n$-orbit
of objects in $\widehat{\Sigma}(n)$ for $n\geq 1$ canonically represented by the unit $e_n\in \Sigma_n$. Thomason starts with this category
and denotes the canonical representative $(e_m;(K_i,L_i)_{i=1}^m)$ by $m[(L_1,K_1),\ldots,(L_m,K_m)]$. 
There is also a canonical choice of representatives
$$
(\varphi,(C_j)_{j=1}^n,\gamma,(\lambda_i)_{i=1}^m,(f_j)_{j=1}^n):(e_m;(K_i,L_i)_{i=1}^m)\to (e_n;(K'_j,L'_j)_{j=1}^n)
$$ 
Since there are no nullary operations $\varphi$ has to be surjective. We give each $\varphi^{-1}(k)$ its natural order. Let $r_k=|\varphi^{-1}(k)|$. Since any 
$$\gamma: e_m\to e_n\ast (C_1\oplus\ldots\oplus C_n)\cdot \sigma = (C_1\oplus\ldots\oplus C_n)\cdot \sigma$$
factors as
$$ e_m\xrightarrow{\sigma} (e_{r_1}\oplus\ldots\oplus e_{r_n})\cdot\sigma\xrightarrow{(C_1\oplus\ldots\oplus C_n)\cdot \sigma} (C_1\oplus\ldots\oplus C_n)\cdot \sigma$$
this representative is equivalent by Relation \ref{4_8}(3) to a representative of the form $(\varphi,(e_{r_j})_{j=1}^n,\sigma, (\lambda_i)_{i=1}^m,(g_j)_{j=1}^n)$.
Since $\sigma$ and the $r_k$ are determined by $\varphi$ we can suppress them and arrive at the definition of morphisms as given by Thomason
in \cite[Construction 3.22]{Thom1}.

\begin{rema}\label{8_8}
There is a minor flaw in Thomason's proof of our Theorem \ref{6_9} for non-unitary permutative categories
\cite[Theorem 4.1]{Thom1}.
 Using a spectral sequence argument he shows that the
map corresponding to our 
$$
\varepsilon_2: \diag N(-/\mathcal{L}\otimes_\mathcal{L}\str\mathbb{M}_\bullet X)\to
 N(-/\mathcal{L}\otimes_\mathcal{L} \str X)
$$ is a homology equivalence and concludes that it is a weak equivalence because the spaces involved are $H$-spaces
 \cite[pp.1641-1644]{Thom1}. This is not
true because he considers only non-unital permutative categories at that point. Since he applies a group
completion functor to his construction this flaw does not affect the main result of the paper.
\end{rema}

Recall from
\cite[Remark 1.9]{BFSV} that each permutative category is an infinite monoi\-dal category of a
special kind: take $\Box_i =\Box$ for all $i$ and $\eta^{ij}=\id\Box c \Box \id$.
There is an operadic interpretation of this remark
in terms of a morphism of operads 
$\lambda : \mathcal{M}_\infty \to \widetilde{\Sigma}$,
which is determined by sending an object
$ i_1\Box_{r_1} i_2\Box_{r_2}\ldots \Box_{r_{n-1}} i_n$
of $\mathcal{M}_\infty (n)$ with any form of bracketing to the
permutation in $\Sigma_n$ sending $k$ to $i_k$. As a corollary to
Theorem \ref{8_7} we obtain

\begin{coro}\label{8_9}
The forgetful functor $\mathcal{P}erm \to \mathcal{C}at^{\mathcal{M}_\infty}$
induced by $\lambda$ defines an equivalence of categories
$$
\mathcal{P}erm [\widetilde{\we^{-1}}]\simeq  
\mathcal{C}at^{\mathcal{M}_\infty}[\widetilde{\we^{-1}}]
$$
\end{coro}
\vspace{2ex}

\textbf{Braided monoidal categories}

\begin{defi}\label{8_10}
A \textit{weak braided monoidal category} is a category $\mathcal{C}$
together with a functor $\Box : \mathcal{C}\times \mathcal{C} \to
\mathcal{C}$ which is strictly associative and has a strict $2$-sided
unit object $0$ and with a natural commutativity morphisms
$c_{A,B}: A\Box B \longrightarrow B\Box A $
such that $c_{A,0}=c_{0,A}=id_A$ and the diagrams
$$
\xymatrix{
A \Box B \Box C \ar[rrrr]^{c_{A,B}\Box C}
\ar[ddrr]_{c_{A,B\Box C}}
&&&& B \Box A \Box C \ar[ddll]^{B\Box c_{A,C}} \\
&&&& \\
&& B \Box C \Box A
}
$$

\vspace{3ex}

$$
\xymatrix{
A \Box B \Box C \ar[rrrr]^{A\Box c_{B,C}}
\ar[ddrr]_{c_{A\Box B,C}}
&&&& A \Box C \Box B \ar[ddll]^{c_{A,C}\Box B} \\
&&&&
\\
&& C \Box A \Box B
}
$$
commute.

If the $c_{A,B}$ are isomorphisms, we call $\mathcal{C}$ a \textit{
braided monoidal category}.
\end{defi}

\begin{rema}\label{8_11}
Our terminology regarding braided monoidal categories differs from
standard usage, as for example in \cite{MacLane}. The usual
notion of braided monoidal categories has
$\Box :\mathcal{C}\times\mathcal{C}\to\mathcal{C}$
only associative up to coherent natural isomorphisms and similarly with
the 2-sided unit $0$. Braided monoidal categories with a strictly
monoidal structure are called strict braided monoidal. Essentially
the same proof as for symmetric monoidal categories, given in \cite{May1},
shows that any braided monoidal category, in this sense, is equivalent
to a strict one. Weak braided monoidal categories, in our sense,
where the braidings are not required to be isomorphisms, have not
been much studied.
\end{rema}

By \cite[Remark 1.5]{BFSV} any weak braided monoidal category is
a special kind of a $2$-fold monoidal category. 
Let $B_n$ denote the braid group  on $n$ strings and
$p: B_n \to \Sigma_n$ its projection onto the symmetric group, whose
kernel is the pure braid group $P_n$. Recall the standard presentation
of $B_n$:
$$
B_n=<\sigma_1,\ldots,\sigma_{n-1}|\sigma_i\sigma_j=\sigma_j\sigma_i
\textrm{ if }|i-j|>1 \textrm{ and }\sigma_i\sigma_{i+1}\sigma_i=
\sigma_{i+1}\sigma_i\sigma_{i+1}>
$$
Let $B_n^+$ denote the braid monoid on $n$ strings defined as a monoid
by this presentation.

We define 
$\mathcal{C}at$-operads
 $\mathcal{B}r$ and $\mathcal{B}r^+$ as follows: 
$ob\mathcal{B}r(n) = \Sigma_n$
and the morphisms from $\sigma$ to $\pi$ in $\mathcal{B}r(n)$ are the elements
$b\in B_n$ such that $p(b)\circ \sigma = \pi$. Composition is the 
obvious one. The action of $\sigma\in\Sigma_n$ on
$\mathcal{B}r(n)$ sends an object $\pi$ to $\pi\circ \sigma$ and a morphism
$b:\pi_1\to \pi_2$ to $b:\pi_1\circ \sigma \to \pi_2\circ \sigma$. The operad composition
is defined as in $\widetilde{\Sigma}$. The definition of $\mathcal{B}r^+$ 
is the same with $B_n$ replaced by $B_n^+$ throughout and $p$ replaced by
the composition with the natural map $u: B_n^+ \to B_n$.
We leave it as an exercise to the reader to check that 
$\mathcal{C}at^{\mathcal{B}r}$ is the category of braided monoidal 
categories and $\mathcal{C}at^{\mathcal{B}r^+}$ the category of weak
braided monoidal categories and that our notion of lax morphisms of these algebras
coincide with the ones in the literature. 

\begin{lem}\label{8_12}
 The operad $\mathcal{B}r$ satisfies the factorization condition \ref{6_7}.
\end{lem}
\begin{proof}
Since any two objects in $\mathcal{B}r(n)$ are isomorphic $\cC(\sigma,\pi,r_1,\ldots,r_n)$ is
isomorphic to $\cC(e_m,e_n,r_1,\ldots,r_n)$. An object $(\gamma_1,\ldots,\gamma_n;b)$ in
$\cC(e_m,e_n,r_1,\ldots,r_n)$ consists of permutations $\gamma_i\in\Sigma_{r_i}$ and a braid
$b\in B_m$ with $p(b)=\gamma_1\oplus\ldots\oplus\gamma_n$. Each object $(\gamma_1,\ldots,\gamma_n;b)$
is initial in its connected component. For if $(\delta_1,\ldots,\delta_n;b')$ is another object in its component 
there is a sequence of morphisms 
$$(b_1^i,\ldots,b_n^i)\in \mathcal{B}r(r_1)\times\ldots\times\mathcal{B}r(r_n)\subset \mathcal{B}r(m)
$$
connecting the two. Since $\mathcal{B}r(r_1)\times\ldots\times\mathcal{B}r(r_n)$ is a group, each of these morphisms 
is an isomorphism. Hence there is a morphism
$$
(b_1,\ldots,b_n):(\gamma_1,\ldots,\gamma_n;b) \to (\delta_1,\ldots,\delta_n;b').
$$ Since $(b_1\oplus\ldots\oplus b_n)\circ b = b'$ and $\mathcal{B}r(m)$ is a group, this
morphism is unique.
\end{proof}

By \cite[Remark 1.5]{BFSV} there are 
morphisms of operads
$$\mathcal{M}_2 \stackrel{\beta}{\rightarrow} \mathcal{B}r^+
\stackrel{\mu}{\rightarrow} \mathcal{B}r.
$$
The morphism $\beta$ sends an object 
$ i_1\Box_{r_1} i_2\Box_{r_2}\ldots \Box_{r_{n-1}} i_n$
of $\mathcal{M}_2 (n)$ with any form of bracketing to the
permutation in $\Sigma_n$ sending $k$ to $i_k$. A morphism
$\alpha\to \beta$ is sent to the braid whose strands join each generator $i$ in $\alpha$
with $i$ in $\beta$. If a generator $a$ is left of a generator $b$ in $\alpha$ but right of
$b$ in $\beta$ (which implies that $\alpha \cap \{ a,b\}=a\square_1 b$ and 
$\beta \cap \{ a,b\}=b\square_2 a$) then the strand connecting the generators $b$ is under
the strand connecting the generators $a$.

\begin{prop}\label{8_13}
There are topological operads $\mathcal{D}^+$ and $\mathcal{D}$
and a commutative diagram of morphisms of operads
$$
\xymatrix{
B(\mathcal{B}r^+) \ar[dd]_{B\mu} && \mathcal{D}^+
\ar[ll] \ar[drr] \ar[dd] && \\
&&&& \mathcal{C}_2 \\
B(\mathcal{B}r) && \mathcal{D}
\ar[ll] \ar[urr] && 
}
$$
which are equivalences.
\end{prop}

We postpone the proof.

Like in the symmetric monoidal category case we get

\begin{theo}\label{8_14}
The classifying space functor and the change of operads functor induce
equivalences of categories
$$
\mathcal{C}at^{\mathcal{B}r}[\widetilde{{\we}^{-1}}]\simeq 
\mathcal{T}op^{B(\mathcal{B}r)}[{\we}^{-1}]\simeq
\mathcal{T}op^{B(\mathcal{B}r^+)}[{\we}^{-1}]\simeq
\mathcal{T}op^{\mathcal{C}_2}[{\we}^{-1}]
$$
\end{theo}

Since we do not know whether or not $\mathcal{B}r^+$ has the homotopy colimit property we do
not have the analogous result for $\mathcal{C}at^{\mathcal{B}r^+}$.

The proof that $B\mu$ is an equivalence uses an observation which might
be of separate interest.

\begin{defi}\label{8_15}
Let $S$ be a finite set.
A \textit{Coxeter matrix} $M=(m_{s,t})$ is a symmetric
$S\times S$-matrix with entries in $\mathbb{N}\cup \{\infty\}$ such
that $m_{s,s}=1$ and $m_{s,t}\geq 2$ for $s\neq t$.
A Coxeter matrix has an associated \textit{Artin monoid} $G^+$ and an
associated \textit{Artin group} $G$ both given by the presentation
$$
<S\;|\produ(m_{s,t};s,t)=\produ(m_{s,t};t,s) \textrm{ for } s\neq t,\;
m_{s,t}<\infty>,
$$
where $\produ(n;x,y)=xyxy\ldots$ ($n$ factors). [So $\produ(2n;x,y)=(xy)^n$ and
$\produ(2n+1;x,y)=(xy)^nx$].\\
We call $G^+$ and $G$ \textit{spherical}, if the associated \textit{Coxeter
group}, obtained from the Artin group by taking the quotient by the additional
relations ${\{s^2=1|s\in S\}}$, is finite.
\end{defi}

Note that the braid monoid is a spherical Artin monoid.

\begin{prop}\label{8_16}
Let $G^+$ be a spherical Artin monoid with Artin group $G$. Then the 
canonical homomorphism $u: G^+\to G$ is injective (see 
\cite[(4.14)]{Del}) and induces a homotopy equivalence
$$
Bu: B(G^+)\simeq B(G)
$$
\end{prop}
\begin{proof}
 For a discrete monoid $M$ and
its algebraic group completion $UM$ the canonical map $u:M\to UM$
always induces an isomorphism $u_\ast:\pi_1(BM)\to \pi_1(BUM)$.
Now let $G^+$ be a spherical Artin monoid with associated Artin group
$G=UG^+$. By \cite[(4.17)]{Del}, we obtain $G$ from $G^+$ by inverting
a certain central element. Since $u$ is also injective \cite[(4.14)]{Del},
it induces an isomorphism
$$
H_\ast(G^+,A)\to H_\ast(G,A)
$$
for any $\mathbb{Z}[G]$-module $A$ by \cite[(X.4.1)]{CE}. 
Consequently, $B(G^+)\simeq B(G)$.
\end{proof}

It remains to prove Proposition \ref{8_13}.

A \textit{braided operad} is defined in the same way as an operad with the
symmetric group replaced by the braid group. To define composition, we
use the homomorphism $p: B_n\to \Sigma_n$. Examples of braided operads
in $\mathcal{C}at$ are
the translation operad $\widetilde{\mathcal{B}}$ of the braid groups and
$\widetilde{\mathcal{B}}^+$, for which $\widetilde{\mathcal{B}}^+(n)$
has the same objects as $\widetilde{\mathcal{B}}(n)$, but the morphisms are
elements of the braid monoid $B_n^+$ rather than the braid group
$B_n$. The inclusions $u:B_n^+ \to B_n$ define
a morphism of braided $\mathcal{C}at$-operads 
$$ \widetilde{\mathcal{B}}^+\rightarrow
\widetilde{\mathcal{B}}
$$
The operad $\widetilde{\mathcal{C}_2}$, where $\widetilde{\mathcal{C}_2}(n)$,
is the universal cover of $\mathcal{C}_2(n)$ is a braided topological
operad, for details see \cite{Fied}.
We have a commutative diagram of braided topological operads
$$
\xymatrix{
B(\widetilde{\mathcal{B}}^+) \ar[dd] && B(\widetilde{\mathcal{B}}^+)
\times \widetilde{\mathcal{C}_2}
\ar[ll] \ar[drr] \ar[dd] && \\
&&&& \widetilde{\mathcal{C}_2} \\
B(\widetilde{\mathcal{B}}) && B(\widetilde{\mathcal{B}})
\times \widetilde{\mathcal{C}_2} \ar[ll] \ar[urr] && 
}
$$
The morphisms are all equivalences, because the $n$-th space of each
operad is contractible for all $n$: this is clear for 
$B(\widetilde{\mathcal{B}})(n)$, for $B(\widetilde{\mathcal{B}}^+)(n)$
this follows from (\ref{8_13}),
 because $B(\widetilde{\mathcal{B}}^+)(n)$
is the universal cover of $B(B_n^+)$ (see proof of \cite[(4.4)]{Fied2}),
for $\widetilde{\mathcal{C}_2}(n)$ this follows from the well-known fact
that $\mathcal{C}_2(n)\simeq K(P_n,1)$, where $P_n$ is the pure braid 
group.

Given a braided operad $\mathcal{P}$, we can ``debraid'' it 
to obtain a genuine operad by replacing
$\mathcal{P}(n)$ by $\mathcal{P}(n)\times_{B_n}\Sigma_n$. 
Note that $B(\mathcal{B}r^+)$, $B(\mathcal{B}r)$, and
$\mathcal{C}_2$ are the debraidings of
$B(\widetilde{\mathcal{B}}^+)$, $B(\widetilde{\mathcal{B}})$, and
$\widetilde{\mathcal{C}_2}$.

The debraiding of the above diagram gives us the diagram required in \ref{8_10}
where $\mathcal{D}^+$ and $\mathcal{D}$ are the debraidings of
$B(\widetilde{\mathcal{B}}^+) \times \widetilde{\mathcal{C}_2}$
and $B(\widetilde{\mathcal{B}}) \times \widetilde{\mathcal{C}_2}$ respectively.

\vspace{2ex}
\textbf{Iterated loop spaces}

It is well-known that the group completion a $\cC_k$-algebra is weakly equivalent to a $k$-fold loop space. 
So the group completion of the classifying space of a $k$-fold monoidal, a braided monoidal, respectively
a permutative category is weakly equivalent to a $k$-fold, $2$-fold respectively infinite loop space.

Since the notion ``group completion of a $\cC_k$-algebra'' is ambiguous we will give an explicit reformulation of
this statement. We distinguish between loop spaces,
$k$-fold loop spaces with $1<k<\infty$, and infinite loop spaces.

\textbf{Loop spaces:} Let $\cA$ be a monoidal category, i.e. an $\cM_1$-algebra. Then its classifying space $B\cA$ is a monoid.

\begin{defi}\label{8_17}
 A homotopy associative $H$-space is called \textit{grouplike} if it has homotopy inverses, i.e. it is a group object in the category
 of spaces and homotopy classes of maps.
\end{defi}

A homotopy associative $H$-space $X$ of the homotopy type of a CW-complex is grouplike if and only if its $H$-structure induces a group
structure on $\pi_0(X)$ \cite[(12.7)]{DKP}.

Since loop spaces are grouplike we have to replace $B\cA$ by a grouplike space via a group completion functor inside $\Mon$, the category of well-based 
topological monoids and homomorphisms. 

Let $M$ be a well-based monoid and let $\Omega_M X$ be the Moore loop space of a well-based space $X$. The canonical map
$$ \iota(M): M\to \Omega_M BM $$
is a homotopy homomorphism whose underlying map is a strong (i.e. genuine) homotopy equivalence
if $M$ is grouplike. There is a functorial diagram in $\Mon$, which we will call \textit{group completion} of $M$ in $\Mon$,
$$ 
\xymatrix{M && U(M) \ar[ll]_{r(M)}\ar[rr]^{u(M)} && \Omega_M BM }
$$ 
such that the underlying map of $r(M)$ is a strong homotopy equivalence and $\iota(M)\circ r(M)\simeq u(M)$ in $\Top$. 
This is justified by a  homotopical universal property. Let $\we_s$ be the class of those homomorphisms in $\Mon$ which are strong homotopy
equivalences of underlying spaces. The localization $\Mon[\we_s^{-1}]$ exists and the group completion defines a natural morphism $\mu (M): M\to \Omega_M BM$ in $\Mon[\we_s^{-1}]$
with the following universal property:
 given a morphism $g:M\to N$ in $\Mon[\we_s^{-1}]$ into a grouplike monoid $N$ there is a unique morphism $\bar{g}:\Omega_M BM \to N$ in
$\Mon[\we_s^{-1}]$ such that $\bar{g}\circ \mu (M)=g$ (for details see \cite[Proposition 5.5]{Vogt2}).

\begin{leer}\label{8_18}
 \textbf{Notational conventions:} Let $\Omega_M\Top$ and $\Omega^k\Top$ denote the categories of Moore loop spaces and Moore loop maps, respectively
 $k$-fold loop spaces and $k$-fold loop maps,  and let $\Mon_{gl}\subset \Mon$ denote the subcategory of grouplike
 monoids. The existence of $\Mon_{gl}[\we^{-1}]$ follows from the existence of $\Mon[\we^{-1}]$: apply first the standard functorial $CW$-approximation
 and then the group completion construction to any sequence of broken arrows in $\Mon$ starting and ending at grouplike objects to obtain a sequence of broken
 arrows in $\Mon_{gl}$ representing the same morphism.
 We denote the full subcategory of $\Mon_{gl}[\we^{-1}]$ of Moore loop spaces by $\widetilde{\Omega_M\Top}$.\\
\end{leer}

Let $\we_g$ in $\Mon$ be the class of all morphisms $f$ for which
$\Omega_M Bf$ or, equivalently, $Bf$ is a weak equivalence, and let $\we_g$ in $\Cat^{\cM_1}$ be the class of those morphisms $f$ for which $Bf\in \Mon$ is
in $\we_g$. We will show later that $\Cat^{\cM_1} [\widetilde{\we_g^{-1}}]$ exists.

\begin{theo}\label{8_19}
 Let $\cA$ be a strict monoidal category, i.e. an $\cM_1$-algebra. Then the group completion $\Omega_M B(B\cA)$ of its classifying space $B\cA$ is
 homotopy equivalent to a loop space, namely to $\Omega B(B\cA)$. Conversely, given a loop space $\Omega Y$, then $\Omega_M Y$ is homotopy equivalent
 to $\Omega Y$ and there is a strict monoidal category $\cA_Y$ such that $B\cA_Y$ is weakly equivalent to $\Omega_M Y$. In particular, the classifying 
 space functor composed with the group completion induce an equivalence of categories
 $$ 
 \Cat^{\cM_1} [\widetilde{\we_g^{-1}}]\simeq \Mon_{gl}[we^{-1}] \simeq \widetilde{\Omega_M\Top}$$
 In the setting of Grothendieck universes we have
 $$\Cat^{\cM_1}[\we_g^{-1}]\simeq \Mon_{gl}[we^{-1}]\simeq \Omega_M\Top[we^{-1}] \simeq \Omega\Top[we^{-1}] $$
\end{theo}

\textbf{$k$-fold loop spaces, $1<k<\infty$:} For $k$-fold monoidal categories with $k\geq 2$ or braided monoidal categories we do not know of a 
group completion construction with an analogous universal property. 
Here we use variants of the homological group completion of May \cite{May1}.

\begin{defi}\label{8_20}
 An H-space $X$ is called admissible if $X$ is homotopy associative and if left translation by any given element of 
$X$ is homotopic to right translation with the same element. An H-map $g: X\to Y$ between admissible H-spaces 
is called a \textit{homological group completion of $X$} if $Y$ is grouplike and the unique morphism of $k$-algebras
$\bar{g}_{\ast}:H_{\ast}(X;k)\left[\pi_0^{-1}X\right]\to H_{\ast}(Y;k) $
which extends $g_{\ast}$ is an isomorphism for all commutative coefficient rings
$k$.
\end{defi}
If $\pi_0X$ is a group (in particular, if $X$ is grouplike) then  $g: X\to Y$ is a weak equivalence of spaces \cite[Remark 1.5]{May1}.

Let $\cM$ be either $\cM_k$ with $k\geq 2$ or $\cB r$. Then $N\cM$ is reduced, i.e $N\cM (0)$ is a point. Let $\cE'$ be 
a cofibrant replacement of $N\cM$  with respect to the model structure
on the category of simplicial operads defined by Berger and Moerdijk \cite{Berger}, and let $\cE$ be its reduced version obtained from $\cE'$
by identifying all nullary operations to a point. Consider the sequence of categories and functors
with $k=2$ if $\cM=\cB r$.
\begin{leer}\label{8_21}
$$\Cat^{\cM}\xrightarrow{N} \SSets^{N\cM} \rightleftarrows \SSets^{\cE} \rightleftarrows \Top^{|\cE|}\rightleftarrows \Top^{\cC_k}\supset \Omega^k\Top$$
\end{leer}
 The canonical $\cC_k$-structure on $\Omega^k X$ makes $\Omega^k\Top$
 a subcategory of $\Top^{\cC_k}$. 
Since the realization functor preserves quotients the reduced operad $|\cE|$ is a quotient of the cofibrant
operad $|\cE'|$. In particular, since $N\cM$ and $\cC_k$ are reduced, there are maps of topological operads 
$$ \xymatrix{
B\cM &&  |\cE|\ar[ll]_{F}\ar[rr]^{G}&& \cC_k }$$
which are equivalences.
The unspecified maps of the diagram are the Quillen equivalences given by the change of operads constructions and by the pair $|-|^{\alg}$ and 
$\Sing^{\alg}$. If $\cA$ is an $\cM$-algebra then $B\cA$ is an $|\cE|$-algebra by pulling back the $B\cM$-structure via $F$, and each 
$k$-fold loop space is an $|\cE|$-algebra by pulling back the $\cC_k$-structure via $G$. The following diagram (cp. \cite[Theorem 13.1]{May})
defines the \textit{group completion}
in the category $\Top^{|\cE|}$ of a $|\cE|$-algebra $X$:
$$\xymatrix{
X & B(\mE,\mE,X)\ar[l]_(.63){\varepsilon}\ar[rr]^(.45){B(\alpha_k\bar{G},\id,\id)} & & B(\Omega^k S^k,\mE,X)\ar[r]^{\gamma_k} & \Omega^kB(S^k,\mE,X)
}
$$
where $\bar{G}:\mE\to \mC_k$ is the morphism of monads associated with the map of operads $G:|\cE|\to \cC_k$. The map $\varepsilon$ is a 
strong deformation retraction, $\gamma_k$ is a weak equivalence. May's argument in the proof of \cite[Theorem 2.2]{May1}
together with the homology calculations of Cohen in \cite{Cohen} imply that $\gamma_k\circ B(\alpha_k\bar{G},\id,\id)$ is a homological
group completion.

The classifying space functor, respectively the realization functor composed with the group completion define functors
$$\Cat^{\cM}\to \Omega^k\Top\qquad \SSets^{\cE}\to \Omega^k\Top.$$
Like above we define $\we_g$ to be the classes of those morphisms in $\Cat^{\cM},\ \SSets^{\cE}$ or $\Top^{|\cE|}$ which are mapped to weak
equivalences in $\Omega^k\Top.$ By \cite[Remark 1.5]{May1}
$\gamma_k\circ B(\alpha_k\bar{G},\id,\id)$ of the above diagram is in $\we_g$. 
We again denote by $\Top^{|\cE|}_{gl}$ the full subcategory of grouplike objects. By the argument above, $\Top^{|\cE|}_{gl}[\we^{-1}]$
exists. We define 
$\widetilde{\Omega^k\Top}$ to be the full subcategory of $\Top^{|\cE|}_{gl}$ of $k$-fold loop spaces. The existence of
$\Cat^{\cM}[\widetilde{\we_g^{-1}}]$ is shown below.
  
 \begin{theo}\label{8_22} 
Let $\cA$ be a $k$-fold monoidal respectively braided monoidal category. Then the group completion of $B\cA$ is the $k$-fold loop space $ \Omega^kB(S^k,\mE,B\cA)$
respectively the double loop space $ \Omega^2B(S^2,\mE,B\cA)$.
Conversely, if $\Omega^kY$ is a $k$-fold loop space respectively a double loop space there is a $k$-fold monoidal category respectively braided monoidal
category $\cA$ such that $B\cA$ is weakly equivalent to $\Omega^kY$ respectively $\Omega^2 Y$
in $\Top^{|\cE|}$. In particular, the classifying 
 space functor composed with the group completion induce equivalences of categories
 $$ \begin{array}{rcccl}
 \Cat^{\cM_k} [\widetilde{\we_g^{-1}}]&\simeq & \Top^{|\cE|}_{gl}[\we^{-1}]&\simeq & \widetilde{\Omega^k\Top} \\
 \Cat^{\cB r} [\widetilde{\we_g^{-1}}]&\simeq & \Top^{|\cE|}_{gl}[\we^{-1}]& \simeq &\widetilde{\Omega^2\Top}
 \end{array}
 $$
 In the setting of Grothendieck universes we have
 $$ \begin{array}{rcccl}
 \Cat^{\cM_k} [\we_g^{-1}] &\simeq & \Top^{|\cE|}_{gl}[\we^{-1}] &\simeq & \Omega^k\Top[we^{-1}] \\
 \Cat^{\cB r} [\we_g^{-1}] &\simeq &\Top^{|\cE|}_{gl}[\we^{-1}] &\simeq &  \Omega^2\Top[we^{-1}]
 \end{array}
 $$
\end{theo}

\textbf{Infinite loop spaces:} 
Throughout this subsection let $\cM$ be either the operad $\cM_\infty$ or $\widetilde{\Sigma}$. Let $\Omega^\infty\Top$ denote the
category of infinite loop sequences. Its objects are sequences $X=(X_i;\ i\geq 0)$ of based spaces such that $X_i=\Omega X_{i+1}$, its morphisms
are sequences $g=(g_i;\ i\geq 0)$ of based maps such that $g_i=\Omega g_{i+1}$. If $X\in \Omega^\infty\Top$ then $X_0$ has a canonical
$\cC_\infty$ structure. We hence have a functor $V: \Omega^\infty\Top \to \Top^{\cC_\infty}$.

We again use a variant of May's group completion construction \cite{May1}.

We apply the simplicial Boardman-Vogt $W$-construction \cite{Berger2} to the sequence of simplicial operads $N\cM_1 \subset N\cM_2\subset N\cM_3 \subset\ldots$.
By \cite[Section 5]{Berger2} the colimit $\colim W\cM_n$ is cofibrant. Let $\cE$ be the reduced version in the above sense of $\colim W\cM_n$ and let
$\cE_n\subset \cE$ be the reduced version of $W\cM_n$. We have categories and functors like in Diagram \ref{8_21}
with $k$ replaced by $\infty$. Since
$\cM_\infty,  \widetilde{\Sigma},\ \cC_n$ and  $\cC_\infty$ are reduced we have morphisms of operads
$$\xymatrix{B\cM& |\cE_n|\ar[l]_{F_n}\ar[r]^{G_n}& \cC_n & & B\cM& |\cE|\ar[l]_{F}\ar[r]^{G}& \cC_\infty}
$$
and for $X$ in $\Top^{|\cE|}$ we have maps of $|\cE_n|$-algebras (cp. \cite[p. 68]{May1})
$$\xymatrix{
X & B(\mE_n,\mE_n,X)\ar[l]_(.63){\varepsilon}\ar[rr]^(.46){B(\alpha_n\bar{G_n},\id,\id)} & & B(\Omega^n S^n,\mE_n,X)\ar[r]^{\gamma_n} & \Omega^nB(S^n,\mE_n,X)
}
$$
where again $\bar{G}:\mE_n\to \mC_n$ is the morphism of monads associated with the morphism $|G_n|$ of operads. We define a functor
$$
B_{\infty}:\Top^{|\cE|}\to \Omega^\infty\Top, \quad X\mapsto (B_iX=\colim \Omega^nB(S^{i+n},\mE_{i+n},X),\ i\geq 0)
$$
Passing to colimits we obtain the \textit{group completion} of $|\cE|$-algebras
$$\xymatrix{
X & B(\mE,\mE,X)\ar[l]_{\varepsilon}\ar[rr]^(.45){B(\alpha_\infty\bar{G},\id,\id)} & & B(\Omega^\infty S^\infty,\mE,X)\ar[r]^(.65){\gamma_\infty} & B_0X
}
$$
where $\varepsilon$ is a deformation retraction, $B(\alpha_\infty\bar{G},\id,\id)$ is a homological group completion, and $\gamma_\infty$ is a weak
equivalence \cite[Theorem 2.3]{May1}.

As before we define $\we_g$ to be the class of those morphisms in $\Cat^{\cM},\ SSets^{\cE}$ and $\Top^{|\cE|}$ which are mapped to weak equivalences
in $\Omega^\infty\Top$. We use notation analogous to the one in Theorem \ref{8_22}.

\begin{theo}\label{8_23}
 Let $\cA$ be a infinite monoidal respectively permutative category. Then the group completion of $B\cA$ is the $0$-th space of the
 infinite loop space $B_\infty(B\cA)$.
Conversely, if $Y$ is an infinite loop space there is a permutative 
category (and hence an infinite monoidal category) $\cA$ such that $B\cA$ is weakly equivalent to the $0$-th space of $Y$.
In particular, the classifying 
 space functor composed with the group completion induce equivalences of categories
 $$ \begin{array}{rcccl}
 \Cat^{\cM_\infty} [\widetilde{\we_g^{-1}}]&\simeq & \Top^{|\cE|}_{gl}[\we^{-1}]&\simeq & \widetilde{\Omega^\infty\Top} \\
 \Cat^{\widetilde{\Sigma}} [\widetilde{\we_g^{-1}}]&\simeq & \Top^{|\cE|}_{gl}[\we^{-1}]& \simeq &\widetilde{\Omega^\infty\Top}
 \end{array}
 $$
 In the setting of Grothendieck universes we have
 $$ \begin{array}{rcccl}
 \Cat^{\cM_\infty} [\we_g^{-1}] &\simeq & \Top^{|\cE|}_{gl}[\we^{-1}] &\simeq & \Omega^\infty\Top[we^{-1}] \\
 \Cat^{\widetilde{\Sigma}} [\we_g^{-1}] &\simeq &\Top^{|\cE|}_{gl}[\we^{-1}] &\simeq &  \Omega^\infty\Top[we^{-1}]
 \end{array}
 $$
\end{theo}

\begin{rema}\label{8_24}
Theorem \ref{8_23} for permutative categories in the setting of Gro\-thendieck universes is the main result of Thomason's paper \cite{Thom2}.
Our result is stronger in so far as it explains what
happens before group completion. Mandell has obtained a similar result by completely
different means \cite{Mandell}.
\end{rema}

It remains to prove the existence of the localizations up to equivalence with respect to $\we_g$
occuring in \ref{8_19}, \ref{8_22} and \ref{8_23}.

\begin{lem}\label{8_25} Let $\cM$ be one of the operads $\cB r,\ \widetilde{\Sigma}$ or $\cM_k,\ 1\leq k\leq \infty$. Then
 there is a localization up to equivalence $\delta:\Cat^{\cM}\to\Cat^{\cM}[\widetilde{\we_g^{-1}}]$.
\end{lem}
\begin{proof}
 Let $\lambda:\cE\to N\cM$ be a cofibrant replacement of $N\cM$. In \cite{Stelzer} the second author proved that 
 $\SSets^{\cE}$ has a Quillen model structure with $\we_g$ as class of weak equivalences and that $\we\subset \we_g$. In particular,
 the localization $\gamma: \SSets^{\cE}\to \SSets^{\cE}[\we_g^{-1}]$ exists. Since $\lambda:\cE\to N\cM$ is a weak equivalence of operads,
 there is a Quillen equivalence $V: \SSets^{N\cM}\rightleftarrows \SSets^{\cE}:U$. Combining this with the nerve functor $N: \Cat^{\cM}
 \to \SSets^{N\cM}$ and the functor $F:\SSets^{N\cM}\to \Cat^{\cM}$ constructed in \ref{C_2} we obtain a pair of functors
 $$ N':\Cat^{\cM} \to \SSets^{\cE}\qquad\textrm{and}\qquad F':\SSets^{\cE}\to \Cat^{\cM}$$
 and by \ref{C_2} and \ref{C_3} functorial sequences of weak equivalences joining $\Id$ and $N'\circ F'$ respectively $\Id$ and $F'\circ N'$.
 Since $\we\subset \we_g$ the
 assumptions of Proposition \ref{C_4} are satisfied. Hence the required localization up to equivalence exists.
\end{proof}

\makeatletter
\renewcommand\appendix{\par
  \setcounter{section}{0}%
  \setcounter{subsection}{0}%
  \setcounter{equation}{0}
  \gdef\thefigure{\@Alph\c@section.\arabic{figure}}%
  \gdef\thetable{\@Alph\c@section.\arabic{table}}%
  \gdef\thesection{\@Alph\c@section}
  \@addtoreset{equation}{section}%
  \gdef\theequation{\@Alph\c@section.\arabic{equation}}%
  \section*{Appendix}    
}
\makeatother

\begin{appendix}
\section{An  explicit description of $\Func(\mathcal{L},\Cat^\mathbb{M}\Lax)$ }

\textbf{Objects:} Objects in $\Func(\mathcal{L},\Cat^\mathbb{M}\Lax)$
 are strict functors
$$
X: \mathcal{L}\to \Cat^\mathbb{M}\Lax.
$$
So for $\lambda:L_0\to L_1$ in $\mathcal{L}$ we have a morphism
$$
(X\lambda,\overline{X\lambda}):XL_0\to XL_1
$$
depicted in the square
$$
\xymatrix{\ar @{} [drr] |{\stackrel{\overline{X\lambda}}{\Longleftarrow}}
\mathbb{M}XL_0 \ar[rr]^{\mathbb{M}X\lambda} \ar[d]_{\xi_{XL_0}} &&
\mathbb{M}XL_1 \ar[d]^{\xi_{XL_1}}
\\
XL_0 \ar[rr]_{X\lambda} &&
XL_1
}
$$
satisfying
\begin{enumerate}
\item $\overline{X\lambda}\circ_1 \mu (XL_0)=(\overline{X\lambda}\circ_1\mM\xi_{XL_0})\circ_2 (\xi_{XL_1}\circ_1
\mM \overline{X\lambda})$
\item $\overline{X\lambda}\circ_1 \iota(XL_0)=\id_{X\lambda}$
\end{enumerate}
where $\xi_{XL_0}$ and $\xi_{XL_1}$ are the structure maps of the $\mathcal{M}$-algebras $XL_0$ and $XL_1$ and
$\mu$ and $\iota$ are the multiplication and unit of the monad $\mM$.

\textbf{Morphisms:} A morphism $j:X\to Y$ in
$\Func(\mathcal{L},\Cat^\mathbb{M}\Lax)$ is a lax natural transformation.

Given $L_0\stackrel{\lambda_0}{\rightarrow} L_1\stackrel{\lambda_1}{\rightarrow} L_2$, we have a diagram
$$
\xymatrix{\ar @{} [drr] |{\stackrel{j_{\lambda_0}}{\Longrightarrow}}
XL_0 \ar[rr]^{X{\lambda_0}} \ar[d]^{j_{L_0}} 
&& \ar @{} [drr] |{\stackrel{j_{\lambda_1}}{\Longrightarrow}}
XL_1 \ar[rr]^{X{\lambda_1}} \ar[d]^{j_{L_1}} 
&& 
XL_2  \ar[d]^{j_{L_2}}
\\
YL_0 \ar[rr]_{Y{\lambda_0}}
&&
YL_1 \ar[rr]_{Y{\lambda_1}}
&&
YL_2
}
$$
such that

$$
j_{\lambda_1\circ\lambda_2}=(j_{\lambda_1}\circ_1 X\lambda_0)\circ_2(Y\lambda_1\circ_1 j_{\lambda_0})
$$
Moreover, $j_{\lambda_0}:Y\lambda_0\circ j_{L_0}\Rightarrow j_{L_1}\circ X\lambda_0$ (and accordingly $j_{\lambda_1}$)
has to satisfy
$$
(j_{\lambda_0}\circ_1\xi_{XL_0})\circ_2(\overline{Y\lambda_0\circ j_{L_0}})= (\overline{j_{L_1}\circ X\lambda_0}) 
\circ_2(\xi_{YL_1}\circ_1\mathbb{M}j_{\lambda_0})
$$
where
$$
\begin{array}{rcl}
(\overline{Y\lambda_0\circ j_{L_0}}) &=& (Y\lambda_0\circ_1 \overline{j_{L_0}})\circ_2 
(\overline{Y\lambda_0}\circ_1 \mathbb{M}j_{L_0})\\
(\overline{j_{L_1}\circ X\lambda_0}) &=& (j_{L_1}\circ_1\overline{X\lambda_0}) \circ_2 (\overline{j_{L_1}}\circ_1 \mathbb{M}X\lambda_0)
\end{array}
$$

\textbf{2-cells:} 
$\xymatrix{
X\ar@/_1pc/[rr]_{k} \ar@/^1pc/[rr]^{j}
& {\Downarrow}^s
&
Y
}$ 
consists of 2-cells $s_L:j_L\Rightarrow k_L$ in $\Cat^{\mathbb{M}}\Lax$, one for each object $L$ of $\mathcal{L}$ such that
$$
k_\lambda\circ_2(Y\lambda\circ_1 s_{L_0})=(s_{L_1}\circ_1 X\lambda)\circ_2 j_\lambda
$$
for each morphism $\lambda:L_0\to L_1$ in $\mathcal{L}$. As a 2-cell in $\Cat^{\mathbb{M}}\Lax$ the natural transformation $s_L$ has to satisfy
$$
\overline{k}_L\circ_2(\xi_{YL}\circ_1\mathbb{M}s_L)=(s_L\circ_1\xi_{XL})\circ_2 \overline{j}_L.
$$

\section{Proof of Lemma \ref{6_17}}
Evaluation defines maps $\varepsilon: \mM_nX\to X$ for each $n\geq 0$ -- we use the same symbol $\varepsilon$ for all of them. In the same manner
we denote all natural transformations $\iota\mM^n : \mM^n\to \mM^{n+1}$, where $\iota :\Id \to \mM$ is the unit of the monad $\mM$, by $\iota$.

To extend $G:\Delta^\op\to\Cat$ together with the augmentation $\varepsilon:G([0])\to G([-1])$ to a lax functor $G_+:\Delta^\op_+\to\Cat$ we have to define functors
$$
s_{-1}:\varepsilon_n/T[P,\bar{L}]\to\varepsilon_{n+1}/T[P,\bar{L}]\quad n=-1,0,1,\ldots
$$
satisfying the simplicial identities \ref{6_13} up to coherent natural transformations.

An object $(Z;\alpha)$ of $\varepsilon_n/T[P,\bar{L}]$ is a pair consisting of an object $Z=[A;(K_i,L_i)_{i=1}^m]$ in
$\hocolim^{\mathbb{M}}\mathbb{M}_nX$ and a morphism
$$
\alpha:[A;(\varepsilon K_i,L_i)_{i=1}^m]\to T[P,\bar{L}]
$$
in $\hocolim^{\mathbb{M}}X$. To define the functors $s_{-1}$ we choose a representative in each $\Sigma_n$-orbit of the objects 
of $\cM(n)$ and use these representatives to obtain ``canonical'' representatives for the objects in our homotopy colimits. We now choose
for each object $(Z;\alpha)$ a representative $(\varphi,(C_j)_{j=1}^r,\gamma,(\lambda_i)_{i=1}^m,(f_j)_{j=1}^r)$ in the notation of \ref{4_1}
for $\alpha$ such that
\begin{itemize}
 \item each $\varphi^{-1}(k)$ has the natural order, i.e. the one inherited from $1<2<\ldots <m$.
 \item if $\gamma :A\to T\ast(C_1\oplus \ldots\oplus C_r)\cdot \sigma$, then $(C_1 , \ldots , C_r; \gamma\cdot \sigma^{-1})$ is an initial
 object in its connected component of the factorization category $\cC(A\cdot \sigma^{-1},T,|\varphi^{-1}(1)|,\ldots,|\varphi^{-1}(r)|)$.
\end{itemize}
Recall that $\sigma^{-1}$ is the underlying permutation of the unique order preserving map of $\underline{m}=\{1<\ldots <m \}$ into the ordered
disjoint union $\varphi^{-1}(1)\sqcup\ldots\sqcup\varphi^{-1}(r)$.

Once these choices are made we have a well-defined decomposition of $\alpha$
$$
[A;(\varepsilon K_i,L_i)_{i=1}^m]\xrightarrow{\overline{\alpha}} [T;(C_j(\lambda\varepsilon K(\varphi^{-1}(j)),\bar{L}_j)_{j=1}^r]
\xrightarrow{(f_j)_{j=1}^r} T[P,\bar{L}]
$$
(we use the notation of \ref{4_1}) with $\overline{\alpha}$ represented by $(\varphi,(C_j)_{j=1}^r,\gamma,(\lambda_i)_{i=1}^m,(\id)_{j=1}^r)$.
We define
$$
s_{-1}(Z;\alpha)=([T;(\iota(C_j(\lambda K(\varphi^{-1}(j))),\bar{L}_j)_{j=1}^r];(f_j)_{j=1}^r).
$$
Now consider two morphisms
$$
(Z_1;\alpha_1)\xrightarrow{\beta_1} (Z_2;\alpha_2)\xrightarrow{\beta_2} (Z_3;\alpha_3)
$$
in $\varepsilon_n/T[P,\bar{L}]$. Then $\beta_i$ is a morphism in $\hocolim^{\mathbb{M}}\mathbb{M}_nX$ such that $\alpha_2\circ \varepsilon(\beta_1)=\alpha_1$
and $\alpha_3\circ \varepsilon(\beta_2)=\alpha_2$. Let $(\varphi_k,(C_j^k)_{j=1}^r,\gamma_k,(\lambda_i^k)_{i=1}^{m_k},(f_j^k)_{j=1}^r),\ k=1,2,3,$  be the chosen
representative of $\alpha_k$ with $ \gamma_k :A_k\to T\ast (C^k_1\oplus\ldots\oplus C^k_r)\cdot \sigma_k$.

We represent $\beta_k, \ k=1,2,$ by $(\psi_k,(B^k_j)_{j=1}^{m_{k+1}},\delta_k,(\rho^k_i)_{i=1}^{m_k},(g^k_j)_{j=1}^{m_{k+1}})$
where each $\psi_k^{-1}(l)$ has the natural order and $\delta_k: A_k\to A_{k+1}\ast (B^k_1\oplus\ldots\oplus B^k_{m_{k+1}})\cdot \tau_k$. 
From this we obtain a representative of $\varepsilon(\beta_k)$ by replacing
$(g^k_j)_{j=1}^{m_{k+1}}$ by $(\varepsilon(g^k_j))_{j=1}^{m_{k+1}}$.
Since we have chosen representatives of the objects, these representatives are uniquely
determined by $\beta_k$ up to Relation \ref{4_8}.3. 

The composite $\alpha_2\circ \varepsilon(\beta_1)$ is represented by $(\varphi_2\circ \psi_1, (E^1_j)_{j=1}^r,(\gamma_2\ast \id)\circ \delta_1,(\mu^1_i)_{i=1}^{m_1},
\linebreak
(h^1_j)_{j=1}^r)$
where
$$
E^1_j=C^2_j\ast\bigoplus_{i\in \varphi_2^{-1}(j)}B^1_i \qquad\qquad \mu^1_i=\lambda^2_{\psi(i)}\circ \rho^1_i,
$$
$h^1_j$ is the composite
$$
\xymatrix{
C^2_j((B^1_i(\mu\varepsilon K^1(\psi_1^{-1}(i))))_{i\in \varphi_2^{-1}(j)}) \ar[rrr]^(.67){h^1_j} \ar[d]_{C^2_j((\overline{\rho}_i)_{i\in  \varphi_2^{-1}(j)})}
&&& P_j & (\ast)\\
C^2_j((\lambda^2_i B^1_i(\rho \varepsilon K^1(\psi_1^{-1}(i))))_{i\in \varphi_2^{-1}(j)})\ar[rrr]^(.58){C^2_j((\rho_i(\varepsilon(g^1_i))_{i\in \varphi_2^{-1}(j)})}
&& &
C^2_j(\lambda^2\varepsilon K^2(\varphi_2^{-1}(j)) \ar[u]^{f^2_j}
}
$$
and
$$
(\gamma_2\ast \id)\circ \delta_1:A_1\to (((T\ast \bigoplus_{j=1}^r C^2_j)\cdot \sigma_2)\ast \bigoplus_{i=1}^{m_2}B^1_i)\cdot\tau_1.
$$
Now

$(((T\ast \bigoplus_{j=1}^r C^2_j)\cdot \sigma_2)\ast \bigoplus_{i=1}^{m_2}B^1_i)\cdot\tau_1$ \\
$$=(T\ast \bigoplus_{j=1}^r( C^2_j\ast \bigoplus_{i\in \varphi_2^{-1}(j)}B^1_i)
\cdot \tau_{1j})\cdot \sigma_1\\
= (T\ast \bigoplus_{j=1}^r E^1_j\cdot\tau_{1j} )\cdot \sigma_1
$$
where $\tau_{1j}^{-1}$ is the underlying permutation of the order preserving map $\varphi_1^{-1}(j)\cong \coprod_{i\in \varphi_2^{-1}(j)} \psi_1^{-1}(i)$.
Hence $\alpha_2\circ \varepsilon(\beta_1)$ is also represented by $(\varphi_2\circ \psi_1, (E^1_j\cdot \tau_{1j})_{j=1}^r,
\linebreak
(\gamma_2\ast \id)\circ \delta_1,(\mu^1_i)_{i=1}^{m_1},
(h^1_j)_{j=1}^r)$.

Since $\alpha_1=\alpha_2\circ \varepsilon (\beta_1)$, this representative has to be related to the chosen representative for $\alpha_1$ by an iterated application of Relation
\ref{4_8}.3. In particular, 
$(C^1_1 , \ldots , C^1_r;\gamma_1\cdot \sigma_1^{-1})$ and $(E^1_1\cdot \tau_{11},\ldots , E^1_r\cdot \tau_{1r}; 
((\gamma_2\ast \id)\circ \delta_1)\cdot \sigma_1^{-1})$ 
are in the same connected component of the factorization category $\cC(A_1\cdot \sigma_1^{-1},T,|\varphi_1^{-1}(1)|,
\linebreak
\ldots,|\varphi_1^{-1}(r)|)$. 
Since $(C^1_1 , \ldots , C^1_r; \gamma_1\cdot \sigma_1^{-1})$ is initial in this component a single
application suffices. So
there are morphisms $u^1_j:C^1_j\to E^1_j\cdot \tau_{1j}$ such that the following diagram commutes
$$
 \xymatrix{[A_1;(\varepsilon K^1_i,L^1_i)_{i=1}^{m_1}]\ar[r]^(.4){\overline{\alpha}_1}\ar[dd]_{\varepsilon(\beta_1)}
 & [T;(C^1_j(\lambda^1\varepsilon K^1(\varphi_1^{-1}(j)),\bar{L}_j)_{j=1}^r]
 \ar[d]_{(u^1_j(\lambda^1\varepsilon K^1(\varphi_1^{-1}(j))_{j=1}^r}\ar[dr]^(.55){(f^1_j)_{j=1}^r} &(\ast\ast)\\
 & [T;(E^1_j\cdot \tau_{1j}(\lambda^1\varepsilon K^1(\varphi_1^{-1}(j)),\bar{L}_j)_{j=1}^r]\ar[d]_{(2)}\ar[r]^(.7){(1)} & T[P,\bar{L}]\\
 [A_2;(\varepsilon K^2_i,L^2_i)_{i=1}^{m_2}]\ar[r]^(.4){\overline{\alpha}_2} &
 [T;(C^2_j(\lambda^2\varepsilon K^2(\varphi_2^{-1}(j)),\bar{L}_j)_{j=1}^r]\ar[ru]_(.55){(f^2_j)_{j=1}^r}
 } 
$$
where (1) is the map $(h_j)_{j=1}^r$ and (2) is given by diagram $(\ast)$ (recall that \linebreak $E^1_j\cdot \tau_{1j}(\lambda^1\varepsilon K^1(\varphi_1^{-1}(j)))=
C^2_j((B^1_i(\mu\varepsilon K^1(\psi_1^{-1}(i))))_{i\in \varphi_2^{-1}(j)})$).

Recall that $s_{-1}(Z_i,\alpha_i)=
([T;(\iota(C^i_j(\lambda^1 K^i(\varphi_i^{-1}(j))),\bar{L}_j)_{j=1}^r];(f^i_j)_{j=1}^r),\ i=1,2.$
We define
$$
s_{-1}(\beta_1): s_{-1}(Z_1,\alpha_1)\to s_{-1}(Z_2,\alpha_2)
$$
to be induced by the composition of the two vertical maps on the right of diagram $(\ast\ast)$.

Clearly, $s_{-1}$ maps identities to identities. To prove that $s_{-1}(\beta_2\circ \beta_1)=s_{-1}(\beta_2)\circ s_{-1}(\beta_1)$ let
$$
(\varphi_3\circ \psi_2, (E^2_j\cdot \tau_{2j})_{j=1}^r,(\gamma_3\ast \id)\circ \delta_2,(\mu^2_i)_{i=1}^{m_2},
(h^2_j)_{j=1}^r)
$$
and
$$
(\varphi_3\circ \psi_2\circ \psi_1, (E^3_j\cdot \tau_{3j})_{j=1}^r,(\gamma_3\ast \id)\circ (\delta_2\ast \id)\circ \delta_1,(\mu^3_i)_{i=1}^{m_1},
(h^3_j)_{j=1}^r)
$$
be constructed as above and represent $\alpha_3\circ\beta_2$ respectively $ \alpha_3\circ\beta_2\circ \beta_1$. Let 
$$
u^2_j:C^2_j\to E^2_j\cdot \tau_{2j}\quad\textrm{and}\quad u^3_j:C^1_j\to E^3_j\cdot \tau_{3j}
$$ 
be defined like $u^1_j$. Then the shuffle maps $\tau_{ij},\ i=1,2,3,$ fit together to yield a commutative diagram
$$
\xymatrix{
C^1_j\ar[rr]^{u^1_j}\ar[rrd]^{u^3_j} && E^1_j\cdot \tau_{1j}=(C^2_j\ast\bigoplus_{i\in \varphi_2^{-1}(j)}B^1_i)\cdot \tau_{1j}\ar[d]\\
&&
E^3_j\cdot \tau_{3j}=(C^3_j\ast (\bigoplus_{k\in \varphi_3^{-1}(j)}(B^2_k)\ast \bigoplus_{i\in \psi_2^{-1}(k)}B^1_k))\cdot\tau_{3j}
}
$$
where the vertical map is induced by $u^2_j$. It is now straight forward to verify the functoriality of $s_{-1}$.

By construction, $s_{-1}$ satisfies all simplicial identities \ref{6_13} except for $d_0\circ s_{-1}=\id$. But as diagram $(\ast\ast)$ shows we  have a natural
transformation $\tau:\id\to d_0\circ s_{-1}$ given by
$$
\tau ([A; (K_i,L_i)_{i=1}^m],\alpha)=\overline{\alpha}.
$$
By inspection, $\tau$ has the following properties
$$
\begin{array}{rcll}
\tau\circ d_i &=& d_i\circ\tau & \textrm{ for all } i\\
\tau\circ s_i &=& s_i\circ\tau & \textrm{ for }i\geq 0\\
\tau\circ s_{-1} &=& \id_{s_{-1}}=s_{-1}\circ\tau
\end{array}\eqno(\ast\ast\ast)
$$
Given morphisms $g$ and $h$ in $\Delta^\op_+$ we present them in the generators in standard form
$$
\begin{array}{rcl}
g &=& s_{k_r}\circ\ldots\circ s_{k_1}\circ d_{l_p}\circ\ldots\circ d_{l_1}\\
h &=& s_{j_t}\circ\ldots\circ s_{j_1}\circ d_{i_s}\circ\ldots\circ d_{i_1}\\
\end{array}
$$
with $k_r>\ldots>k_1\ge -1$, $l_p>\ldots>l_1\ge 0$, $j_t>\ldots> j_1\ge-1$, $i_s>\ldots>i_1\ge 0$.
In the notation of \ref{2_8} we define
$$
\sigma(h,g):G_+(h\circ g)\to G_+(h)\circ G_+(g)
$$
by
$$
\sigma(h,g)=G_+(s_{j_t}\ldots s_{j_1} d_{i_s}\ldots d_{i_2})\circ_1\tau\circ_1 G_+(s_{k_r-1}\ldots s_{k_2-1} d_{l_p}\ldots d_{l_1})
$$
if $i_1=0$ and $k_1=-1$, and to be the identity otherwise. Here we use the relation
$$
s_{k_r}\circ\ldots\circ s_{k_2}\circ s_{-1}=s_{-1}\circ s_{k_r-1}\circ\ldots\circ s_{k_2-1}.
$$
Since $G_+$ preserves identities, we define $\rho([n])=\id_{G_+([n])}$. The equations $(\ast\ast\ast)$ for $\tau$ ensure that the coherence diagrams \ref{2_8} commute. 
\hfill\ensuremath{\Box}

\section{Proof of Proposition \ref{7_3}}

\begin{nota}\label{C_1}
 Let $\cB$ be a category with a subcategory $\we$ of weak equivalences. Let
 $F,\ G:\cA \to \cB$ be functors. We denote a sequence of functors and
 natural weak equivalences joining $F$ and $G$ by
 $$F \sim\sim\sim G$$
\end{nota}

In \cite[Section 5]{FV} we constructed a functor
$$
\hat{Q}_\bullet :\Top^{B\mM} \longrightarrow \scatm $$
for any monad associated with a $\Sigma$-free operad in $\Cat$.
Moreover we showed that if we restrict to the full subcategory of objects
in $\Top^{B\mM}$ whose underlying objects are CW-complexes there is a sequence 
of natural weak equivalences in $\stopbm$
$$
\const_T \sim\sim\sim B\hat{Q}_\bullet(-)
$$
where $\const_T:\Top^{B\mM}\to \stopbm$ is the constant simplicial object functor \cite[6.3]{FV}.
Applying the Quillen equivalence
$$
|-|^{\alg}:\SSets^{N\mathbb{M}}\leftrightarrows\Top^{B\mathbb{M}}\ :\Sing^{\alg}.
$$
to this sequence and adding the weak equivalences 
$$
K\to \Sing |K|\qquad \textrm{and} \qquad N\hat{Q}_\bullet(|K|)\to \Sing B\hat{Q}_\bullet(|K|)
$$
given by the unit of the adjunction, we obtain a sequence of natural weak equivalences in 
$\cS^2\cS ets^{N\mM}$ joining the functors
$$\const_S,\ N\hat{Q}_\bullet(|-|):\SSets^{N\mathbb{M}}\to \cS^2\cS ets^{N\mM}
$$
where $\const_S$ is the constant bisimplicial set functor. Since the diagonal preserves weak
equivalences this sequence induces the first sequence of the following diagram in $\SSets^{N\mathbb{M}}$
$$
\Id \sim\sim\sim \diag N\hat{Q}_\bullet(|-|) \sim\sim\sim \hocolim^{N\mM}N\str \hat{Q}_\bullet(|-|)
$$
while the second sequence is given by \ref{7_7}. 
If $\cM$ has the homotopy colimit property we have an additional weak equivalence
$$\hocolim^{N\mM}N\str \hat{Q}_\bullet(|-|)\to N(\hocolim^{\mM}\hat{Q}_\bullet(|-|)).
$$
We summarize:

\begin{leer}\label{C_2}
 Let $F=\hocolim^{\mM}\hat{Q}_\bullet(|-|):\SSets^{N\mathbb{M}}\to \Cat^\mathbb{M}$. If $\cM$ has the 
 homotopy colimit property there is a sequence
 $$
 \Id \sim\sim\sim N\circ F
 $$
 of natural weak equivalences in $\SSets^{N\mathbb{M}}$.
\end{leer}
By \cite[6.6]{FV} there is a sequence of natural weak equivalences in $\scatm$ 
$$
\const_C \sim\sim\sim \hat{Q}_\bullet(|N(-)|)
$$
where again $\const_C$ is the constant simplicial object functor. We apply the homotopy colimit functor
to this sequence and by \ref{7_8} and \ref{7_9} obtain

\begin{leer}\label{C_3}
 If $\cM$ has the  homotopy colimit property there is a sequence 
 $$\Id \sim\sim\sim F\circ N
 $$
 of natural weak equivalences in $\Cat^\mathbb{M}$. 
\end{leer}

Since levelwise weak equivalences in $\scatm$ and $\cS^2\cS ets^{N\mM}$ are weak equivalences 
we have analogous functors for $\scatm$ and $\cS^2\cS ets^{N\mM}$ by levelwise prolongation.

Proposition \ref{7_3} and Lemma \ref{8_25} are applications of the following result.

\begin{prop}\label{C_4}
Let $\cA$ be a category with a subcategory $\we_{\cA}$ of weak equivalences containing all 
isomorphisms and satisfying
the (2 out of 3)-condition and suppose that the localization $\gamma: \cA\to \cA [\we_{\cA}^{-1}]$ exists.
Let $N:\cB\rightleftarrows \cA:F$ be a pair of functors. Define $\we_{\cB}$ to consist
of all morphisms $f$ in $\cB$ for which $N(f)\in \we_{\cA}$. Let $\widetilde{\cB}\subset \cA [\we_{\cA}^{-1}]$
be the full subcategory of objects $N(B)$ with $B\in \cB$. Then $\gamma\circ N$ factors
$$
\xymatrix{
\cB \ar[r]^N \ar[d]^\delta & \cA \ar[d]^\gamma \\
\widetilde{\cB}\ar[r]^(.4)j & \cA [\we_{\cA}^{-1}]
}
$$
where $j$ is the inclusion. Suppose there are sequences of natural weak equivalences
 $$
 \varepsilon: \Id \sim\sim\sim N\circ F \qquad \textrm{and}\qquad \eta: \Id \sim\sim\sim F\circ N.
 $$
 Then\\
 (1) in the setting of Grothendieck universes, $\delta:\cB\to \widetilde{\cB}$ is the localization
 with respect to $\we_{\cB}$, provided $N$ is injective on objects.\\
 (2) in the G\"odel-Bernay set theory setting, $\delta:\cB\to \widetilde{\cB}$ is a localization
 up to equivalence  with respect to $\we_{\cB}$.\\
In both cases $j:\widetilde{\cB}\to \cA [\we_{\cA}^{-1}]$ is an equivalence of categories.
\end{prop}

\begin{proof}
 By definition of $\widetilde{\cB}$, $j$ is fully faithful, and by assumption, each object $X$ in
 $\cA [\we_{\cA}^{-1}]$ is isomorphic to $NF(X)=jF(X)$. Hence $j$ is an equivalence of categories.
 By \cite[IV.4 Prop. 2]{MacLane} there is a functor $p:\cA [\we_{\cA}^{-1}]\to \widetilde{\cB}$
 such that $p\circ j=\id$ and $j\circ p$ is naturally isomorphic to $\id$.
 
 Let $H:\cB\to \cD$ be a functor mapping weak equivalences to isomorphisms. Then $H\circ F$ maps
  weak equivalences to isomorphisms because $F$ preserves weak equivalences by our assumptions. Hence
  there is a unique functor $\overline{HF}:\cA [\we_{\cA}^{-1}]\to \cD$ such that $\overline{HF}\circ 
  \gamma= H\circ F$. Let $G_1=\overline{HF}\circ j$. Then $\eta$ induces a natural isomorphism
  $$
  G_1\circ \delta = \overline{HF}\circ j\circ \delta = \overline{HF}\circ \gamma \circ N 
  =H\circ F\circ N\cong H.
  $$
 Suppose $G_2: \widetilde{\cB}\to \cD$ is a second functor such that $G_2\circ \delta\cong H$.
 Then for $i=1,\ 2$
 $$
 H\circ F\cong G_i\circ \delta \circ F=G_i\circ p\circ j\circ \delta \circ F
 =G_i\circ p\circ \gamma\circ N\circ F\cong G_i \circ p\circ \gamma.
 $$
 Since $\gamma$ is a localization this implies that $G_1\circ p\cong G_2\circ p$ (see \cite[1.2.2]{Hovey}).
 From this we get
 $$
 G_1=G_1\circ p\circ j \cong G_2\circ p\circ j=G_2.
 $$
 Hence $\delta$ is a localization up to equivalence.
 
 In the setting of Grothendieck universes the localization $\gamma_{\cB}:\cB\to \cB[\we_{\cB}^{-1}]$ exists,
 and there is a unique functor
 $$
 q: \cB[\we_{\cB}^{-1}]\to \widetilde{\cB}
 $$
 such that $q\circ \gamma_{\cB} = \delta$. Since $\gamma_{\cB}$ also is a localization up to equivalence
 $q$ is an equivalence of categories. Since $q$ is bijective on objects it is an isomorphisms.
  
\end{proof}

\end{appendix}

\end{document}